\documentclass[10pt]{article}
\usepackage{amsfonts, epsfig, amsmath, amssymb, xcolor,mathabx,amsthm}
\usepackage{mathrsfs}

\usepackage[english]{babel}

\newcommand{\eps}{\varepsilon}

\newcommand{\E}{\mathbf{E}}
\renewcommand{\P}{\mathbf{P}}

\renewcommand {\epsilon}{\varepsilon}

\newtheorem{thm}{Theorem}[section]
\newtheorem{cor}[thm]{Corollary}
\newtheorem{lem}[thm]{Lemma}

\theoremstyle{definition}

\newtheorem{rem}[thm]{Remark}
\newtheorem{exa}[thm]{Example}

\DeclareMathSymbol{\ophi}{\mathalpha}{letters}{"1E}

\newcommand{\e}{\varepsilon}

\renewcommand{\phi}{\varphi}
\newcommand{\lc}{\leq_C}

\newcommand{\be}{\begin{equation}}
\newcommand{\ee}{\end{equation}}
\newcommand{\ben}{\begin{equation*}}
\newcommand{\een}{\end{equation*}}

\newcommand{\ba}{\begin{equation}\begin{aligned}}
\newcommand{\ea}{\end{aligned}\end{equation}}

\renewcommand{\i}{\mathrm{i}}
\newcommand{\ex}{\mathrm{e}}
\newcommand{\di}{\mathrm{d}}

\newcommand{\cF}{\mathcal{F}}

\newcommand{\cL}{\mathcal{L}}

\newcommand{\bI}{\mathbb{I}}

\newcommand{\bN}{\mathbb{N}}
\newcommand{\bR}{\mathbb{R}}

\newcommand{\rB}{\mathscr{B}}
\newcommand{\rF}{\mathscr{F}}
\newcommand{\Xn}{X^{(n)}}
\newcommand{\XN}{X^{\langle n\rangle}}

\newcommand{\rmD}{\mathrm{D}}
\newcommand{\rmF}{\mathrm{F}}
\newcommand{\rmM}{\mathrm{M}}

\newfont{\cyrfnt}{wncyr10}
\def\J3{\cyrfnt{\rm \u{\cyrfnt I}}}
\def\j3{\cyrfnt{\rm \u{\cyrfnt i}}}

\allowdisplaybreaks[4]

\numberwithin{equation}{section}

\begin{document}
\title{A Tail-Respecting Splitting Numerical Scheme\\ for L\'evy-Driven SDEs With Superlinear Drifts}

\author{Olga Aryasova\footnote{Institute of Mathematics, Friedrich Schiller University Jena, Inselplatz 5,
07743 Jena, Germany,
and Institute of Geophysics, National Academy of Sciences of Ukraine, Palladin Ave.\ 32, Kyiv
03680, Ukraine, and
Igor Sikorsky Kyiv Polytechnic Institute, Beresteiskyi Ave.\ 37, Kyiv 03056, Ukraine;
oaryasova@gmail.com},\
Oleksii Kulyk\footnote{Wroclaw University of Science and Technology, Faculty of Pure and Applied Mathematics, Wybrzeźe Wyspiańskiego
Str.\ 27, 50-370 Wroclaw, Poland; kulik.alex.m@gmail.com} \
and Ilya Pavlyukevich\footnote{Institute of Mathematics, Friedrich Schiller University Jena, Inselplatz 5,
07743 Jena, Germany; ilya.pavlyukevich@uni-jena.de}}

\maketitle

\begin{abstract}
We present an explicit splitting numerical approximation scheme, denoted by $\{X^n\}$, for the effective simulation
of solutions $X$ to a multivariate stochastic differential equation (SDE)
with a superlinearly growing $\kappa$-dissipative drift, where $\kappa>1$, driven by a multiplicative
heavy-tailed L\'evy process that has a finite $p$-th moment, with $p>0$.
We show that  the strong $L^{p_X}$-convergence
$\sup_{t\in[0,T]}\E \|X^n_t-X_t\|^{p_X}=\mathcal O (h_n^{\gamma})$
holds for any $p_X\in (0,p+\kappa-1)$, which is exactly the range where the $p_X$-moment of the solution $X$
is known to be finite.
%(see Theorem 2.12 in \cite{KP21}).
Additionally, for any $p_X\in (0,p)$ we establish strong uniform convergence:
$\E\sup_{t\in[0,T]} \|X^n_t-X_t\|^{p_X}=\mathcal{O} (  h_n^{\delta} )$.
In both cases we determine the convergence rates $\gamma$ and $\delta$.

In the special case of SDEs driven solely by a Brownian motion,
our numerical scheme preserves super-exponential moments of the solution.

The scheme $\{X^n\}$ is realized as a combination of a well-known Euler method with a Lie--Trotter type splitting technique.
\end{abstract}


\smallskip


% \tableofcontents

\section{Main results}

\subsection{Setting and results\label{s:1.1}}

In this paper we focus on strong numerical approximations of solutions of an  It\^o SDE with multiplicative noise
\ba
\label{e:1}
X_t=x & + \int_0^t A(X_s)\,\di s + \int_0^t a(X_{s})\,\di s + \int_0^t b(X_{s})\,\di B_s + \int_0^t c(X_{s-})\,\di Z_s,
\quad x\in\bR^d,\ t\in[0,\infty),
\ea
driven by a $d$-dimensional standard Brownian motion $B$, $d\in\bN$, and an independent $d$-dimensional L\'evy process $Z$
with the characteristic function
\ba
\E \ex^{\i \langle \lambda, Z_t \rangle}=\exp\Big(t \int_{\bR^d}
\big(\ex^{\i \langle\lambda, z\rangle} -1
- \i \langle\lambda, z\rangle\bI(\|z\|\leq 1)\big)\,\nu(\di z)
\Big),\quad \lambda\in\bR^d,\ t\in [0,\infty).
\ea
Here, $\nu$ denotes the L\'evy measure of $Z$, i.e., a measure on $(\bR^d,\rB(\bR^d))$ that satisfies the integrability condition
$\int_{\|z\|>0} (\|z\|^2\wedge 1)\,\nu(\di z)<\infty$.
The dimensions of the processes $X$, $B$ and $Z$ are taken to be equal for notational convenience.

For the remaining notation used throughout the paper, we refer the reader to
Section~\ref{s:nota}. In particular, we will frequently use the convenient
notation $\leq_C$, meaning that
$\text{l.h.s.} \leq_C \text{r.h.s.}$
is equivalent to
$\text{l.h.s.} \leq C\,\text{r.h.s.}$
for some constant $C\in(0,\infty)$ whose precise value is irrelevant for our purposes.

In this paper, we will work under various assumptions from the following list.
Depending on the context, only a subset of these assumptions will be required.

\medskip

\noindent
$\mathbf{H}^\mathrm{diss}_A$: The function $A\colon \bR^d\to\bR^d$ is locally
Lipschitz continuous and is \emph{superlinearly dissipative} in the following sense:
there are $\kappa\in(1,\infty)$ and   $C_{1},C_{2}\in(0,\infty)$ such that
\ba
\langle A(x),x\rangle \leq -C_1\|x\|^{1+\kappa}+C_2,\quad x\in\bR^d.
\ea

\smallskip

\noindent
$\mathbf{H}_{a,b,c}^{\mathrm{Lip}_b}$: The functions $a\colon \bR^d\to\bR^d$,
$b\colon \bR^d\to\bR^{d\times d}$ and $c\colon \bR^d\to\bR^{d\times d}$
are bounded and globally Lipschitz continuous.

\smallskip

\noindent
$\mathbf{H}_{\nu,p}$: There is $p\in(0,\infty)$ such that
\ba
\int_{\|z\|>1}\|z\|^p\nu(\di z)<\infty.
\ea

\noindent
$\mathbf{H}_A^{\text{Lip}_+}$ The function $A\colon \bR^d\to\bR^d$ is continuous
and satisfies the one-sided Lipschitz condition, i.e.,
there is $L\in (0,\infty)$ such that
\ba
\label{e:Lip+}
\langle A(x)-A(y),x-y\rangle \leq L \|x-y\|^2,\quad x,y\in\bR^d.
\ea

\noindent
$\mathbf{H}_{A_x,A_{xx}}$: $A\in C^2(\bR^d,\bR^d)$.
There exists $\chi\in[0,\infty)$ such that
\ba
\label{e:Ax-chi}
\|A_x(x)\| &\lc 1+\|x\|^\chi,
\ea
and  there exists $\eps\in(0,\infty)$ such that
for any $x\in \bR^d$ and any vector
$\phi=(\phi^1,\dots,\phi^d)^T$
\ba
\label{e:Axx}
\eps\|\phi\| \sum_{k=1}^d \|A^k_{xx}(x)\| |\phi^k|   + \langle A_x(x) \phi,\phi\rangle\lc \|\phi\|^2,
\ea
where
$A^k_{xx}$ is the Hesse matrix of the
$k$-th component $A^k$ of the drift $A$.

\begin{rem}
In dimension one, $d=1$, \eqref{e:Lip+} is equivalent to the estimate
\ba
A'(x)\leq L, \quad x\in\bR,
\ea
whereas \eqref{e:Axx} is equivalent to the estimate
\ba
\label{e:A''A'}
|A''(x)|\leq_C 1+ (A'(x))_-,\quad x\in\bR,
\ea
see further discussion and examples in Section \ref{s:sim} below.
\end{rem}

First we note that under $\mathbf{H}^{\mathrm{diss}}_A$, $\mathbf{H}^{\mathrm{Lip}_b}_{a,b,c}$ $\mathbf{H}_{\nu,p}$, equation
\eqref{e:1} admits a unique non-explosive strong solution for each initial value $x\in\bR^d$ (see the argument in Section \ref{s:3}).

SDEs of the form \eqref{e:1} attracted considerable attention in the study of dissipative systems driven by
$\alpha$-stable L\'evy processes, see
\cite{ChechkinGKMT-02,
ChechkinGKM-04,
ChechkinGKM-05,
DybGudHng07,
dybiec2010stationary,
dybiec2023escape}
or non-linear randomly perturbed friction processes
\cite{Lindner2007diffusion, Lindner2008diffusion,Lindner2010diffusion,KP-19}.
In financial mathematics, such SDEs are also used for the simulation
of energy prices, see
\cite{borovkova2006modelling,borovkova2009modeling}.

Since the generator of the Markov process defined by \eqref{e:1} is a non-local operator, it is in general difficult to
determine analytic expressions for characteristics of $X$ like moments,
first exit probabilities etc. Therefore, the numerical approximation of $X$ is of utmost importance in applications.

Note that in SDE \eqref{e:1}, the superlinear drift $A$ and the bounded drift $a$ can be combined into an ``effective drift"
$\widetilde A:= A+a$, which also satisfies the condition $\mathbf{H}^\mathrm{diss}_A$.
Therefore, without loss of generality, we could have assumed that $a\equiv 0$.
However, in certain situations it will be convenient to have the option of splitting the ``effective drift'' $\widetilde A$
into two components by isolating a globally Lipschitz continuous and bounded part, as illustrated in Example \ref{ex:Aa} below.

The most straightforward method for approximating solutions of the SDE \eqref{e:1} is the explicit Euler scheme.
For $T\in[0,\infty)$, let
$[0,T]$ be a fixed time interval. We consider a family of equidistant partitions $\{t^n_k,\ k=0,\ldots,n\}_{n\in\mathbb N}$ with
\ba
t^n_k& =\frac{k}{n}T,\quad k=0,\ldots,n,\\
\ea
and the step size
\ba
h_n &= t^n_{k+1}- t^n_k= \frac{T}{n}.
\ea
Then the explicit Euler scheme for \eqref{e:1} reads
\ba
\label{e:Euler}
X^{\mathrm{E},n}_{0}&=x,\\
X^{\mathrm{E},n}_{t_{k+1}^n}&=X^{\mathrm{E},n}_{t_{k}^n} +\widetilde A( X^{\mathrm{E},n}_{t_{k}^n}) h_n
+ b(X^{\mathrm{E},n}_{t_{k}^n}) ( B_{t_{k+1}^n} - B_{t_{k}^n})
+ c(X^{\mathrm{E},n}_{t_{k}^n}) ( Z_{t_{k+1}^n} - Z_{t_{k}^n}),\quad k=0,\dots,n-1.
\ea
For $c\equiv 0$, i.e., in the absence of jumps, the SDE
\eqref{e:1} with the superlinearly
increasing drift
and its Euler approximations
have been extensively studied in the literature, see, e.g.,
\cite{gyongy1996existence,gyongy1998note,krylov1999kolmogorov}. It is well known that \eqref{e:1}
posesses a unique strong solution $X$ such that for each $r\in(0,\infty)$ and any $T\in[0,\infty)$ the $r$-th maximal
moment of $X$ is finite,
$\E\sup_{t\in[0,T]}\|X_t\|^r<\infty$, see, e.g., Section 2.3 in \cite{Mao2007}.

The Euler scheme $X^{\mathrm{E},n}$
converges in probability
to the true solution uniformly in $t$ on bounded intervals,
namely, for each fixed initial value $x\in\bR^d$, any $T\in[0,\infty)$ and any $\e>0$
\ba
\label{e:ucp}
\lim_{n\to\infty}\P\Big(\max_{0\leq kh\leq T}\|X^{\mathrm{E},n}_{kh} - X_{kh} \|>\e\Big)=0,
\ea
see Theorem 2.4 in \cite{gyongy1996existence}.

However in practice, computer implementations of the explicit Euler scheme \eqref{e:Euler}
$X^{\mathrm{E},n}$ may explode in finite time and yield NaN (``Not a Number'') outputs with positive probability.
Moreover, as demonstrated in Theorem 2.1 in \cite{hutzenthaler2011strong},
the individual absolute moments of Euler approximations
diverge to infinity, i.e., $\lim_{n\to\infty}\E\| X^{\mathrm{E},n}_{T}\|^r= \infty$.

In the continuous case, $c\equiv 0$, the deficiencies of the explicit Euler method can be addressed in two main ways.
First, one can employ split-step backward Euler methods (or, more generally, implicit methods); see
\cite{KloPla-95,schurz2002numerical,higham2002strong}.
These methods require increased computational effort and will not be considered further in this work.
An alternative approach, proposed in \cite{hutzenthaler2012strong}, is the \emph{tamed explicit Euler scheme}:
\ba
\label{e:TEuler}
X^{\mathrm{TE},n}_{0}&=x,\\
X^{\mathrm{TE},n}_{t_{k+1}^n}&=X^{\mathrm{TE},n}_{t_{k}^n}
+\frac{\widetilde A( X^{\mathrm{TE},n}_{t_{k}^n}) h_n}{1+ \|\widetilde A( X^{\mathrm{TE},n}_{t_{k}^n})\| h_n}
+ b(X^{\mathrm{TE},n}_{t_{k}^n}) ( B_{t_{k+1}^n} - B_{t_{k}^n}),\quad k=0,\dots,n-1.
\ea
Since the drift in \eqref{e:TEuler} is bounded, the NaN-problem no longer arises.
Moreover, under an additional assumption that the superlinear part of the drift $A$ satisfies $\mathbf{H}_A^{\text{Lip}_+}$,
one has, for any $r\in(0,\infty)$,
\ba
\E\sup_{t\in[0,T]}\|X^{\mathrm{TE},n}_{kh} - X_{kh} \|^r\leq_C h_n^{r/2},
\ea
see Theorem 1.1 in \cite{hutzenthaler2012strong}.

A comprehensive analysis of various generalizations of the
tamed Euler scheme \eqref{e:TEuler} is provided in \cite{hutzenthaler2015numerical},
while further refinements appear in \cite{sabanis2013note,sabanis2016euler}.
An alternative \emph{truncated Euler--Maruyama method}
with similar convergence properties
was introduced in
\cite{mao2015truncated,mao2016convergence,li2019explicit}.

In the presence of L\'evy-driven component, i.e., when $c(\cdot)\not\equiv 0$ and under assumption
$\mathbf{H}_{\nu,p}$,
the challenges known from the continuous setting become even more pronounced.
First, the presence of (heavy) jumps causes NaN outputs from the explicit Euler scheme to arise more frequently than in the Gaussian case; see Example \ref{ex:main}.
Second, unlike in the Gaussian setting, neither the driving L\'evy process $Z$, nor the true solution $X$
generally possess all absolute moments. The difficulties can be
illustrated by the example of the tamed explicit Euler scheme
that has been adapted to the L\'evy case in \cite{dareiotis2016tamed,kumar2016tamed,kumar2017explicit}.
Suppose,
for definiteness, that $c_*:=\inf_{x\in\bR^d}\|c(x)\|>0$, and let $q\in (p,\infty)$ be such that
$\int_{\|z\|>1}\|z\|^q\nu(\di z)=+\infty$. Since the drift in the tamed Euler scheme
is bounded from above by $1$,
the conditional expectation of the scheme's $q$-th moment diverges:
\ba
\label{e:tamedinf}
\E\Big[\|X^{\mathrm{TE},n}_{t_{k+1}^n}\|^q\Big|\cF_{t_{k}^n}\Big]
&\geq_C
-\|X^{\mathrm{TE},n}_{t_{k}^n}\|^q -1 - \|b\|^q \E\|  B_{t_{k+1}^n} - B_{t_{k}^n}\|^q
+c_*^q \E \|Z_{t_{k+1}^n} - Z_{t_{k}^n}\|^q\\
&\geq_C
-\|X^{\mathrm{TE},n}_{t_{k}^n}\|^q -1 + \E \|Z_{t_{k+1}^n} - Z_{t_{k}^n}\|^q
=+\infty \quad \mathrm{a.s.}
\ea
This explains why the works \cite{dareiotis2016tamed,kumar2016tamed,kumar2017explicit}
consider $p$-integrable L\'evy noises  with
$p\in[2,\infty)$ and establish
convergence in the $L^r$-sense for $r\in (0,p)$ only, see, e.g., Theorem 1 in \cite{kumar2017explicit}. The same restrictions hold for
the truncated Euler--Maruyama or adaptive schemes, see, e.g.,
\cite{deng2019truncated,kieu2022strong,kelly2025strong}.

We emphasize that in the presence of heavy-tailed jump noise,
the tail and moment structure of the solution $X$
becomes an inherent feature of the system. We now state the following result,
which will be proved in Section~\ref{s:3}.

\begin{thm}
\label{t:Xpdiss}
Let assumptions $\mathbf{H}^\mathrm{diss}_A$,
$\mathbf{H}_{a,b,c}^{\mathrm{Lip}_b}$, and $\mathbf{H}_{\nu,p}$ hold. Then, for each $p_X\in(0,p+\kappa-1)$ the
solution $X$ to \eqref{e:1} satisfies
\ba
\label{e:estXpX}
\sup_{t\in[0,\infty)}\E \|X_t\|^{p_X}\leq_C 1+\|x\|^{p_X},\quad x\in\bR^d.
\ea
\end{thm}
This theorem reveals a general phenomenon:
a $\kappa$-dissipative drift transforms the $p$-moment of the noise into (approximately) a $(p + \kappa-1)$-moment of the
solution $X$.
This effect has been studied extensively; for a discussion of the relevant literature,
see Section~1 in \cite{KP21}. In particular, it is known that the upper bound $p + \kappa-1$ is sharp, see \cite[Section~2.6]{KP21}.

The entire line of research presented in this paper is motivated by the following natural question:
\emph{How can one construct an explicit numerical approximation scheme for solutions of \eqref{e:1}
that does not explode and captures (respects) the tail behavior of the true solution $X$,
such that the moments of the numerical approximation converge to the corresponding finite moments of $X$?}

Addressing this question is far from straightforward, as intuitive candidates ---
such as the explicit Euler or the tamed/truncated explicit Euler schemes ---
fail to accurately reproduce the moment behaviour in the range $p_X\in (p, p+\kappa-1)$.

In this paper, we propose a novel explicit approximation scheme, which is \emph{tail-respecting}
in the sense described above.
The scheme combines the Lie--Trotter operator-splitting method for semigroup approximation with an
Euler-type discretization. For background on operator splitting and related product formula approximations,
we refer the reader to the classical works \cite{trotter1959product,strang1968construction,chernoff1974product}
and the recent review \cite{vovchanskyi2024quick}.

More precisely, the solution $X$ is approximated in two steps by decoupling
the deterministic dynamics governed by the superlinearly growing vector field $A$
from the remaining components. The resulting scheme is as follows.

Let $\Phi=\Phi(t,x)$, $t\in [0,\infty)$, $x\in\bR^d$, be the unique solution of the ODE
\ba
\label{e:ODE}
\dot\Phi(t,x)&=A(\Phi(t,x)),\\
\Phi(0,x)&=x,
\ea
see Section \ref{s:Phi} for the properties of the mapping $\Phi$.
Then we set
\ba
\label{e:Xndisc}
X_0^n&:=x,\\
Y^n_{t^n_{k+1}}&:=X^n_{t^n_{k}} + a(X^n_{t^n_{k}})h_n + b(X^n_{t^n_{k}})( B_{t^n_{k+1}} - B_{t^n_{k}})
+ c(X^n_{t^n_{k}})( Z_{t^n_{k+1}} - Z_{t^n_{k}}),\\
X^n_{t^n_{k+1}}&:= \Phi(h_n,  Y^n_{t^n_{k+1}}),\quad k=0,\ldots,n-1.
\ea
As we can see, our scheme  deviates slightly from the Lie--Trotter splitting.
In the latter, the mapping $\Phi$
is composed with the terminal value $Y_{t_{k+1}^n}$ of the solution $Y$ to the SDE
\ba
\label{e:Y}
Y_t&:= X^n_{t^n_{k}} +  \int_{t^n_{k}}^t a(Y_s)\, \di s +  \int_{t^n_{k}}^t b(Y_s )\,\di B_s +  \int_{t^n_{k}}^t c(Y_{s-})\, \di Z_s,
\quad t\in[ t^n_{k},t^n_{k+1}].
\ea
Instead, in \eqref{e:Xndisc} we
employ the one-step Euler approximation of \eqref{e:Y}. Therefore,
the scheme \eqref{e:Xndisc} can be called the explicit
Lie--Trotter--Euler scheme. For brevity, we will refer to \eqref{e:Xndisc} and its modifications as the
splitting scheme.

The practical implementation of the scheme \eqref{e:Xndisc} requires knowledge of the
function $\Phi$. In the most favourable case, the ODE \eqref{e:ODE} can be solved explicitly;
see the examples in Section \ref{s:sim} and Section VI of \cite{pavlyukevich2025simulate}.
Otherwise, $\Phi$ must itself be approximated using a deterministic ODE solver
(see, e.g., \cite{butcher2000numerical,hairer2015numerical} and references therein).
We emphasize, however, that the analysis in the present paper is carried out for the \emph{exact} flow $\Phi$
and relies on the properties established in Section~2. The effect of replacing $\Phi$ by a numerical
approximation is beyond the scope of this work. We return to this issue in Example~\ref{exa:Lorenz-radial}.

As is common in such cases, it will be more convenient to work with
the \emph{continuous time} version of \eqref{e:Xndisc}, which coincides with \eqref{e:Xndisc} at the partition points.  We denote
\ba
\eta_0^n&=0,\\
\eta_t^n&:= t^n_k,\quad t\in  (t^n_k,t^n_{k+1}],\quad k\in\bN_0,
\ea
and set
\ba
\label{e:ns}
\delta^n_t:=t-\eta_t^n,\quad
B^n_t:=B_t - B_{\eta^n_t},\quad
Z^n_t:=Z_t - Z_{\eta^n_t}\quad
\text{and}\quad X^{(n)}_t:=X^n_{\eta^n_t}.
\ea
The continuous time version of \eqref{e:Xndisc} is then given by:
\ba
\label{e:Xncont}
X_0^n&:=x,\\
Y^n_{t}&:= X^{(n)}_t + a(X^{(n)}_t) \delta^n_t + b(X^{(n)}_t)B^n_t
+ c(X^{(n)}_t) Z^n_t,\\
X^n_t&:=\Phi (\delta_t^n,Y^n_t ),\quad t\in[0,T].
\ea

\medskip

Our first main result is the following Theorem that will be proved in Section~\ref{s:3}.
\begin{thm}
\label{t:Xnpdiss}
Let assumptions
$\mathbf{H}^\mathrm{diss}_A$,
$\mathbf{H}_{a,b,c}^{\mathrm{Lip}_b}$,
$\mathbf{H}_{\nu,p},$
$\mathbf{H}_A^{\mathrm{Lip}_+}$, and
$\mathbf{H}_{A_x,A_{xx}}$,
hold.
Then, for any $T\in[0,\infty)$ and any
$p_X\in (0, p+\kappa-1)$
\ba
\label{e:extXn}
\sup_{n\in\bN}\sup_{t\in[0,\infty)} \E \|X^n_t \|^{p_X}\leq_C 1+\|x\|^{p_X},\quad x\in\bR^d.
\ea
\end{thm}

Theorem \ref{t:Xnpdiss} demonstrates that the scheme
\eqref{e:Xncont} preserves the same tail-improving effect as the original SDE \eqref{e:1}.
It is important to emphasize that this property is quite delicate and is not shared
by the tamed Euler scheme or an alternative splitting scheme, which we briefly discuss below.

It is  well known, that the splitting order the in Lie--Trotter approximations can be arbitrary. Therefore,
one may also consider the reverse variant of \eqref{e:Xndisc}:
\ba
\label{e:X2}
\widehat X_0^n&:=x,\\
 \widehat Y^n_{t^n_{k+1}}&:= \Phi(h_n,  \widehat X^n_{t^n_{k}}),\\
\widehat X^n_{t^n_{k+1}}&:= \widehat Y^n_{t^n_{k+1}}
+ a( \widehat Y^n_{t^n_{k+1}})h_n
+ b( \widehat Y^n_{t^n_{k+1}})( B_{t^n_{k+1}} - B_{t^n_{k}})
+ c( \widehat Y^n_{t^n_{k+1}})(Z_{t^n_{k+1}} - Z_{t^n_{k}}),\quad k=0,\ldots,n-1.\\
\ea
The scheme \eqref{e:X2} shares the same structure as the tamed Euler scheme, in that their right-hand sides
contains summands of the form
$c( \widehat Y^n_{t^n_{k+1}})(Z_{t^n_{k+1}} - Z_{t^n_{k}})$.
Then, for any $q\in (p,\infty)$ such that $\int_{\|z\|>1}\|z\|^q\,\nu(\di z)=+\infty$,
the same calculation as in \eqref{e:tamedinf} combined with the estimate \eqref{e:Phi} for the flow $\Phi$
yields
\ba
\E\Big[\|\widehat X^{n}_{t_{k+1}^n}\|^q\Big|\cF_{t_{k}^n}\Big]
&\geq_C
-\|\widehat \Phi(h_n,  \widehat X^n_{t^n_{k}})\|^q -\|a\|^qh_n^q -  \|b\|^q\E \|B_{t^n_{k+1}} - B_{t^n_{k}}\|^q
+ c_*^q\E \|Z_{t_{k+1}^n} - Z_{t_{k}^n}\|^q\\
&\geq _C -\|\widehat X^n_{t^n_{k}}\|^q -1+ \E \|Z_{t_{k+1}^n} - Z_{t_{k}^n}\|^q
=+\infty \quad \mathrm{a.s.}
\ea
Therefore, \eqref{e:X2} does not preserve the finite moments of $X$ beyond the order allowed by the integrability of $Z$.

Moreover, neither the tamed Euler scheme, nor the scheme \eqref{e:X2} possess the necessary structure
to satisfy the analogue of Theorem \ref{t:Xnpdiss}.
The theoretical explanation for this lies in the fact that both Theorems \ref{t:Xpdiss} and \ref{t:Xnpdiss}
are derived from the general result in Theorem~2.12 of \cite{KP21},
which provides sharp bounds on the tails of It\^o semimartingales with dissipative drifts and heavy-tailed jumps.
From this perspective, the scheme \eqref{e:Xndisc}, \eqref{e:Xncont}
is the only one among those discussed that, while formulated in the canonical semimartingale form,
aligns with the framework of Theorem~2.12 in \cite{KP21}.

We analyze the rate of convergence of the scheme \eqref{e:Xndisc}, \eqref{e:Xncont}
in two steps. The first step focuses on the $L^q$-convergence rates for $q\in(0,p)$.
This step does not require $\kappa$-dissipativity and does not involve the tail-improving effect.

\begin{thm}
\label{t:thm_q}
Let assumptions
$\mathbf{H}^{\mathrm{Lip}_+}_A$,
$\mathbf{H}_{A_x,A_{xx}}$,
$\mathbf{H}_{a,b,c}^{\mathrm{Lip}_b}$, and
$\mathbf{H}_{\nu,p}$
 hold. Then,
 for any $T\in[0,\infty)$ and $q\in (0,p)$ there is convergence
\ba
\sup_{t\in[0, T]} \E \|X_t^n-X_t\|^q \leq_C h_n^{\bar \delta}(1+\|x\|^p),\quad x\in\bR^d,
\ea
with the convergence rate
\ba
\bar\delta=\frac{p-q}{\chi}\wedge\frac{q}{2}\wedge 1.
\ea
In particular, if $\chi=0$, i.e., if the drift $A$ has bounded first derivatives, the convergence rate reduces to
\ba
\bar\delta=\frac{q}{2}\wedge 1.
\ea
\end{thm}
The proof of Theorem \ref{t:thm_q} will be given in Section \ref{s:4}.
Within the framework of this Theorem,
the uniform convergence of
$\E \sup_{t\in [0,T]}\|X_t^n-X_t\|^q$ can also be established,
although with a slightly weaker approximation rate. Specifically, we have the following result,
see Section \ref{s:5} for the
proof.
\begin{thm}
\label{t:thm_q_sup}
 Let assumptions
$\mathbf{H}^{\mathrm{Lip}_+}_A$,
$\mathbf{H}_{A_x,A_{xx}}$,
$\mathbf{H}_{a,b,c}^{\mathrm{Lip}_b}$, and
$\mathbf{H}_{\nu,p}$
 hold.
 Then, for any $T\in[0,\infty)$ and $q\in (0,p)$ there is uniform convergence
\ba
 \E \sup_{t\in[0, T]}\|X_t^n-X_t\|^q \leq_C h_n^{\delta}(1+\|x\|^p),\quad x\in\bR^d,
\ea
with the convergence rate
\ba
\delta=\frac{p-q}{\chi}\wedge\frac{q}{4}\wedge \frac{1}{2}.
\ea
\end{thm}

Since Theorem \ref{t:thm_q} and \ref{t:thm_q_sup} focus solely on the moments of the L\'evy process $Z$,
we strongly believe that analogous results should hold for the tamed Euler scheme as well as the reverse splitting scheme
\eqref{e:X2}.

The following theorem, which is the second main result of the paper, is specific to the scheme
\eqref{e:Xndisc}, \eqref{e:Xncont}
and combines the tail-improving effect with convergence rates.

\begin{thm}
\label{t:P0}
Let assumptions $\mathbf{H}^\mathrm{diss}_A$, $\mathbf{H}_{a,b,c}^{\mathrm{Lip}_b}$,
$\mathbf{H}_{\nu,p},$
$\mathbf{H}_A^{\mathrm{Lip}_+}$,
and
$\mathbf{H}_{A_x,A_{xx}}$ hold.
Then, for any $T\in[0,\infty)$, any
$p_X\in [p, p+\kappa-1)$
and any
\ba
\label{e:gamma}
\gamma
<\begin{cases}
\frac{p(p+\kappa-1 - p_X)}{(\chi+2)(\kappa-1)+\chi p},&\quad p\leq \chi+2, \\
\frac{p+\kappa-1 - p_X}{\kappa+\chi-1},  &\quad  p> \chi+2,
\end{cases}
\ea
we have
\ba \sup_{t\in[0,T]}
 \E\|X^n_t-X_t\|^{p_X}
\leq_C  h_n^{\gamma }(1+\|x\|^{p+\kappa-1}),\quad x\in\bR^d.
\ea
\end{thm}

We note that the tail-improving effect from Theorems \ref{t:Xpdiss} and \ref{t:Xnpdiss}
does not apply to
$\sup_{t\in [0,T]}\|X_t\|$ and $\sup_{t\in [0,T]}\|X_t^n\|$ because the large jumps of the noise $Z$,
while not visible at every fixed $t$, cause
\ba
\mathbf{E}\sup_{t\in [0,T]}\|X_t\|^q=\mathbf{E}\sup_{t\in [0,T]}\|X_t^n\|^q=+\infty
\ea
for any $q\in (p,\infty)$ such that $\int_{\|z\|>1}\|z\|^q\nu(\di z)=+\infty$.
Our proof of Theorem \ref{t:P0} relies on an interpolation between the convergence
rate from Theorem \ref{t:thm_q} and the moment bounds from Theorems \ref{t:Xpdiss} and \ref{t:Xnpdiss}.
Therefore we do not expect an analogue of Theorem \ref{t:P0} to hold for
$\sup_{t\in [0,T]}\|X_t\|$, see Example \ref{ex:main}.

The tail-improving effect can also be observed in the continuous case, $c\equiv 0$.
Specifically, the following results will be proved in Sections \ref{s:cont} and \ref{s:cont2}.
Let $C_\mathrm{diss}:=C_1$ be the constant from the condition $\mathbf{H}^\mathrm{diss}_A$, and let
\ba
\label{e:Lambda}
\Lambda=\frac{2C_\mathrm{diss}}{d\|b(\cdot)\|^2}\in(0,\infty).
\ea
\begin{thm}
\label{t:XG}
Let $c\equiv 0$ and let assumptions
$\mathbf{H}^\mathrm{diss}_A$ and
$\mathbf{H}_{a,b,c}^{\mathrm{Lip}_b}$
hold.
Then, for any $T\in[0,\infty)$ and $\lambda\in (0,\Lambda)$
\ba
\label{e:XmomG}
\sup_{t\in[0,T]}\E\Big[\ex^{\frac{\lambda}{1+\kappa}\|X_t\|^{1+\kappa}} \Big]
\leq_C \ex^{\frac{\lambda}{1+\kappa}\|x\|^{1+\kappa}},\quad x\in\bR^d,
\ea
and for any $T\in[0,\infty)$ and $\lambda\in (0,\frac{\Lambda}{2})$
\ba
\label{e:Xmom}
\E\Big[\sup_{t\in[0,T]}\ex^{\frac{\lambda}{1+\kappa}\|X_t\|^{1+\kappa}} \Big]
 \leq_C  \ex^{\frac{\lambda}{1+\kappa}\|x\|^{1+\kappa}},\quad x\in\bR^d.
\ea
\end{thm}
Theorem \ref{t:XG} demonstrates that a $\kappa$-dissipative drift transforms Gaussian
super-exponential tails with index $2$
into solution's super-exponential tails with index $1+\kappa>2$.  The following theorem,
similar to Theorem \ref{t:Xnpdiss}, shows that our splitting scheme preserves this
effect, as well.

\begin{thm}
\label{t:XnG}
Let $c\equiv 0$ and let assumptions
$\mathbf{H}^\mathrm{diss}_A$,
$\mathbf{H}_{a,b,c}^{\mathrm{Lip}_b}$,
$\mathbf{H}_A^{\mathrm{Lip}_+}$,
and
$\mathbf{H}_{A_x,A_{xx}}$ hold.
Then, for any $T\in[0,\infty)$ and $\lambda\in (0,\Lambda)$
\ba
\label{e:momXnG}
\limsup_{n\to\infty}\sup_{t\in[0,T]}\E \Big[\ex^{\frac{\lambda}{1+\kappa}\|X^n_t\|^{\kappa+1}}  \Big]
\leq_C  \ex^{\frac{\lambda}{1+\kappa}\|x\|^{1+\kappa}},\quad x\in\bR^d,
\ea
and for any $T\in[0,\infty)$ and $\lambda\in (0,\frac{\Lambda}{2})$
\ba
\label{e:momXnGsup}
\limsup_{n\to\infty}\E \Big[\sup_{t\in[0,T]}\ex^{\frac{\lambda}{1+\kappa}\|X^n_t\|^{\kappa+1}} \Big]
\leq_C  \ex^{\frac{\lambda}{1+\kappa}\|x\|^{1+\kappa}},\quad x\in\bR^d.
\ea
\end{thm}

The final theorem in this section provides the rate of convergence of our splitting scheme,
accounting for the super-exponential tails with index $1+\kappa>2$.
\begin{thm}
\label{t:XucpG}
Let $c\equiv 0$ and let assumptions
$\mathbf{H}^\mathrm{diss}_A$,
$\mathbf{H}_{a,b,c}^{\mathrm{Lip}_b}$,
$\mathbf{H}_A^{\mathrm{Lip}_+}$,
and
$\mathbf{H}_{A_x,A_{xx}}$ hold.
Then, for any $T\in[0,\infty)$, $r\in(0,\infty)$ and $\lambda\in (0,\Lambda)$
\begin{align}
\label{e:Gauss0}
\sup_{t\in[0,T]}
\E\Big[ \Big(\ex^{\frac{\lambda}{1+\kappa}\|X^n_t\|^{1+\kappa}}
+\ex^{\frac{\lambda}{1+\kappa}\|X_t \|^{1+\kappa}} \Big)\|X^n_t-X_t\|^r\Big]
\leq_C h_n^\frac{r}{2}(1+\|x\|^{r(1+\chi)})\ex^{\frac{\lambda}{1+\kappa}\|x\|^{1+\kappa}},\quad x\in\bR^d,
\end{align}
and for any $T\in[0,\infty)$, $r\in(0,\infty)$ and $\lambda\in (0,\frac{\Lambda}{2})$
\begin{align}
\label{e:Gauss1}
&\E\Big[ \sup_{t\in[0,T]}\Big(\ex^{\frac{\lambda}{1+\kappa}\|X^n_t\|^{1+\kappa}}
+\ex^{\frac{\lambda}{1+\kappa}\|X_t \|^{1+\kappa}} \Big)\|X^n_t-X_t\|^r\Big]
\leq_C h_n^\frac{r}{2}(1+\|x\|^{r(1+\chi)})\ex^{\frac{\lambda}{1+\kappa}\|x\|^{1+\kappa}},\quad x\in\bR^d.
\end{align}
\end{thm}

We conclude this section by noting that, although there exists a substantial literature
on splitting methods for SDEs and SPDEs, a comprehensive review of the subject is still lacking.

In the Gaussian setting, extensive bibliographic information can be found in
\cite[Section~12]{vovchanskyi2024quick} and \cite[Section~1]{buckwar2022splitting}.
Splitting methods for SDEs with cubic nonlinearity and additive noise were studied in \cite{buckwar2022splitting};
see also the discussion in Example~\ref{ex:main} below. More recently, \cite{pilipovic2024parameter}
investigated various splitting schemes for SDEs with $A$ satisfying the one-sided Lipschitz condition
$\mathbf{H}_A^{\mathrm{Lip}_+}$, linear drift coefficient $a(\cdot)$, and constant diffusion coefficient
$b(\cdot)$, with particular emphasis on statistical applications. In \cite{hausenblas2020theoretical},
splitting techniques were employed in the numerical analysis of pattern formation for nonlinear Gray--Scott equations.
However, none of the above works addresses the preservation of super-exponential moment bounds of the
type established in Theorem~\ref{t:XucpG}.

To the best of our knowledge, the present paper is the first work devoted specifically to splitting methods
for solutions of L\'evy-driven SDEs.

\subsection{Examples, simulations, and discussion\label{s:sim}}

\begin{exa}\label{ex:main}
To illustrate our findings, we consider an exemplary one-dimensional SDE
\ba
\label{e:XZ}
\di X_t=-X_t^3\,\di t + \di Z_t,\quad X_0=0,
\ea
with $A(x)=-x^3$, $a=b\equiv 0$, $c\equiv 1$ and
$Z$ being the Cauchy process with the L\'evy measure
$\nu(\di z)=\pi^{-1}|z|^{-2}\bI(|z|>0)\,\di z$
and the characteristic function $\E \ex^{\i \lambda Z_t} =\ex^{-t|\lambda|}$, $\lambda\in \bR$.
The Assumption
$\mathbf{H}_{\nu,p}$,
is satisfied for any $p\in(0,1)$ and  $\mathbf{H}_{a,b,c}^{\mathrm{Lip}_b}$ holds trivially. The coefficient $A(x)=-x^3$
satisfies the assumptions $\mathbf{H}^\mathrm{diss}_A$, $\mathbf{H}_A^{\mathrm{Lip}_+}$,
and
$\mathbf{H}_{A_x,A_{xx}}$ with $\chi=2$ and $\kappa=3$.
By Theorem \ref{t:Xpdiss}, the SDE \eqref{e:XZ} has a unique strong global solution
$X$ which satisfies $\sup_{t\in[0,\infty)}\E|X_t|^{p_X}<\infty$ for any $p_X\in(0,3)$.

Table \ref{tab:sim1} presents the simulation results for the explicit Euler scheme \eqref{e:Euler}
and the tamed Euler scheme \eqref{e:TEuler} with random increments
$Z_{t_{k+1}^n} - Z_{t_{k}^n}$.
The simulations confirm the expected difficulties.
The second column of the Table \ref{tab:sim1}
contains the number of NaN outputs (i.e., blow-ups of the scheme) in a Monte Carlo simulation
with $N=10^7$ runs
for the Euler scheme $X^{\mathrm{E},n}_T$ with $T=5$ and different time steps $h_n$.
Note that NaN outputs merely reflect numerical instabilities of the Euler scheme and are unrelated to the existence or non-existence of moments of the solution $X$.

As explained in Section \ref{s:1.1}, the drift term in the tamed Euler scheme is bounded, the therefore its sample paths
remain finite.
In columns 3--7 of Table \ref{tab:sim1}, we present the empirical absolute moments of the tamed Euler scheme $X^{\mathrm{TE},n}_T$ which are
calculated as follows. Let
$\{X^{\mathrm{TE},n}_T (i)\}_{1\leq i\leq N}$ be the empirical results obtained in
$N$ independent series of simulations.
Since these values are always finite,
the empirical absolute $p_X$-moments of $X^{\mathrm{TE},n}_T$, defined as
\ba
\langle  |X^{\mathrm{TE},n}_T|^{p_X} \rangle :=\frac{1}{N} \sum_{i=1}^N |X^{\mathrm{TE},n}_T (i)|^{p_X},
\ea
are finite for any $p_X\in (0,\infty)$.

\begin{table}[h]
\begin{center}
\begin{tabular}{|l||r||c|c|c|c|c|}
\hline
$h_n$ & $\sharp$ NaN for $X^{\mathrm{E},n}_5$
& $ \langle |X^{\mathrm{TE},n}_5|^{1/2} \rangle$
& $ \langle |X^{\mathrm{TE},n}_5| \rangle$
& $ \langle |X^{\mathrm{TE},n}_5|^{3/2} \rangle $
& $ \langle |X^{\mathrm{TE},n}_5|^2 \rangle$
& $ \langle |X^{\mathrm{TE},n}_5|^{5/2} \rangle $
\\\hline
$10^{-2}$ & 1994317  &  1.755 & 233.3& $2.2\cdot 10^{6}$ & $4.5 \cdot 10^{10}$& $9.8 \cdot 10^{14}$ \\\hline
$10^{-3}$ & 686855   &   1.156& 287.2  & $1.7\cdot 10^{6}$  & $1.7 \cdot 10^{10}$ & $1.9 \cdot 10^{14}$\\\hline
$10^{-4}$ &  223077  &   0.905  &65.0  & $9.6\cdot 10^{4}$  & $2.0\cdot 10^{8}$ & $5.0 \cdot 10^{11}$ \\\hline
$10^{-5}$ &  71020   &  0.839   &34.1  & $1.1\cdot 10^{5}$ & $4.7\cdot 10^{8}$  & $2.2\cdot 10^{12} $ \\\hline
\end{tabular}
\end{center}
\caption{The number of NaNs out of $N=10^7$ runs for the explicit Euler scheme
and the empirical absolute moments of the tamed Euler scheme for the SDE \eqref{e:XZ}.\label{tab:sim1}}
\end{table}
The results in Table \ref{tab:sim1} indicate that the explicit Euler scheme fails to converge due to explosions,
while the tamed Euler scheme also does not converge for $p_X\in\{1,1.5,2,2.5\}$.
However, the observed convergence of the moments of order $p_X=0.5<1$
is consistent with our conjecture
that the results of
\cite{dareiotis2016tamed,kumar2016tamed,kumar2017explicit} remain valid for all
$p\in (0,\infty)$, and suggests that Theorem \ref{t:thm_q} should apply to the tamed Euler method.

In the practical implementation of our splitting scheme \eqref{e:Xndisc} we use the explicit formula
for the solution of the ODE $\dot \Phi=-\Phi^3$, which leads to the result that
\ba
\label{e:Phi3}
\Phi(t,x)=\frac{x}{\sqrt{2tx^2 +1}},\quad t\in[0,\infty),\ x\in\bR.
\ea
Hence the numerical scheme \eqref{e:Xndisc} takes the following explicit form:
\ba
\label{e:Xexa}
X^n_0&=0,\\
Y^n_{t_{k+1}^n}&=X^n_{t_{k}^n} + h_n Z_{k+1},\\
X^n_{t_{k+1}^n}&=\frac{Y^n_{t_{k}^n}}{\sqrt{2h_n (Y^n_{t_{k}^{n+1}})^2 +1}},\quad k=0,\dots, n-1,
\ea
where $\{Z_k\}_{k\in\bN}$ are iid Cauchy distributed random variables with the p.d.f.\ $f(x)=\frac{1}{\pi(x^2 + 1)}$.

In order to estimate the accuracy of our method we compare the simulated empirical moments with the theoretically known moments
for the limit (stationary) distribution of $X$.
From general theory, it is known that the process $X$ is exponentially ergodic (see, e.g., \cite{Kulik17}),
so the stationary moments can serve as a proxy for the moments $\E |X_T|^{p_X}$ evaluated at a sufficiently large time $T=5$.

The process $X$ is known to have the stationary law $X_\infty$ with the density
\ba
\label{e:m}
m(x)=\frac{1}{\pi(x^4 - x^2 +1)}, \quad x\in\bR,
\ea
as given in Eq.\ (3.5) in \cite{ChechkinGKMT-02}; also see \cite{DubSpa07} for a general approach.
Hence, the absolute moments of $X_\infty$ can be evaluated explicitly, as shown at the bottom of Table \ref{tab:2}.

\begin{table}[h]
%\label{tab:sim2}
\begin{center}
\begin{tabular}{|l|c|c|c|c|c|}
\hline
$h_n$
& $ \langle |X^n_5|^{1/2} \rangle$
& $ \langle |X^n_5|       \rangle$
& $ \langle |X^n_5|^{3/2} \rangle$
& $ \langle |X^n_5|^2     \rangle$
& $ \langle |X^n_5|^{5/2} \rangle$
\\\hline
$10^{-2}$ & 0.814 & 0.760 & 0.787 & 0.899 & 1.137    \\\hline
$10^{-3}$ & 0.816 & 0.768 & 0.811 & 0.967 & 1.354 \\\hline
$10^{-4}$ & 0.816 & 0.769 & 0.815 & 0.989 & 1.479 \\\hline
$10^{-5}$ & 0.817 & 0.770 & 0.817 & 0.998 & 1.539 \\\hline\hline
          & $\sqrt{2/3}= 0.816$& $4\sqrt 3/9=0.770$& $\sqrt{2/3}= 0.816$ & 1& $2\sqrt{2/3}=  1.633$\\ \hline
& $\E |X_\infty|^{1/2} $
& $ \E |X_\infty| $
& $ \E |X_\infty|^{3/2} $
& $ \E|X_\infty|^2 $
& $ \E |X_\infty|^{5/2} $\\
\hline
\end{tabular}
\end{center}
\caption{The empirical absolute moments of the splitting scheme \eqref{e:Xexa}, $N=10^7$, vs.\ the emprical
absolute moments of the stationary law \eqref{e:m}.\label{tab:2}}
\end{table}

In Table~\ref{tab:rates}, we present the theoretical convergence rates predicted by
Theorems~\ref{t:thm_q} and \ref{t:P0}, together with the rates estimated from the empirical data.
As can be seen, the empirical convergence rates are significantly higher than the theoretical ones.
This can be attributed to the particularly simple one-dimensional structure of the model,
the additive nature of the noise, and the symmetry of the underlying dynamics.

\begin{table}[h]
%\label{tab:sim2}
\begin{center}
\begin{tabular}{|l|c|c|c|c|c|}
\hline
& $ p_X=1/2$
& $ p_X=1  $
& $ p_X=3/2$
& $ p_X=2  $
& $ p_X={5/2}$
\\\hline
\text{Theoretical convergence rate}& $h_h^{0.25-}$ & $h_n^{0.2-}$ & $h_n^{0.15-}$ & $h_n^{0.1-}$ & $h_n^{0.05-}$  \\\hline
\text{Estimated convergence rate}  & $0.014 h_h^{0.46}$ & $0.08 h_n^{0.54}$ & $0.24h_n^{0.51}$ & $0.49h_n^{0.38}$ & $1.14h_n^{0.20}$  \\\hline
\hline
\end{tabular}
\end{center}
\caption{The theoretical and estimated convergence rates of the splitting scheme \eqref{e:Xexa}, $N=10^7$.\label{tab:rates}}
\end{table}
The computations were carried out on a virtual Linux machine running Ubuntu 24.04.4 LTS equipped with 16
virtual CPU cores based on Intel Xeon (Skylake) processors and 125 GB of RAM.
The Monte Carlo simulations were implemented in C++ and parallelized using OpenMP.
The run time increased almost linearly ranging from 15 seconds for $h_n=10^{-2}$ to about 4 hours for $h_n=10^{-5}$.

As observed, the Monte Carlo simulations recover the absolute moments of $X_\infty$ quite well up to $p_X=2$,
but the moment of order
$p_X=2.5$ appears to be systematically underestimated.
This phenomenon can be explained as follows.
Note that the $\kappa$-dissipativity of the drift $A$
implies that the deterministic solution $\Phi$ returns from infinity
to some neighbourhood of the origin in finite time.  In our example, formula \eqref{e:Phi3} shows that for any $x\in\bR$
\ba
\label{e:Kn}
|\Phi(h_n,x)|\leq K_n:=\frac{1}{\sqrt{2h_n}}.
\ea
This means that the values of the splitting approximation scheme $X^n$ evaluated on the discrete time grid $\{t_n^k\}$
are \emph{bounded} with probability 1,
\ba
\max_{k=1,\dots,n} |X^n_{t^n_k}| \leq \frac{1}{\sqrt{2h_n}}.
\ea
Therefore, \textit{practically} the numerical scheme $\{X^n\}$ approximates the \emph{truncated moments}
\ba
\E \Big[|X_{T}|^{p_X}\bI(|X_{T}|\leq K_n)\Big].
\ea
\begin{table}
\begin{center}
\begin{tabular}{|l|r|r|r|r|r|r|}
\hline
$h_n$ & $K_n$
&  $\E_{\leq K_n}|X_\infty|^{1/2}$
&  $\E_{\leq K_n} |X_\infty| $
&  $\E_{\leq K_n} |X_\infty|^{3/2} $
&  $\E_{\leq K_n} |X_\infty|^{2} $
&  $\E_{\leq K_n} |X_\infty|^{5/2} $ \\\hline
$10^{-2}$   &   7 &  0.815& 0.763& 0.794& 0.909& 1.152 \\\hline
$10^{-3}$   &  22 &  0.816& 0.769& 0.812& 0.972& 1.364 \\\hline
$10^{-4}$  &  71  &  0.816& 0.770& 0.816& 0.991& 1.482 \\\hline
$10^{-5}$  &  224 &  0.816& 0.770& 0.816& 0.997& 1.548 \\\hline\hline
    0    &  $+\infty$    &  0.816& 0.770& 0.816&   1  & 1.633\\ \hline
\end{tabular}
\end{center}
\caption{The values of the truncated absolute moments of the stationary law \eqref{e:m},
$\E_{\leq K_n} |X_\infty|^{p_X}:=\E[|X_\infty|^{p_X}\bI(|X_\infty|\leq K_n)]$,
$K_n$ is given in \eqref{e:Kn}.\label{tab:3}}
\end{table}
Clearly, as $n\to\infty$ we have convergence to the true value
\ba
\lim_{n\to\infty}\E \Big[|X_{T}^n|^{p_X}\bI(|X^n_{T}|\leq K_n)\Big]=\E |X_{T}|^{p_X},
\ea
but the  term $\E  [|X_{T}|^{p_X}\bI(|X_{T}|> K_n) ]$
contributes the systematic error of the order
$
\mathcal{O}(h_n^{\frac{p+\kappa-1-p_X}{\kappa-1}})= \mathcal{O}(h_n^{\frac{3-p_X}{2}}),
$
which can be observed in simulations, see Table \ref{tab:3}.

Finally, in Table \ref{tab:sim4}, we present the simulation of the
empirical moment with the help of
the
reverse splitting scheme \eqref{e:X2}:
\ba
\label{e:Xrev}
\widehat X^n_0&=0,\\
\widehat Y^n_{t_{k+1}^n}&=\frac{ \widehat X^n_{t_{k}^n}}{\sqrt{2h_n (\widehat X^n_{t_{k}^n} )^2 +1}},\\
\widehat X^n_{t_{k+1}^n}&=\widehat Y^n_{t_{k+1}^n} +  h_n Z_{k+1}, \quad k=0,\dots, n-1.
\ea
The empirical absolute moments at time $T=5$ appear to be close to the true values for small enough $h_n$,
which seemingly  contradicts the fact that
$\E |\widehat X^n_T|^{p_X}=+\infty$ for $p_X\in[1,\infty)$.
This effect can be explained as follows.

It is informative to consider both the splitting and the reverse splitting schemes as by-products of one
(hidden) two-step scheme, where the ``deterministic'' steps governed by the flow $\Phi$ are
interchanged with the ``stochastic'' Euler-type approximation steps. The difference between the splitting
and the reverse splitting schemes then is determined  by the choice of steps after which the values of the scheme are taken out:
deterministic  for the splitting scheme, and stochastic for the reverse splitting one. Since each deterministic step keeps
corresponding values of the two-step scheme under the threshold $K_n$, the terminal value of the reverse splitting scheme satisfies
$\E |\widehat X^n_T|^{p_X}=+\infty$ for $p_X\in[1,\infty)$ only because the corresponding moments of the increment of $Z$
on the very last stochastic step are infinite. Apparently this theoretical effect can be hidden in the practical simulations, because, as
$h_n\to 0$, the probability for the very last increment of $Z_T-Z_{T-h_n}$ to be large becomes too small, so that this event is not
observed for $h_n=10^{-5}$ under the given number of the Monte--Carlo simulations
$N=10^7$ (the last row in the Table 4). For smaller values
$h_n=10^{-2}, 10^{-3}$ the number $N=10^7$ is clearly large enough to make this event observable,
and we expect that for any $h_n$ and the
number of the Monte--Carlo simulations large enough the empirical moments of the order $p_X\in[1,\infty)$ will exhibit a similar blow-up.
We also think that the over-estimation of the moment of the order
$5/2$ for $h=10^{-4}, N=10^7$ of appears due to the same ``large last jump''
effect.

\begin{table}[h]
\begin{center}
\begin{tabular}{|l|r|r|r|r|r|}
\hline
$h_n$
& $ \langle |\widehat X^n_5|^{1/2}  \rangle$
& $ \langle |\widehat X^n_5|        \rangle$
& $ \langle |\widehat X^n_5|^{3/2}  \rangle$
& $ \langle |\widehat X^n_5|^2      \rangle$
& $ \langle |\widehat X^n_5|^{5/2}  \rangle$
\\\hline
$10^{-2}$ &  0.821&  0.837&    8.5&  $2.3\cdot 10^2$&  $8.4\cdot 10^5$  \\\hline
$10^{-3}$ &  0.817&  0.797&   10.8&  $4.5\cdot 10^3$&  $2.0\cdot 10^6$   \\\hline
$10^{-4}$ &  0.817&  0.770&  0.819&  1.058&  2.954     \\\hline
$10^{-5}$ &  0.817&  0.770&  0.817&  1.002&  1.586     \\\hline
\end{tabular}
\end{center}
\caption{The empirical absolute moments of the reverse splitting scheme \eqref{e:Xrev}, $N=10^7$.
\label{tab:sim4}}
\end{table}

The difference between the splitting and reverse splitting schemes becomes particularly
apparent when the convergence of empirical moments is examined uniformly in time.
In accordance with Theorem~\ref{t:P0}, the splitting scheme $X^n$ preserves the uniform convergence of
individual moments, albeit with a slower convergence rate for moments $p_X$ approaching the critical value
$p+\kappa-1=3$; see Table~\ref{tab:sim6}. In contrast, the results in Table~\ref{tab:sim5} clearly
demonstrate that the reverse splitting scheme \eqref{e:Xrev} fails to preserve moments uniformly over time.

This phenomenon admits a simple explanation. Although the number of Monte--Carlo simulations, $N=10^7$,
is typically insufficient for an exceptionally large \emph{final} increment of $Z$ to occur when
$h_n=10^{-5}$, it is sufficiently large for at least one unusually large increment to appear somewhere
along a simulated trajectory. Due to the structure of the reverse splitting scheme, such an increment
may generate an extremely large value of the approximation, which in turn causes the corresponding uniform empirical moment to explode.

\begin{table}[h]
\begin{center}
\begin{tabular}{|l|r|r|r|r|r|}
\hline
$h_n$
& $ \max\limits_{5\leq t_k^n\leq 6} \langle | X^n_{t_k^n}|^{1/2}\rangle $
& $ \max\limits_{5\leq t_k^n\leq 6}\langle |  X^n_{t_k^n}|   \rangle $
& $ \max\limits_{5\leq t_k^n\leq 6}\langle|  X^n_{t_k^n}|^{3/2} \rangle $
& $ \max\limits_{5\leq t_k^n\leq 6}\langle |  X^n_{t_k^n}|^2 \rangle$
& $ \max\limits_{5\leq t_k^n\leq 6}\langle|  X^n_{t_k^n}|^{5/2} \rangle $
\\\hline
$10^{-2}$ &  0.814& 0.760 & 0.788 & 0.899 & 1.139     \\\hline
$10^{-3}$ &  0.816& 0.769 & 0.812 & 0.970 & 1.367     \\\hline
$10^{-4}$ &  0.817& 0.770 & 0.818 & 1.000 & 1.530     \\\hline
$10^{-5}$ &  0.817& 0.771 & 0.819 & 1.026 & 1.846 \\\hline
\end{tabular}
\end{center}
\caption{Uniform estimates of the empirical absolute moments of the splitting scheme \eqref{e:Xexa}, $N=10^7$.
\label{tab:sim6}}
\end{table}

\begin{table}[h]
\begin{center}
\begin{tabular}{|l|r|r|r|r|r|}
\hline
$h_n$
& $ \max\limits_{5\leq t_k^n\leq 6} \langle |\widehat X^n_{t_k^n}|^{1/2}\rangle $
& $ \max\limits_{5\leq t_k^n\leq 6} \langle |\widehat X^n_{t_k^n}|   \rangle $
& $ \max\limits_{5\leq t_k^n\leq 6} \langle |\widehat X^n_{t_k^n}|^{3/2} \rangle $
& $ \max\limits_{5\leq t_k^n\leq 6} \langle |\widehat X^n_{t_k^n}|^2 \rangle$
& $ \max\limits_{5\leq t_k^n\leq 6} \langle |\widehat X^n_{t_k^n}|^{5/2} \rangle $
\\\hline
$10^{-2}$ & 0.822& 2.1 & $3.7\cdot 10^3$&  $2.2\cdot 10^6$ & $4.9\cdot 10^9$  \\\hline
$10^{-3}$ & 0.817& 3.0 & $1.1\cdot 10^4$&  $5.0\cdot 10^7$ & $2.4\cdot 10^{11}$  \\\hline
$10^{-4}$ & 0.818& 31.0& $5.3\cdot 10^5$&  $9.2\cdot 10^8$ & $9.1\cdot 10^8$ \\\hline
$10^{-5}$ & 0.817& 1.13& $4.4\cdot 10^2$&  $1.3\cdot 10^6$ & $2.4\cdot 10^9$  \\\hline
\end{tabular}
\end{center}
\caption{Uniform estimates of the empirical absolute moments of the reverse splitting scheme \eqref{e:Xrev}, $N=10^7$.
\label{tab:sim5}}
\end{table}

Finally, we recall the work \cite{buckwar2022splitting}, where the authors
applied a splitting approach to study the scalar Gaussian SDE
$\di X_t=-X_t^3\,\di t+\di W_t$.
A distinctive feature of their method is that they first rewrite the equation as
$\di X_t = (-X_t^3+X_t)\,\di t - X_t\,\di t + \di W_t$,
and then split the dynamics into an exactly solvable exponentially ergodic Ornstein--Uhlenbeck component
$\di Y_t = -Y_t\,\di t + \di W_t$,
(which, in our framework, would correspond to the \emph{exact} solution of \eqref{e:Y})
and an exactly solvable nonlinear deterministic component given by the ODE
$\dot{\Phi}=-\Phi^3+\Phi$.
These components are then composed in \emph{reverse Lie--Trotter order},
first applying the deterministic flow $\Phi$ and then the stochastic flow $Y$.

Theorem~2 in \cite{buckwar2022splitting} establishes pointwise mean-square
convergence of order $1$, whereas our results (see, e.g., Lemma~\ref{t:thmGr})
yield uniform convergence of order $1/2$ for the splitting scheme \eqref{e:Xndisc}. The lower convergence rate in our
setting is primarily due to the stronger requirement of uniform-in-time error bounds and the presence of multiplicative noise.

More generally, we expect that in the Gaussian setting most reasonable splitting schemes achieve the
classical convergence rates $1$ or $1/2$, depending on the notion of convergence under consideration.
The numerical experiments reported in \cite{buckwar2022splitting} indicate that the Strang splitting performs
slightly better than the reverse Lie--Trotter scheme considered therein. This suggests that a Strang-type
modification of our method may lead to improved convergence rates in the Gaussian case, too.
Whether a similar improvement can be achieved in the presence of heavy-tailed L\'evy noise remains an open problem.
\color{black}

Altogether, these effects illustrate the peculiar and often deceptive behaviour of heavy-tail SDEs and their approximations.

More ready-to-use illustrative examples and effects can be found in
\cite{pavlyukevich2025simulate}.
\end{exa}

The next example demonstrates why it is useful to decompose the ``effective drift'' into the components
$A(\cdot)$ and $a(\cdot)$.
\begin{exa}
\label{ex:Aa}
Consider a one dimensional SDE
  \ba
\label{e:XZ1}
\di X_t=-X_t^3\,\di t +\sin X_t\, \di t+ \di Z_t,\quad X_0=x\in\bR.
\ea
with $Z$ being a L\'evy process that satisfies assumption $\mathbf{H}_{\nu,p}$.
By taking $A(x)=-x^3+\sin x$, $a=b\equiv 0$, $c\equiv 1,$ we arrive at the nonlinear ODE $\dot \Phi=-\Phi^3+\sin \Phi$.
This ODE cannot be solved explicitly, and one needs
to use additional stable numerical approximation methods for deterministic ODEs here, see, e.g., \cite{humphries1994runge}.
On the other hand, the same SDE can be treated as \eqref{e:1} with $A(x)=-x^3$, $a(x)=\sin x$. This results in the following
simple explicit numerical scheme:
\ba
X^n_0&=x,\\
Y^n_{t_{k+1}^n}&=X^n_{t_{k}^n} + h_n\sin(X^n_{t_{k}^n})+Z_{t_{k+1}^n}-Z_{t_{k}^n},\\
X^n_{t_{k+1}^n}&=\frac{ Y^n_{t_{k+1}^n}}{\sqrt{2h_n (Y^n_{t_{k+1}^n} )^2 +1}},\quad k=0,\dots,n-1.
\ea
\end{exa}

We conclude this section with several remarks and examples regarding the assumptions imposed on the function $A$. First, we observe that
the dissipativity condition $\mathbf{H}^\mathrm{diss}_A$
together with the growth condition \eqref{e:Ax-chi} implies that
\ba
\label{e:kappachi}
\kappa\leq \chi+1.
\ea
A typical example of this is a polynomial drift.
\begin{exa}
Let $A\colon \bR\to\bR$ be a polynomial of an odd order,
\ba
\label{e:poly}
A(x)&= -x^{2n+1}+ c_{2n}x^{2n}+ \cdots + c_1 x+ c_0,\\
A'(x)&= -(2n+1)x^{2n}+ o(x^{2n}),\\
A''(x)&= -2n(2n+1)x^{2n-1}+ o(x^{2n-1}),\\
\ea
so that $\kappa=2n+1$,  $\chi=2n$, and the
inequality \eqref{e:kappachi} turns into equality.
Conditions $\mathbf{H}_A^{\mathrm{Lip}_+}$ and \eqref{e:A''A'} are satisfied, too. Actually, in this case
\ba
A''(x)=o((A'(x))_-), \quad |x|\to \infty.
\ea
% \hfill $\Box$
\end{exa}

We note that there are examples of drifts $A$ satisfying all the assumptions of Theorem \ref{t:Xnpdiss},
for which the inequality
\eqref{e:kappachi} is strict. However, we do not provide these examples here,
as they do not arise naturally in practice. Nevertheless, this observation helps to explain why we retain both parameters
$\kappa$ and $\chi$ in the formulation of our main results

In the following example, $A$ is allowed to have highly oscillatory components,
which, unlike in the polynomial case, make $A''(x)$ comparable with $(A'(x))_-$.
\begin{exa}
Let $A(x)=-x^3+\sin x^2$, $x\in\bR$. Then
\ba
  A'(x)=-3x^2+2x\cos x^2, \quad A''(x)=-6x-4x^2\sin x^2+2\cos x^2,
\ea
  so that
\ba
  \limsup_{|x|\to \infty}\frac{A''(x)}{(A'(x))_-}=\frac{4}{3}>0.
\ea
  Both assumptions $\mathbf{H}_A^{\mathrm{Lip}_+}$ and $\mathbf{H}_{A_x,A_{xx}}$ are satisfied in this case with $\chi=2$,
  and $\mathbf{H}^\mathrm{diss}_A$ is satisfied with $\kappa=3$. Since $x\mapsto \sin x^2$ is not globally Lipschitz,
  we cannot identify it with the function $a(\cdot)$.
\end{exa}

Next, we give several multivariate examples. It is known that the one-sided Lipschitz condition \eqref{e:Lip+}
for $A$ is equivalent to
the boundnedness of the symmetrized (distributional) gradient matrix $A_x$ from above, i.e.,
\ba\label{e:grad_sym}
[A_x(x)]_{\mathrm{sym}}:=\frac{A_x(x) +  A_x(x)^T}{2}\leq L\cdot \text{Id},\quad
\ea
which means that for any $\phi\in\bR^d$ and $x\in\bR^d$
\ba
\langle [A_x(x)]_{\mathrm{sym}} \phi,\phi  \rangle\leq L\|\phi\|^2,
\ea
see Lemma 2.2 in \cite{bouchut2005uniqueness}. In dimension $d=1$, this simply means that $A'(x)\leq L$, $x\in\bR$.

\begin{exa}
\label{ex:multidim1}
Let $f\in C^2(\bR_+,\bR)$ and let
\ba
\label{e:Af}
  A(x)=- f(\|x\|^2)x,\quad  x\in \mathbb R^d.
\ea
Then
\ba
  A_x(x)= -f(\|x\|^2)\text{Id}-2 f'(\|x\|^2)xx^T),
\ea
  and
\ba
  A^k_{xx}(x)=-2x_kf'(\|x\|^2)\text{Id}-4f'(\|x\|^2)(x\otimes \mathrm{e}_k)-4x_kf''(\|x\|^2)xx^T,
\ea
where $\mathrm{e}_k$ stands for the $k$-th basis vector in $\mathbb{R}^d$. Then $\mathbf{H}_A^{\mathrm{Lip}_+}$ holds true provided that
\ba
-f(r)-2rf'(r)\leq L, \quad r\in[0,\infty),
\ea
and \eqref{e:A''A'} is satisfied if
\ba
\gamma\Big(r^{3/2}|f''(r)|+r^{1/2}|f'(r)|\Big)-f(r)-2rf'(r)\leq C, \quad r\in[0,\infty),
\ea
for some $\gamma,C\in(0,\infty)$.
These assumptions are satisfied, for example, when $f(\cdot)$ is a polynomial, such as $f(r)=r^n+\dots$
In this case all the assumptions of Theorem \ref{t:P0} hold with
$\kappa=2n+1$, $\chi=2n$.
One can also consider more complex examples that include fast oscillations, e.g.,
\ba
A(x)=-(\|x\|^4+\sin \|x\|^2) x,
\ea
where $f(r)=r^2+\sin r$ satisfies the above assumptions, and all assumptions of Theorem \ref{t:P0} hold with
$\kappa=5$, $\chi=4$.

The drift $A$ given in \eqref{e:Af} is quite similar to those we encountered in the one-dimensional setting, as
we can represent $A$ as a gradient of a rotationally invariant potential $U$, i.e.,
$A(x)=-\nabla U(x)$, where
\ba
U(x)=\frac12\int_0^{\|x\|^2} f(r)\, \di r.
\ea
% \hfill $\Box$
\end{exa}
It appears, that for drifts without rotational symmetry, the one-sided
Lipschitz conditions becomes a strong limitation.
To illustrate this, we provide a negative example of such a case.

\begin{exa}
Let
\ba
\label{e:Lorenz}
A(x,y,z)=-\e
\begin{pmatrix}
             x^{2n+1}\\
             y^{2n+1}\\
             z^{2n+1} \\
           \end{pmatrix}
+
\begin{pmatrix}
             -ax-y^2-z^2\\
             -y+xy-bxz \\
             -z+bxy+xz \\
           \end{pmatrix}
            =:A_1(x,y,z)+A_2(x,y,z)
, \quad (x,y,z)^T\in \mathbb{R}^3,
\ea
with $n\in \mathbb{N}_0$, $\e\in[0,\infty)$, and $a,b>0$. In the case $\e=0$, $a=0.25$ and $b=4$, this is the Lorenz-84 atmospheric
circulation model with the thermal forcing terms $F=G=0$, see, e.g., Section 3 in \cite{arnold2003reduction}.
We have
\ba
\langle A(x,y,z), (x,y,z)^T \rangle
&=-\e (x^{2n+2}+y^{2n+2}+z^{2n+2})
-(a x^{2}+y^{2}+z^{2})\\
&\leq_C - \e   (x^{2}+y^{2}+z^{2})^{n+1} -  (x^{2}+y^{2}+z^{2}),
\ea
so that the drift $A$ is $1$-dissipative for $\e=0$ and
$2n+1$-dissipative for $\e\in (0,\infty)$.
Therefore, \eqref{e:Lorenz} with $n\in \mathbb N$ and $\e\in(0,\infty)$
may be viewed as a superlinearly dissipative extension of the Lorenz-84 model.

It is easy to check that the component $A_1$ is one-sided Lipschitz continuous. However, adding the tangential term
$A_2$ disrupts the one-sided Lipschitz structure for the entire vector field $A$ what is seen by the following example:
\ba
\langle A(x,1,0) - A(x,0,0),(x,1,0)^T - (x,0,0)^T \rangle
=x-1-\e \to+\infty,\quad x\to+\infty.
\ea

Further negative multivariate  examples of dissipative but not one-sided Lipschitz continuous dynamical systems
can be found in Chapter 4 in \cite{hutzenthaler2015numerical}.
\hfill $\Box$
\end{exa}

We finalize this discussion by giving two modifications of the Lorenz-84 model that
satisfy Assumptions of this paper.

The first example is two-dimensional.
\begin{exa}
Let $n\in\mathbb N$, $\e\in(0,\infty)$, $b\in (0,\infty)$, and $x_0\in \mathbb{R}$ be fixed, and let
\begin{equation}\label{e:Lorenz_frozen}
A(y,z)=-\e\begin{pmatrix}
             y^{2n+1}\\
             z^{2n+1}\\
             \end{pmatrix} +
   \begin{pmatrix}
             -y + x_0y-bx_0z \\
             -z + bx_0y+x_0z\\
             \end{pmatrix}
             = A_1(y,z)+A_2(y,z), \quad (y,z)^T\in \mathbb{R}^2.
\end{equation}
The function $A$ is just a $(y,z)$-part of the drift from the previous example,
with the variable $x$ `frozen' at the level $x=x_0$.
We have
\ba
\langle A(y,z),(y,z)^T\rangle=-\e (y^{2n+2}+y^{2n+2})+(x_0-1)y^2+(x_0-1)z^2\lc - (y^{2}+z^{2})^{n+1}+1,
\ea
so that the drift $A$ is $(2n+1)$-dissipative.
Next, $A_1$ is one-sided Lipschitz continuous as in the previous example, whereas $A_2$ is linear.
Therefore, $A$ satisfies $\mathbf{H}_{A}^{\mathrm{Lip}_+}$.
It is easy to verify that
$\mathbf{H}_{A_x,A_{xx}}$ is satisfied with $\chi=2n$. That is, $A$ satisfies
all the assumptions required in Theorem \ref{t:P0} with $\kappa=2n+1, \chi=2n$.
\hfill $\Box$
\end{exa}

\begin{exa}
\label{exa:Lorenz-radial}
 Let $n\in\mathbb N$ and $\e\in(0,\infty)$ and
\ba
\label{e:Aphi}
A(x,y,z)
&=  -\e (x^2+y^2+z^2)^n\begin{pmatrix}
             x\\
             y\\
             z\\
           \end{pmatrix}
+\begin{pmatrix}
- ax- y^2-z^2\\
 -y+xy-bxz\\
-z+bxy+xz\\
           \end{pmatrix}\\
& = A_1(x,y,z)+A_2(x,y,z)
           , \quad (x,y,z)^T\in \mathbb{R}^3.
\ea
The drift $A$ is clearly $(2n+1)$-dissipative.
Remarkably, despite the quadratic terms appearing in $A_2$, the resulting drift $A$ is still one-sided Lipschitz continuous.
To show this, we denote $u=(x,y,z)^T$. Then, see \eqref{e:dnorm},
\ba
\label{e:nablaA}
\nabla A_1 &= -2n \e \|u\|^{2n-2} uu^T
-\e \|u\|^{2n}\mathrm{Id},\\
\nabla A_2& =
\begin{pmatrix}
-a & -2y & -2z\\
y- bz  &x-1& -bx\\
by+z & bx & x-1
\end{pmatrix},\\
[\nabla A]_\text{sym}
&=\nabla A_1
+\begin{pmatrix}
-a & -(y+bz)/2 & (by-z)/2\\
-(y+bz)/2 & x-1 & 0\\
(by-z)/2& 0 & x-1
\end{pmatrix}.
\ea
Let $\phi\in\bR^3$ be a test vector. Then,
\ba
\phi^T [\nabla A]_\text{sym}\phi & =- 2n\e\|u\|^{2n-2}(\phi^T u)^2 -\e\|u\|^{2n}\|\phi\|^2 + \phi^T [\nabla A_2]_\text{sym}\phi\\
&\leq -\e\|u\|^{2n}\|\phi\|^2 + \phi^T [\nabla A_2]_\text{sym}\phi.
\ea
Since $\nabla A_2$ is a linear function of the coordinates, we have
\ba
| \phi^T [\nabla A_2]_\text{sym}\phi|&\leq C\| [\nabla A_2]_\text{sym} \|\|\phi\|^2 \leq C \|u\|\|\phi\|^2
\ea
for some $C\in (0,\infty)$.
Therefore, there is $L\in (0,\infty)$ such that
\ba
\phi^T [\nabla A]_\text{sym}\phi \leq -\e\|u\|^{2n}\|\phi\|^2 +C\|u\|\|\phi\|^2\leq L\|\phi\|^2.
\ea
It is left to verify condition $\mathbf{H}_{A_x,A_{xx}}$. It follows from \eqref{e:nablaA} that $\chi=2n$.
Furthermore, the Hesse matrices of the components of $A_1$ have the growth order $\mathcal O(\|u\|^{2n-1})$.
The Hesse matrices of the components of $A_2$ have the growth order $\mathcal O(1)$. The Jacobi matrix $\nabla A_2$ is of linear growth
with respect to $\|u\|$. Combining these estimates with the formula for $\nabla A_1$, we get that for any test vector
$\phi\in\mathbb R^3$
\ba
\|\phi\| & \sum_{k=1}^3 \Big(\|\nabla^2 A^k_1\| + \|\nabla^2 A^k_2\|\Big) |\phi^k|
+
\langle \nabla A_1\phi,\phi\rangle
+\langle \nabla A_2 \phi,\phi\rangle \\
&\leq C(1+\|u\|^{2n-1})\|\phi\|^2  -\e\|u\|^{2n}\|\phi\|^2 + C(1+\|u\|)\|\phi\|^2\\
&\lc \|\phi\|^2.
\ea
This shows that the system $A$ satisfies the setting of our paper. Although there is no benchmark values, we present here the
simulation results for the system $A$, perturbed by additive isotropic
three-dimensional
Cauchy noise $Z$ with the characteristic function $\E \ex^{\i \langle \lambda, Z_t\rangle}=\ex^{-t\|\lambda\|}$, $\lambda\in\bR^3$.
Since the ODE $\dot \Phi=A(\Phi)$ cannot be solved analytically, the flow map $\Phi$
was evaluated numerically using the adaptive explicit Runge--Kutta Cash--Karp 5(4) method (\verb|runge_kutta_cash_karp54|)
provided by the Boost.Odeint C++ library.
\begin{table}[h]
%\label{tab:sim2}
\begin{center}
\begin{tabular}{|l|c|c|c|c|c|c|c|c|c|}
\hline
$h_n$
& $ \langle X^n_T \rangle$
& $ \langle Y^n_T       \rangle$
& $ \langle Z^n_T \rangle$
& $ \langle |X^n_T|^2     \rangle$
& $ \langle |Y^n_T|^2 \rangle$
& $ \langle |Z^n_T|^2 \rangle$
& $ \langle X^n_TY_T^n     \rangle$
& $ \langle X^n_T Z^n_T \rangle$
& $ \langle Y^n_TZ_T^n \rangle$
\\\hline
$10^{-2}$ &  $-0.670$ & 0.000 &0.000& 1.036& 0.391& 0.392& 0.000 &0.001& 0.001 \\\hline
$10^{-3}$ &  $-0.682$ & 0.000& 0.000& 1.075& 0.420& 0.420& 0.000& 0.000& 0.000 \\\hline
$10^{-4}$ &  $-0.692$ & 0.000 & 0.000& 1.091& 0.429& 0.429& 0.001& 0.000& 0.000  \\\hline
$10^{-5}$ &  $-0.694$ & 0.000 & 0.000 & 1.099& 0.431& 0.432& 0.000& 0.000 &0.000  \\\hline
\end{tabular}
\end{center}
\caption{Empirical first and second moments of the modified Lorenz-84 system
\eqref{e:Aphi} driven by isotropic Cauchy noise, computed using the splitting scheme
\eqref{e:Xndisc};
$N=10^6$, $T=5$, $n=2$,
$\e=0.1$, $a=0.25$, $b=4$.
\label{tab:Lorenz}}
\end{table}
The computations were performed on a virtual Linux machine (see Example~\ref{ex:main} for the specifications).
The run time grew nearly linearly, ranging from 27 seconds for \(h_n=10^{-2}\) to about 7.5 hours for \(h_n=10^{-5}\).
\end{exa}

\section{Properties of the mapping $\Phi$\label{s:Phi}}

In this section, we study properties of the mapping $\Phi$ determined by the ODE \eqref{e:ODE}.
First we show that despite the superlinear growth of the drift $A$, its
continuity together with the
one-sided Lipschitz condition guarantees existence, uniqueness and exponential growth of the solution $\Phi$.

\begin{lem}
Let assumption $\mathbf{H}^{\mathrm{Lip}_+}_A$ hold.
Then the ODE \eqref{e:ODE}
has a unique global solution that satisfies the following estimates for all $x\in\bR^d$ and $t\in[0,\infty)$:
\begin{align}
\label{e:Phi}
&\|\Phi(t,x)\|\leq  \|x\| \ex^{Lt} + \|A(0)\|\frac{\ex^{Lt}-1}{L},\\
\label{e:Phi-x}
&\|\Phi(t,x)-x\|\leq \|A(x)\|\frac{\ex^{Lt}-1}{L}.
\end{align}
\end{lem}
\begin{proof}
1.  Existence and uniqueness of the ODE \eqref{e:ODE} is established in Lemma 3.19 in \cite{pardoux2014stochastic}.

\noindent
2. For the solution $\Phi$ we have:
\ba
\|\Phi(t,x)\|^2 & =\|x\|^2 +2\int_0^t \langle A(\Phi(s,x)),\Phi(s,x)\rangle\,\di s\\
&=\|x\|^2 +2\int_0^t \langle A(\Phi(s,x))-A(0),\Phi(s,x)-0\rangle\,\di s + 2\int_0^t \langle A(0),\Phi(s,x)\rangle\,\di s\\
&\leq \|x\|^2 + 2L\int_0^t \|\Phi(s,x)\|^2\,\di s + 2\|A(0)\|\int_0^t \|\Phi(s,x)\|\,\di s.
\ea
Hence by Theorem 4.9 in \cite{bainov2013integral}
\ba
\|\Phi(t,x)\|^2 & \leq \Big(\|x\| \ex^{Lt} + \|A(0)\|\int_0^t \ex^{Ls}\,\di s\Big)^2
\ea
and consequently,
\ba
\|\Phi(t,x)\|  & \leq  \|x\| \ex^{Lt} + \|A(0)\|\frac{\ex^{Lt}-1}{L}.
\ea

\noindent
3. Analogously, we write
\ba
\|\Phi(t,x)-x\|^2 & =2\int_0^t \langle A(\Phi(s,x)),\Phi(s,x)-x\rangle\,\di s\\
&=2\int_0^t \langle A(\Phi(s,x))-A(x),\Phi(s,x)-x\rangle\,\di s + 2\int_0^t \langle A(x),\Phi(s,x)-x\rangle\,\di s\\
&\leq  2L\int_0^t \|\Phi(s,x)-x\|^2\,\di s + 2\|A(x)\|\int_0^t \|\Phi(s,x)-x\|\,\di s,
\ea
so that as in 2.\ we obtain the estimate
\ba
\|\Phi(t,x)-x\|  & \leq \|A(x)\|\frac{\ex^{Lt}-1}{L}.
\ea
\end{proof}
\begin{cor}
Let assumptions $\mathbf{H}^{\mathrm{Lip}_+}_A$ and $\mathbf{H}_{A_x,A_{xx}}$ hold. Then, for
all $x\in\bR^d$,
$t\in [0,1]$ we have
\begin{align}
\label{e:Phi01}
\|\Phi(t,x)\|&\lc 1+\|x\|,\\
\label{e:Phi-x01}
\|\Phi(t,x)-x\|&\lc t(1+\|x\|^{\chi+1}).
\end{align}
\end{cor}
\begin{proof}
The estimate \eqref{e:Phi01} follows immediately from \eqref{e:Phi}.
The estimate \eqref{e:Phi-x01} follows from \eqref{e:Phi-x} and the integrated estimate \eqref{e:Ax-chi}.
\end{proof}

\begin{lem}
Let assumptions $\mathbf{H}^{\mathrm{Lip}_+}_A$ and
$\mathbf{H}_{A_x,A_{xx}}$ hold. Then, for all $x\in\bR^d$, $t\in [0,1]$,
\begin{align}
\label{e:Phix01}
\|\Phi_x(t,x)\|& \leq \sqrt{d}\ex^{Lt},\\
\label{e:Phix-x01}
\|\Phi_x(t,x)-\mathrm{Id}\|&\lc   t (1+\|x\|^\chi).
\end{align}
\end{lem}
\begin{proof}
The gradient matrix $\Phi_x$ satisfies the matrix ODE
\ba
\dot \Phi_x(t,x)&= A_x(\Phi(t,x) )\Phi_x(t,x),\\
\Phi_x(0,x)&=\mathrm{Id}.
\ea
For each $\phi\in\bR^d$ we have
\ba
\frac{\di }{\di t}\|\Phi_{x}(t,x) \phi\|^2
& = \frac{\di }{\di t} \langle \Phi_{x}(t,x)\phi, \Phi_{x}(t,x)\phi\rangle\\
&=\langle \dot \Phi_{x}(t,x)\phi, \Phi_x(t,x)\phi\rangle + \langle \Phi_{x}(t,x)\phi, \dot \Phi_x(t,x)\phi\rangle\\
&=\langle (A_x(\Phi(t,x))+ A_x(\Phi(t,x)^T)   \Phi_{x}(t,x)\phi,  \Phi_{x}(t,x)\phi\rangle\\
&\leq 2 L\|\Phi_{x}(t,x) \phi\|^2.
\ea
Recalling that $\|\Phi_x\|$ denotes the Euclidean matrix norm, by Gronwall's Lemma we obtain the estimate
\ba
\|\Phi_{x}(t,x)\|\leq \sqrt{d}\ex^{Lt},
\ea
which proves \eqref{e:Phix01}.
Furthermore using \eqref{e:Phi01} and \eqref{e:Phix01}  we get for $t\in [0,1]$:
\ba
\label{e:Phi1}
\|\Phi_{x}(t,x)-\mathrm{Id}\|&\leq \int_0^t \|A_x(\Phi(s,x))\| \|\Phi_x(s,x)\|\,\di s\\
&\leq_C t \max_{s\in[0,t]} \|A_x(\Phi(s,x))\|\\
&\leq_C t \max_{s\in[0,t]} (1+ \|\Phi(s,x)\|^\chi)\\
&\leq_C t (1+ (1+\|x\|)^\chi)\\
&\leq_C t (1+ \|x\|^\chi).
\ea
\end{proof}

\begin{lem}
Let assumptions $\mathbf{H}^{\mathrm{Lip}_+}_A$ and
$\mathbf{H}_{A_x,A_{xx}}$ hold. Then, for all $x\in\bR^d$, $t\in [0,1]$,
\begin{align}
\label{e:Phixx}
\|\Phi_{xx}(t,x)\|&\lc 1,\\
\label{e:Phixx-0}
\|\Phi_{xx}(t,x)\|&\lc   t (1+\|x\|^\chi).
\end{align}
\end{lem}
\begin{proof}
To prove \eqref{e:Phixx}, it is enough to show that, for $i,j=1,\dots,d$,
\ba\label{e:Phixxij}
\|\Phi_{x^ix^j}(t,x)\|\lc 1.
\ea
The derivative $\Phi_{x^ix^j}$ satisfies the ODE
\ba
\frac{\di }{\di t}\Phi_{x^ix^j}(t,x)&= \partial_{x^j}\Big(A_{x}(\Phi(t,x) )\Phi_{x^i}(t,x)\Big),\\
\Phi_{x^ix^j}(0,x)&=0.
\ea
Hence, for each coordinate $k=1,\dots, d$ we have
\ba
\label{e:Phixx1}
\frac{\di }{\di t} \Phi^k_{x^ix^j} &= \partial_{x^j}\Big(\sum_{l=1}^d A^k_{x^l}(\Phi )\Phi^l_{x^i} \Big)\\
&=\sum_{l,m=1}^d A^k_{x^lx^m}(\Phi )\Phi^m_{x^j} \Phi^l_{x^i} + \sum_{l=1}^d A^k_{x^l}(\Phi )\Phi^l_{x^ix^j}\\
&=\langle A^k_{xx}(\Phi)\Phi_{x^i}, \Phi_{x^j}\rangle +\langle  A^k_x(\Phi),\Phi_{x^ix^j}\rangle.
\ea
Denote by $Q_{ij}(t,x)$ the vector with the coordinates
\ba
Q^k_{ij}(t,x):=\langle A^k_{xx}(\Phi(t,x))\Phi_{x^i}(t,x), \Phi_{x^j}(t,x)\rangle,\quad k=1,\dots, d.
\ea
Then the ODE \eqref{e:Phixx1} can be rewritten in the following matrix form
\begin{equation}
\label{e:dPhixx}
\frac{\di }{\di t} \Phi_{x^ix^j}(t,x) = Q_{ij}(t,x) +A_x(\Phi(t,x))\Phi_{x^ix^j}(t,x)
\end{equation}
and therefore
\ba
\frac{\di }{\di t}\|\Phi_{x^ix^j}(t,x)\|^2 & = 2 \langle \frac{\di }{\di t} \Phi_{x^ix^j}(t,x) ,\Phi_{x^ix^j}(t,x) \rangle \\
&= 2\langle Q_{ij}(t,x) ,\Phi_{x^ix^j}(t,x) \rangle + 2\langle A_x(\Phi(t,x))\Phi_{x^ix^j}(t,x), \Phi_{x^ix^j}(t,x)\rangle.
\ea
By \eqref{e:Phix01} we have
\begin{equation}\label{e:Q0}
|Q^k_{ij}(t,x)|\leq \|A^k_{xx}(\Phi(t,x))\| \|\Phi_{x^i}(t,x)\| \|\Phi_{x^j}(t,x)\|
\lc \|A^k_{xx}(\Phi(t,x))\|,
\end{equation}
so that for some constant $C_*\in(0,\infty)$ we get
\begin{equation}
\label{e:Q}
\langle Q_{ij}(t,x) ,\Phi_{x^ix^j}(t,x) \rangle\leq C_*\sum_{k=1}^d\|A^k_{xx}(\Phi(t,x))\| |\Phi_{x^ix^j}^k(t,x)|.
\end{equation}
Let $\e\in(0,\infty)$ be the constant from Assumption \eqref{e:Axx} and
let $t\in[0,1]$. Recall that
$\Phi_{x^ix^j}(0,x)=0$.
If
\ba
\|\Phi_{x^ix^j}(t,x)\|\leq  C_*\eps^{-1},
\ea
then we have nothing to prove.
If
\ba
\|\Phi_{x^ix^j}(t,x)\|> C_*\eps^{-1},
\ea
then there exists
$t_*(x)\in(0,t)$ such that
$\|\Phi_{x^ix^j}(t_*(x),x)\|= C_*\e^{-1}$ and
\ba
\|\Phi_{x^ix^j}(s,x)\|\geq C_*\eps^{-1},\quad  s\in [t_*(x), t].
\ea
The latter inequality and assumption \eqref{e:Axx} applied to the mapping $s\mapsto \Phi_{x^ix^j}(s,x)$
on the interval
$s\in[t_*(x),t]$ yield
\ba
\frac{\di }{\di s}\|\Phi_{x^ix^j}(s,x)\|^2&=2\langle Q_{ij}(s,x) ,\Phi_{x^ix^j}(s,x) \rangle
+ 2\langle A_x(\Phi(s,x))\Phi_{x^ix^j}(s,x), \Phi_{x^ix^j}(t,x)\rangle\\
&\leq 2C_*\sum_{k=1}^d \|A^k_{xx}(\Phi(s,x))\| |\Phi_{x^ix^j}^k(s,x)|
+ 2\langle A_x(\Phi(s,x))\Phi_{x_ix_j}(s,x), \Phi_{x^ix^j}(s,x)\rangle\\
& \leq 2\eps\|\Phi_{x_ix_j}(s,x)\|\sum_{k=1}^d\|A^k_{xx}(\Phi(s,x))\| |\Phi^k_{x^ix^j}(s,x)|
+ 2\langle A_x(\Phi(s,x))\Phi_{x_ix_j}(s,x), \Phi_{x^ix^j}(s,x)\rangle\\
&\leq C_1\|\Phi_{x^ix^j}(s,x)\|^2
\ea
with some $C_1\in(0,\infty)$. Hence by the Gronwall Lemma we get that
\ba
\|\Phi_{x^ix^j}(t,x)\|^2\leq C_*^2\eps^{-2}\ex^{{C_1}(t-t_*)}\leq C_*^2\eps^{-2}\ex^{{C_1}t}\leq C_*^2\eps^{-2}\ex^{C_1},
\ea
which completes the proof of \eqref{e:Phixxij} and, consequently, \eqref{e:Phixx}.

Using \eqref{e:dPhixx}, \eqref{e:Phixxij}, and \eqref{e:Q0}, we get that
\ba
\|\Phi_{x^ix^j}(t,x)\|
=\|\Phi_{x^ix^j}(t,x)-\Phi_{x^ix^j}(0,x)\|
\lc \int_0^t\Big( \|A_{xx}(\Phi(s,x))\|+\|A_x(\Phi(s,x))\|\Big)\, \di s.
\ea
Then using \eqref{e:Ax-chi}, \eqref{e:Axx}, and repeating the estimates from \eqref{e:Phi1}, we get that
\ba
\|\Phi_{x^ix^j}(t,x)\|\lc t(1+\|x\|^\chi),
\ea
which proves \eqref{e:Phixx-0}.
\end{proof}

\section{Proofs of Theorems \ref{t:Xpdiss} and \ref{t:Xnpdiss}\label{s:3}}

We will use the L\'evy--It\^o representation of the L\'evy process $Z$ in the following form:
\ba
\label{e:Z}
Z_t=\int_0^t \int_{\|z\|\leq 1} z\widetilde N(\di z,\di s)  + \int_0^t \int_{\|z\|> 1} z\, N(\di z,\di s),
\ea
where $N(\di z,\di t)$ is a Poisson random measure on $(\bR^d\times \bR_+,\rB(\bR^d\times \bR_+))$ with the intensity measure
$\nu(\di z)\,\di t$, and $\widetilde N(\di z,\di t)=N(\di z,\di t)-\nu(\di z)\,\di t$
is the compensated Poisson random measure.
The stochastic integral with respect to $Z$ in \eqref{e:1}
is understood as a sum of stochastic integrals
with respect to $N$ and
$\widetilde N$, i.e.\ the SDE \eqref{e:1} can be equivalently rewritten as
\ba
X_t=x & + \int_0^t A(X_s)\,\di s + \int_0^t a(X_{s})\,\di s + \int_0^t b(X_{s})\,\di B_s\\
&+ \int_0^t\int_{\|z\|\leq 1} c(X_{s-})z\, \widetilde N(\di z,\di s) +
\int_0^t \int_{\|z\|>1} c(X_{s-})z\, N(\di z,\di s).
\ea

We deduce statements of both Theorem \ref{t:Xpdiss} and  Theorem \ref{t:Xnpdiss} from a general moment
estimate provided by \cite[Theorem 2.12]{KP21}. Specifically, existence and uniqueness of a local solution
of the SDE \eqref{e:1} is guaranteed, e.g., by Theorem 6.2.11 in \cite{Applebaum-09}. Applying \cite[Theorem 2.12]{KP21}
we get immediately that this solution is global and that the estimate \eqref{e:estXpX} holds,
which proves Theorem \ref{t:Xpdiss}.

To prove Theorem \ref{t:Xnpdiss}, we show that approximations $\{X^n\}$ defined in \eqref{e:Xncont},
while considered as continuous time semimartingales, satisfy conditions of \cite[Theorem 2.12]{KP21}
uniformly over $n\in\mathbb N$.

For each $n\in\mathbb N$,
we apply the It\^o formula to \eqref{e:Xncont} to see  that $X^n$ is
an It\^o semimartigale with the following representation:
\ba
\label{e:Xn}
X^n_t&= x + \int_{0}^t A(X^n_s)\,\di s
+ \int_0^t \Phi_x(\delta^n_s,Y^n_s) a(X^{(n)}_s)\,\di s\\
&+ \int_0^t \Phi_x(\delta^n_s,Y^n_s) b(X^{(n)}_s) \,\di B_s
+\frac12 \int_0^t  \operatorname{tr}(\Phi_{xx}(\delta^n_s,Y^n_s)bb^T(X^{(n)}_s ))\, \di s\\
&+\int_0^t\int_{|z|\leq 1} \Big[ \Phi(\delta^n_s, Y^n_{s} +  c(X^{(n)}_s)z) -
\Phi(\delta^n_s, Y^n_{s}) -\Phi_x (\delta^n_s, Y^n_s) c(X^{(n)}_s)z \Big]\,\nu(\di z)\,\di s\\
&+ \int_0^t\int_{|z|\leq 1} \Big[ \Phi(\delta^n_s, Y^n_{s-} +  c(X^{(n)}_s   )z) -
\Phi(\delta^n_s, Y^n_{s-})
\Big]\,\widetilde N(\di z,\di s)\\
&+ \int_0^t\int_{|z|> 1} \Big[ \Phi(\delta^n_s, Y^n_{s-} +  c(X^{(n)}_s   )z) -
\Phi(\delta^n_s, Y^n_{s-})
\Big]\,N(\di z,\di s).
\ea
We rewrite \eqref{e:Xn} in the canonical L\'evy--It\^o form:
\ba
\label{e:Xncanon}
X^n_t = x
+ \int_0^t a^n_s\,\di s
+ \int_0^t b^n_s\,\di B_s
+ \int_0^t\int_{\bR^d} \zeta \Big( N^n(\di \zeta, \di s) - \bI(\|\zeta\|\leq 1) \nu^n(\di \zeta, \di s)\Big)
\ea
with the local progressively measurable characteristics
\ba
a^n_s & = A(X^n_s) + \Phi_x(\delta^n_s,Y^n_s) a(X^{(n)}_s)
+ \frac12 \operatorname{tr}(\Phi_{xx}(\delta^n_s,Y^n_s)bb^T(X^{(n)}_s))\\
&+\int_0^t\int_{|z|\leq 1} \Big[ \Phi(\delta^n_s, Y^n_s +  c(X^{(n)}_s)z) -
\Phi(\delta^n_s, Y^n_s) -\Phi_x (\delta^n_s, Y^n_s) c(X^{(n)}_t)z \Big]\,\nu(\di z)\,\di s,\\
b^n_s&= \Phi_x(\delta^n_s,Y^n_s) b(X^{(n)}_s) ,\\
\ea
and the random jump measure $N^n$ with
the
predictable characteristics
\ba
\nu^n(\di \zeta,\di s)&= K^n_{s-}(\di \zeta)\,\di s,\\
\ea
where
\ba
K^n_s(U) &= \nu(z\in \bR^d\backslash\{0\}
\colon \Phi(\delta^n_s, Y^n_{s} + c(X^{(n)}_s)z) - \Phi(\delta^n_s, Y^n_{s}) \in U),\quad
U\in \rB(\bR^n).
\ea
We check the assumptions of Theorem 2.12 in \cite{KP21} for the terms of the canonical decomposition \eqref{e:Xncanon}.

Assumption $\mathbf{A}_M$
holds true  since for all $s\in[0,\infty)$ and $n\in\bN$
\ba
\|b^n_s\|\leq_C \sup_{u\in[0,h_n]}\sup_{x\in\bR^d}\|\Phi_x(u,x)\|\|b\| \leq_C 1\quad \text{a.s.}
\ea
For $s\in[0,\infty)$ we denote
\ba
\label{e:Psi}
\Psi^n_s(z)&=\Phi(\delta^n_s, Y^n_{s} + c(X^{(n)}_s)z) - \Phi(\delta^n_s, Y^n_{s})
=\int_0^1 \Phi_x(\delta^n_s, Y^n_{s} + \theta c(X^{(n)}_s)z )c(X^{(n)}_s) z\,\di \theta,
\ea
so that the following estimate holds:
\ba
\|\Psi^n_s(z)\|&\leq  \sup_{s\in[0,h_n]}\sup_{x\in\bR^d}\|\Phi_x(s,x)\|\|c\|\|z\|\leq_C\|z\|.
\ea
Assumption $\mathbf{A}_{\nu,\leq 1}$ is satisfied because
\ba
\int_{\|\zeta\|\leq 1} \|\zeta\|^2\, K^n_s(\di \zeta)
&=\int_{\bR^d} \bI( \| \Psi^n_s(z)\|\leq 1) \| \Psi^n_s(z)  \|^2 \, \nu( \di z)\\
&\leq \int_{\bR^d} 1\wedge \| \Psi^n_s(z)  \|^2 \, \nu( \di z)
\leq_C \int_{\bR^d} 1\wedge \| z \|^2 \, \nu( \di z)\leq_C 1 \quad \text{a.s.}
\ea
Assumption $\mathbf{A}_{\nu,p}$ is satisfied because
\ba
\int_{\|\zeta\|> 1} \|\zeta\|^p\, \nu^n(\di \zeta)
=\int_{\bR^d} \bI( \|  \Psi^n_s(z)\|> 1) \| \Psi^n_s(z)  \|^p \, K_s^n( \di z)
\leq_C \int_{\|z\|>1} \| z \|^p \, \nu( \di z)\leq_C 1\quad \text{a.s.}
\ea
Finally, we estimate the drift term. With the help of \eqref{e:Phix01} and \eqref{e:Phixx} we get for $n$ large enough:
\ba
&\|\Phi_x(\delta^n_s,Y^n_s) a(X^{(n)}_s)\|
\leq_C 1,\\
&\|\operatorname{tr}(\Phi_{xx}(\delta^n_s,Y^n_s)bb^T(X^{(n)}_s))\|\leq_C 1,\\
&\int_{|z|\leq 1} \Big[ \Phi(\delta^n_s, Y^n_{s} + c(X^{(n)}_s)z)  -
X^n_s -\Phi_x (\delta^n_s, Y^n_s)c(X^{(n)}_s)z
\Big]\,\nu  (\di z)\\
&\leq \frac12 \|c\| \int_{|z|\leq 1}\|z\|^2 \|\Phi_{xx}(\delta^n_s, \theta^n_s(z))\|\,\di z \leq_C 1.
\ea
Hence the drift $a^n_s$ is locally bounded (condition $\mathbf{A}_{a,\textrm{loc}}$). Furthermore,
\ba
\langle X^n_s , a^n_s \rangle
&\leq -C_1 \|X_s^n\|^{1+\kappa} + C_2 + C_3 \|X^n_s\|\\
&\leq  -C_1\|X_s^n\|^{1+\kappa}+ C_4,
\ea
so that the drift $a^n_s$ is also dissipative with the dissipation rate $\kappa$ (condition $\mathbf{A}_{a,\kappa}$). Finally,
the balance condition $p+\kappa>1$ (condition $\mathbf{A}_\mathrm{balance}$) holds
because $p\in(0,\infty)$
and $\kappa\in(1,\infty)$.

To finish the proof we note that the constant $C$ in the statement of \cite[Theorem 2.12]{KP21}
depends only on the constants in the conditions
$\mathbf{A}_M$, $\mathbf{A}_{\nu,\leq 1}$, $\mathbf{A}_{\nu,p}$ and $\mathbf{A}_{a,\kappa}$.
Thus, the estimate (2.30) in \cite{KP21} yields a uniform bound for the family $\{X^n\}$.

\section{Proof of Theorem \ref{t:thm_q}\label{s:4}}

Let $\widehat Z$ and $Q$ be the large and small jump components of $Z$ appearing in the L\'evy--It\^o decomposition \eqref{e:Z}:
\ba
\widehat Z_t:=\int_0^t \int_{\|z\|\leq 1} z\widetilde N(\di z,\di s),\quad Q_t:= \int_0^t \int_{\|z\|> 1} z\, N(\di z,\di s).
\ea
Let $\{\tau_k,J_k\}_{k\in \bN}$ be the
sequence of the arrival times and jump sizes of the compound Poisson process $Q$.
It is well known that the process $\widehat Z$ and the sequences $\{\tau_k\}_{k\in \bN}$ and $\{J_k\}_{k\in \bN}$
are independent.
The process $Q$ has jump intensity $\lambda=\nu(\|z\|>1)$. We
assume that $\lambda\in(0,\infty)$, otherwise the arguments below can be substantially simplified.

For $T\in[0,\infty)$, we denote by
\ba
N_T=N(\{\|z\|>1\}\times [0,T])
\ea
the number of jumps of $Q$ on the time interval $[0, T]$. It is well known that $N_T$
has the Poisson distribution with the parameter $T\lambda$, and, conditioned on the event
$\{N_T=N\}$, $N\in \bN$, the jump times $\{\tau_k\}_{1\leq k\leq N}$ have the law of the order statistics
of the uniform distribution on $[0,T]$, see, e.g., Proposition 3.4 in \cite{Sato-99}.

For $N\in\bN_0$ and $0<t_1<\dots< t_N<T$, by $\P^{(N)}_{t_1,\dots, t_N}(\cdot)$ we denote the regular conditional probability
\ba\label{e:cPP}
\P^{(N)}_{t_1,\dots, t_N}(\cdot )=\P(\cdot |N_T=N, \tau_1=t_1,\dots, \tau_N=t_N).
\ea
Under $\P^{(N)}_{t_1,\dots, t_N}$,
the driving noise for both processes $X^n$ and $X$ is given by the Brownian motion $B$ and the independent L\'evy process
$\widehat Z$, interlaced by large
jumps $J_1,\dots, J_N$ at fixed time instants $t_1,\dots, t_N$.
Under $\P^{(N)}_{t_1,\dots, t_N}$, the jump sizes $J_1,\dots, J_N$ are iid random vectors, independent of
$B$ and $\widehat Z$, with the probability law
$\lambda^{-1}\nu(\cdot \cap \{\|z\|>1\})$.

We also denote $t_0:=0$ and $t_{N+1}:=T$. Note that for $N=0$,
the actual large jumps are absent, and the estimates below are valid on the entire time interval $[t_0, t_{1}]=[0, T]$.
For $N\in\bN$, the large jump at time $t_{N+1}=T$ a.s. does not occur either, so that
estimates below are valid on the closed time interval $[t_N, t_{N+1}]$. Finally, we mention that for all $N\in\bN_0$ and $k=0,\dots, N+1$, $\Delta \widehat Z_{t_k}=0$ $\P^{(N)}_{t_1,\dots, t_N}$-a.s.

The SDE with bounded jumps and one-sided Lipschitz continuous drift
\ba
\label{e:Itohat}
\widehat X_t=x  + \int_0^t A(\widehat X_s)\,\di s + \int_0^t a(\widehat X_{s})\,\di s + \int_0^t b(\widehat X_{s})\,\di B_s
+ \int_0^t\int_{\|z\|\leq 1} c(\widehat X_{s-})z\, \widetilde N(\di z,\di s)
\ea
has a unique strong solution following the same argument as the one presented in the previous section for the original SDE \eqref{e:1}.
The solution $X$ to \eqref{e:1} can be obtained by interlacing the solution $\widehat X$ with the large jumps
that occur at the arrival times of the compound Poisson process $Q$,
see, e.g., Section 3.5 in \cite{Kunita-04}.

We introduce the following processes (recall \eqref{e:ns} and \eqref{e:Xncont}):
\ba
\Delta_t^n&=X_t^n-X_t, \quad \Xn_t=X^n_{\eta_t^n}, \quad \XN_t=\Phi(\delta_t^n, \Xn_t),\\
Y^n_t&=X^{(n)}_t+a(X^{(n)}_t)\delta^n_t+b(X^{(n)}_t)B^n_t+c(X^{(n)}_t)Z^n_t.
\ea
Denote also
\ba
U^n_t=h_n+\|B_t^n\|+\|Z_t^n\|.
\ea
Let us give some preparatory estimates. First, combining \eqref{e:Phi01} and \eqref{e:Phi-x01} with the elementary inequality
\eqref{e:max}, we get that for any $\gamma\in[0,1]$
\begin{equation}
\label{e:Phig}
  \|\Phi(t,x)-x\|\leq_C t^\gamma(1+\|x\|^{\gamma\chi+\gamma})(1+\|x\|^{1-\gamma})\leq_C t^\gamma(1+\|x\|^{\gamma\chi+1}).
\end{equation}
Similarly, from \eqref{e:Phix01} and \eqref{e:Phix-x01}, and \eqref{e:Phixx} and \eqref{e:Phixx-0}
we get
\begin{align}\label{e:Phih}
  \|\Phi_x(t,x)-\mathrm{Id}\|\lc t^\gamma(1+\|x\|^{\gamma\chi}),
  \\\label{e:Phij}
  \|\Phi_{xx}(t,x)\|\lc t^\gamma(1+\|x\|^{\gamma\chi}).
\end{align}
Next, we have the following elementary estimates.
\begin{lem}
\label{lA10}
We have
\begin{align}
\label{estimates}
\|X_t^n-\XN_t\|&\leq_C U^n_t,\\
\label{estimates2}
\|\Phi_x(\delta_t^n, Y_t^n)-\Phi_x(\delta_t^n, \Xn_t)\| &\leq_C U^n_t,
\end{align}
and for any $\gamma\in [0,1]$ we have
\begin{align}
\label{estimates1}
\|\XN_t-\Xn_t\|&\leq_C h_n^{\gamma}(1+\|X^{(n)}_t\|^{\gamma\chi+1}),\\
\label{estimates0}
\|X^n_t-\Xn_t\|&\leq_C h_n^{\gamma}(1+\|X^{(n)}_t\|^{\gamma\chi+1})  + U^n_t.
\end{align}
\end{lem}

\begin{proof}
1. We take into account the boundnedness of $\Phi_x$ (see \eqref{e:Phix01}) and the boundnedness of the coefficients $a$, $b$, $c$
to get \eqref{estimates}:
\ba
\|X_t^n-\XN_t\|&=\|\Phi(\delta_t^n, Y_t^n)-\Phi(\delta_t^n, \Xn_t)\|\\
&\leq \|\Phi_x\|\|Y_t^n-\Xn_t\|\\
&\leq_C \|a(X^{(n)}_t)\delta^n_t+b(X^{(n)}_t)B^n_t+c(X^{(n)}_t)Z^n_t\|\\
&\lc U^n_t.
\ea
2. Similarly, the boundnedness of $\Phi_{xx}$ (see \eqref{e:Phixx}) yields
 \eqref{estimates2}.

\noindent
3. Since
\ba
\XN_t-\Xn_t=\Phi(\delta_t^n, \Xn_t)-\Xn_t,
\ea
\eqref{estimates1} follows immediately by \eqref{e:Phig}.

\noindent
4. Finally,
\ba
\|X^n_t-\Xn_t\|\leq  \|X^n_t-\XN_t\|\ + \|\XN_t -\Xn_t\|
\leq h_n^{\gamma}(1+\|X^{(n)}_t\|^{\gamma\chi+1})  + U^n_t.
\ea
\end{proof}

We will need the following auxiliary moment estimate.

 \begin{lem}
\label{l:estXn}
For any $r\in(0,\infty)$, any $T\in[0,\infty)$, any $N\in\bN_0$ and all $k=0,\dots,N$ the following estimate holds:
\ba
\label{e:estXn}
\sup_{t\in[t_k,t_{k+1})}
\E^{(N)}_{t_1,\dots,t_N} \Big[\|X_t^n\|^r \Big|\rF_{t_k}\Big] \leq_C 1+\|X_{t_k}^n\|^r
\ea
uniformly over $n\in\bN$.
\end{lem}
\begin{proof}
It is sufficient to consider $r\in[2,\infty)$.  Applying the It\^o formula with $f(x)=\|x\|^r$ to the semimartingale
\eqref{e:Xn} we get for $t\in [t_k, t_{k+1})$ the following representation:
\ba
\label{ItoXn}
\|X_t^n\|^r
&=\|X_{t_k}^n\|^r+r \int_{t_k}^t \|X_s^n\|^{r-2}\langle X_s, A(X_s^n) \rangle \,\di s
+  r \int_{t_k}^t\|X_s^n\|^{r-2} \langle X_s^n, \Phi_x(\delta^n_s,Y^n_s) a(X^{(n)}_s) \rangle  \,\di s\\
&+ \frac{r}{2}
\int_{t_k}^t\|X_s^n\|^{r-2} \langle X_s^n,  \operatorname{tr}(\Phi_{xx}(\delta^n_s,Y^n_s)  bb^T(X^{(n)}_s )) \rangle  \,\di s\\
&+\frac{r(r-1)}{2} \int_{t_k}^t\|X_s^n\|^{r-4}
\|(\Phi_x(\delta^n_s,Y^n_s) b(X^{(n)}_s))^T X^n_s\|^2\, \di s\\
&+\frac{r}{2} \int_{t_k}^t\|X_s^n\|^{r-2}
\|  \Phi_x(\delta^n_s,Y^n_s) b(X^{(n)}_s)  \|^2 \, \di s\\
&+
 r \int_{t_k}^t\|X_s^n\|^{r-2} \langle X_s^n,
\int_{\|z\|\leq 1}
  \Big[ \Phi(\delta^n_s, Y^n_s +  c(X^{(n)}_s)z) -
\Phi(\delta^n_s, Y^n_s) -\Phi_x (\delta^n_s, Y^n_s) c(X^{(n)}_s)z \Big]
\,\nu(\di z) \rangle  \,\di s\\
&+
\int_{\|z\|\leq 1}
\Big( \|X_{s}^n+  \Phi(\delta^n_s, Y^n_s +  c(X^{(n)}_s   )z) - \Phi(\delta^n_s, Y^n_{s})   \|^r-\|X_{s}^n\|^r\\
&\qquad\qquad -r\|X_{s}^n\|^{r-2} \langle X_{s}^n,  \Phi(\delta^n_s, Y^n_s +  c(X^{(n)}_s   )z - \Phi(\delta^n_s, Y^n_{s}) \rangle
\Big)\, \nu(\di z)\, \di s\\
&+r \int_{t_k}^t\|X_s^n\|^{r-2} \langle X_s^n ,  \Phi_x(\delta^n_s,Y^n_s) b(X^{(n)}_s)\,\di  B_s \rangle \\
&+\int_{t_k}^t\int_{\|z\|\leq 1}
\Big( \|X_{s-}^n+  \Phi(\delta^n_s, Y^n_{s-} +  c(X^{(n)}_s   )z) -
\Phi(\delta^n_s, Y^n_{s-})   \|^r-\|X_{s-}^n\|^r\Big)\,
\widetilde{N}(\di z, \di s).
\ea

Recall that $\Phi_x(\cdot,\cdot )$, $\Phi_{xx}(\cdot,\cdot )$, $a(\cdot)$, $b(\cdot)$ and $c(\cdot)$
are bounded and $\|z\|\leq 1$, hence by the Taylor formula
\ba
\Big\| \Phi(\delta^n_s, Y^n_s +  c(X^{(n)}_s)z) -
\Phi(\delta^n_s, Y^n_s) -\Phi_x (\delta^n_s, Y^n_s) c(X^{(n)}_s)z \Big\|\leq_C\|z\|^2
\ea
and
\ba
\Big| \|X_{s}^n+  \Phi(\delta^n_s, Y^n_s +  c(X^{(n)}_s   )z) & - \Phi(\delta^n_s, Y^n_{s})   \|^r-\|X_{s}^n\|^r
-r\|X_{s}^n\|^{r-2} \langle X_{s}^n,  \Phi(\delta^n_s, Y^n_s +  c(X^{(n)}_s   )z - \Phi(\delta^n_s, Y^n_{s}) \rangle
\Big|\\
&\leq_C (1+ \|X_t^n\|^r)\|z\|^2.
\ea
Therefore
\ba
\|X_t^n\|^r=\|X_{t_k}^n\|^r  + \int_{t_k}^t D_s^{n,k}\, \di s+ M_t^{n,k}, \quad t\in [t_k, t_{k+1}),
\ea
with a local martingale $M^{n,k}$, $M^{n,k}_{t_{k}}=0$, and the drift $D^{n,k}$ that satisfies
\begin{equation}
D_t^{n,k}\leq_C  1+ \|X_t^n\|^r .
\end{equation}
The estimate \eqref{e:estXn} follows by the standard argument, involving localization, the Gronwall lemma, and the Fatou lemma.
\end{proof}

With these preparations complete, we now  proceed with the  proof of Theorem \ref{t:thm_q}. The first step is provided by the following
\begin{lem}
\label{lA1Delta}
Let $r\in[2,\infty)$. Then, for any $\gamma\in [0,1]$
\ba
\sup_{t\in[t_k,t_{k+1})}\E^{(N)}_{t_1,\dots,t_N}\Big[\|\Delta_t^n\|^r\Big|\mathcal{F}_{t_k}\Big]
\leq_C \|\Delta_{t_k}^n\|^r
&+h_n^{r\gamma +1}(1+\|X_{t_k}^{(n)}\|^{r\gamma\chi+r}) +h_n^{r\gamma }(1 + \|X^n_{t_k}\|^{r\gamma\chi+r})\\
&+h_n  (U_{t_k}^n)^r
+ h_n^{r\gamma+1} (U_{t_k}^n)^{r\gamma\chi}  + h_n
\ea
for $n$ large enough uniformly over
$N\in\bN_0$, partitions $\{t_1,\dots, t_N\}$, and $k=0, \dots, N$.
\end{lem}

\begin{proof}
We divide the proof of Lemma
\ref{lA1Delta} into five steps.

\noindent
i) Semimartingale representation of
$\Delta^n$. On
$[t_k, t_{k+1})$, there are no large jumps for $Z$, hence on this interval we have
\ba
\label{e:ItoX}
\di X_t= A(X_t)\,\di t + a( X_{t})\,\di t + b( X_{t})\,\di B_t
+ \int_{\|z\|\leq 1}
c( X_{t-})z\, \widetilde N(\di z,\di t).
\ea
Together with a semi-martingale representation \eqref{e:Xn} for $X^n$ this gives that, for $t\in [t_k, t_{k+1})$,
\ba
\label{e:ItoDelta}
\di\Delta_t^n
&=(A(X_t^n)-A(X_t))\,\di t\\
&+\Big(\Phi_x(\delta_t^n, Y_t^n)a(X_t^{(n)})-a(X_t)\Big)\,\di t\\
&+\frac12 \operatorname{tr}\Big(\Phi_{xx}(\delta_t^n, Y_t^n)bb^T(\Xn_t)\Big)\, \di t\\
&+\int_{|z|\leq 1}
\Big(\Phi(\delta_t^n, Y_{t}^n+c(\Xn_{t})z)-\Phi(\delta_t^n, Y_{t}^n)-\Phi_x(\delta_t^n, Y_{t}^n)c(\Xn_{t})z\Big)\,\nu(\di z)\,\di t
\\&+\Big(\Phi_x(\delta_t^n, Y_t^n)b(X_t^{(n)})-b(X_t)\Big)\,\di B_t,
\\&+\int_{|z|\leq 1}\Big(\Phi(\delta_t^n, Y_{t-}^n+c(\Xn_{t-})z)-\Phi(\delta_t^n, Y_{t-}^n)-c(X_{t-})z\Big)\widetilde N(\di z, \di t).
\ea
We decompose the terms \eqref{e:ItoDelta} further to facilitate effective approximation estimates.
We set
\ba
\Phi_x(\delta_t^n, Y_t^n)a(X_t^{(n)})-a(X_t)&=a(X^n_t)-a(X_t)\\
&+a(\XN_t)-a(X_t^n)\\
&+a(\Xn_t)-a(\XN_t)\\
&+\Big((\Phi_x(\delta_t^n, Y_t^n) - \Phi_x(\delta_t^n, \Xn_t)\Big)a(\Xn_t)\\
&+\Big(\Phi_x(\delta_t^n, \Xn_t)-\mathrm{Id}\Big)a(\Xn_t)\\
&=:\sum_{i=0}^4\Lambda_t^{a,i,n}
\ea
and similarly
\ba
\Phi_x(\delta_t^n, Y_t^n)b(X_t^{(n)})-b(X_t)=:\sum_{i=0}^4\Lambda_t^{b,i,n}.
\ea
We also decompose
\ba
\Phi(\delta_t^n, Y_t^n+c(\Xn_t)z)&-\Phi(\delta_t^n, Y_t^n)-c(X_t)z
\\&=c(X^n_{t})z-c(X_{t})z
\\&+c(\XN_{t})z-c(X_{t}^n)z
\\&+c(\Xn_{t})z-c(\XN_{t})z
\\&+\Big(\Phi(\delta_t^n, Y_{t}^n+c(\Xn_{t})z)-\Phi(\delta_t^n, Y_{t}^n)\Big)
-\Big(\Phi(\delta_t^n,\Xn_{t}+c(\Xn_{t})z)-\Phi(\delta_t^n, \Xn_{t})\Big)
\\&+\Phi(\delta_t^n,\Xn_{t}+c(\Xn_{t})z)-\Phi(\delta_t^n, \Xn_{t})-c(\Xn_{t})z
\\&=:\sum_{i=0}^4\Lambda_t^{c,i,n}(z),
\ea
and denote
\ba
\Lambda_t^{b,5,n}&=\frac{1}2\operatorname{tr}\Big(\Phi_{xx}(\delta_t^n, Y_t^n)bb^T(X_t^{(n)})\Big),\\
\Lambda_t^{c,5,n}(z)&=\Phi(\delta_t^n, Y_{t}^n+c(\Xn_{t})z)-\Phi(\delta_t^n, Y_{t}^n)-\Phi_x(\delta_t^n, Y_{t}^n)c(\Xn_{t})z.
\ea
With this notation, we have
\ba
\di\Delta_t^n=(A(X^n_t)-A(X_t))\,\di t
&+\sum_{i=0}^4\Lambda_t^{a,i,n}\,\di t+\Lambda_t^{b,5,n}\, \di t+\int_{\|z\|\leq 1}\Lambda_t^{c,5,n}(z)\,\nu(\di z)
\\
&+\sum_{i=0}^4\Lambda_t^{b,i,n}\,\di B_t
+\int_{\|z\|\leq 1} \sum_{i=0}^4\Lambda_{t-}^{c,i,n} (z)\widetilde{N}(\di z, \di t), \quad t\in [t_k, t_{k+1}).
\ea
ii) Semimartingale
representation of $\|\Delta^n\|^r$.
On the interval $t\in [t_k, t_{k+1})$,
we apply the It\^o formula to obtain
\ba
\label{Ito}
\di\|\Delta_t^n\|^r
&=r  \|\Delta_t^n\|^{r-2}\langle \Delta^n_t, A(X^n_t)-A(X_t) \rangle \,\di t\\
&+\sum_{i=0}^4  r \|\Delta_t^n\|^{r-2} \langle \Delta^n_t,   \Lambda_t^{a,i,n}\rangle  \,\di t\\
&+r \|\Delta_t^n\|^{r-2} \langle \Delta^n_t,   \Lambda_t^{b,5,n}\rangle  \, \di t\\
&+r \|\Delta_t^n\|^{r-2} \langle \Delta_t^n, \int_{\|z\|\leq 1}\Lambda_t^{c,5,n}(z)\,\nu(\di z) \rangle \, \di t \\
&+\frac{r(r-1)}{2} \|\Delta_t^n\|^{r-4}
\Big\| \Big(\sum_{i=0}^4\Lambda_t^{b,i,n}\Big)^T
\Delta_t^n\Big\|^2\, \di t\\
&+\frac{r}{2} \|\Delta_t^n\|^{r-2}
\Big\| \sum_{i=0}^4\Lambda_t^{b,i,n}\Big\|^2\, \di t\\
&+\int_{\|z\|\leq 1}
\Big(\Big\|\Delta^n_t +\sum_{i=0}^4\Lambda_t^{c,i,n} (z)\Big\|^r
-\|\Delta^n_t\|^r
-r\|\Delta^n_t\|^{r-2}\langle  \Delta^n_t ,  \sum_{i=0}^4\Lambda_t^{c,i,n} (z)  \rangle\Big) \,\nu(\di z)\, \di t\\
&+r \|\Delta_t^n\|^{r-2} \langle \Delta_t^n , \sum_{i=0}^4\Lambda_t^{b,i,n}\,\di  B_t \rangle \\
&+\int_{\|z\|\leq 1} \Big(\Big\|\Delta^n_{t-}+\sum_{i=0}^4\Lambda_t^{c,i,n} (z)\Big\|^r-\|\Delta^n_{t-}\|^r\Big)\,
\widetilde{N}(\di z, \di t).
\ea
Recall that $c(\cdot)$ is Lipschitz continuous and $\|z\|\leq 1$, so that
\ba
\|\Lambda^{c,0,n}_s\|\leq_C \|\Delta^n_s\|\|z\|.
\ea
Hence by the Taylor formula
\ba
 \Big\|\Delta^n_t &+\sum_{i=0}^4\Lambda_t^{c,i,n} (z)\Big\|^r
-\|\Delta^n_t\|^r
-r\|\Delta^n_t\|^{r-2}\langle  \Delta^n_t ,  \sum_{i=0}^4\Lambda_t^{c,i,n} (z)  \rangle\\
&\leq_C \Big( \|\Delta_t^n\|^{r-2}   + \Big\|\sum_{i=0}^4\Lambda_t^{c,i,n} (z)\Big\|^{r-2}  \Big)
\Big\|\sum_{i=0}^4\Lambda_t^{c,i,n} (z)\Big\|^2\\
&\leq_C  \Big( \|\Delta_t^n\|^{r-2} +\Big\|\sum_{i=1}^4\Lambda_t^{c,i,n} (z)\Big\|^{r-2}  \Big)
\Big( \|\Delta^n_s\|^2\|z\|^2      +\Big\|\sum_{i=1}^4\Lambda_t^{c,i,n} (z)\Big\|^2\Big)\\
&\leq_C
\|\Delta_t^n\|^r\|z\|^2
+ \|\Delta^n_s\|^{r-2}\sum_{i=1}^4\|\Lambda_t^{c,i,n} (z)\|^2
+ \|\Delta^n_s\|^2\|z\|^2 \sum_{i=1}^4\|\Lambda_t^{c,i,n} (z)\|^{r-2}
+\sum_{i=1}^4\|\Lambda_t^{c,i,n} (z)\|^r.
\ea
For $r=2$, the r.h.s.\ of the latter formula reduces to
\ba
\cdots\leq_C\|\Delta^n_s\|^2\|z\|^2 +\sum_{i=1}^4\|\Lambda_t^{c,i,n} (z) \|^2.
\ea
The Young inequality \eqref{e:young}  gives
\ba
\int_{\|z\|\leq 1}
\Big(\Big\| &\Delta^n_t +\sum_{i=0}^4\Lambda_t^{c,i,n} (z)\Big\|^r
-\|\Delta^n_t\|^r
-r\|\Delta^n_t\|^{r-2}\langle  \Delta^n_t ,  \sum_{i=0}^4\Lambda_t^{c,i,n} (z)  \rangle\Big) \,\nu(\di z)\\
&\leq_C
\|\Delta_t^n\|^r
+ \|\Delta^n_s\|^{r-2} \int_{\|z\|\leq 1}\sum_{i=1}^4\|\Lambda_t^{c,i,n} (z)\|^2\,\nu(\di z)\\
&+ \|\Delta^n_s\|^2  \int_{\|z\|\leq 1}  \|z\|^2 \sum_{i=1}^4\|\Lambda_t^{c,i,n} (z)\|^{r-2}\,\nu(\di z)
+ \int_{\|z\|\leq 1} \sum_{i=1}^4\|\Lambda_t^{c,i,n} (z)\|^r\,\nu(\di z)\\
&\leq_C
\|\Delta_t^n\|^r
+ \Big( \int_{\|z\|\leq 1}\sum_{i=1}^4\|\Lambda_t^{c,i,n} (z)\|^2\,\nu(\di z)\Big)^\frac{r}{2}\\
&+  \Big( \int_{\|z\|\leq 1}  \|z\|^2 \sum_{i=1}^4\|\Lambda_t^{c,i,n} (z)\|^{r-2}\,\nu(\di z)
\Big)^{\frac{r}{r-2}}
+ \int_{\|z\|\leq 1}\sum_{i=1}^4 \|\Lambda_t^{c,i,n} (z)\|^r\,\nu(\di z),
\ea
whereas for $r=2$ we just get
\ba
\cdots \leq_C \|\Delta_t^n\|^2
+ \int_{\|z\|\leq 1}\sum_{i=1}^4 \|\Lambda_t^{c,i,n} (z)\|^2\,\nu(\di z).
\ea
This gives the representation
\begin{equation}
\label{Ito2}
\|\Delta_t^n\|^r=\|\Delta_{t_k}^n\|^r + \int_{t_k}^t \rmD_s^{n,k}\, \di s+ \rmM_t^{n,k}, \quad t\in [t_k, t_{k+1}),
\end{equation}
with a local martingale $\rmM^{n,k}$, $\rmM^{n,k}_{t_{k}}=0$, and a drift term $\rmD^{n,k}$ that satisfies
\ba
\label{estG}
\rmD_t^{n,k} & \leq_C  \|\Delta_t^n\|^r + \rmF_t^{n,k},\\
\rmF_t^{n,k}
&:=\sum_{i=1}^4 \|\Lambda_t^{a,i,n}\|^r
+\sum_{i=1}^5 \|\Lambda_t^{b,i,n}\|^r+\sum_{i=1}^5 \int_{\|z\|\leq 1}\|\Lambda_t^{c,i,n}(z)\|^r\,\nu(\di z)\\
&+\sum_{i=1}^4 \Big( \int_{\|z\|\leq 1} \|\Lambda_t^{c,i,n} (z) \|^2\,\nu(\di z)\Big)^\frac{r}{2}
+ \sum_{i=1}^4 \Big( \int_{\|z\|\leq 1}  \|z\|^2\|\Lambda_t^{c,i,n} (z)\|^{r-2}\,\nu(\di z) \Big)^{\frac{r}{r-2}},
\ea
where the last line vanishes for $r=2$.

Then by the standard argument, involving localization, the Gronwall lemma, and the Fatou lemma, we get that,
for $t\in [t_k, t_{k+1})$,
\ba
\label{estDelta}
\E^{(N)}_{t_1,\dots,t_N}\Big[\|\Delta_t^n\|^r\Big|\mathcal{F}_{t_k}\Big]
\leq_C \|\Delta_{t_k}^n\|^r+ \int_{t_k}^{t }\E^{(N)}_{t_1,\dots,t_N}\Big[ \rmF_s^{n,k} \Big|\mathcal{F}_{t_k}\Big]\,\di s.
\ea
iii)
Estimates of the $\Lambda$-terms in \eqref{estG}.

\noindent
1. Since $a,b,c$ are Lipschitz continuous, we have by \eqref{estimates} that
\ba
\|\Lambda_s^{a,1,n}\|+\|\Lambda_s^{b,1,n}\|  & \leq_C  U_s^n, \\
\|\Lambda_s^{c,1,n}(z)\| & \leq_C  U_s^n\|z\|.
\ea
2. By \eqref{estimates1} we get
\ba
\|\Lambda_s^{a,2,n}\|+\|\Lambda_s^{b,2,n}\| &\leq_C  h_n^{\gamma }(1+\|X^{(n)}_s\|^{\gamma\chi+1}), \\
\|\Lambda_s^{c,2,n}(z)\| & \lc h_n^{\gamma }(1+\|X^{(n)}_s\|^{\gamma\chi+1})\|z\|.
\ea
3. By \eqref{estimates2} we have
\ba
\|\Lambda_s^{a,3,n}\|+\|\Lambda_s^{b,3,n}\|&\leq_C  U_s^n.
\ea
Since $c(\cdot)$ is bounded, $\|z\|\leq 1$, and $\|\Phi_{xx}\|$ is bounded by \eqref{e:Phixx} we get
\ba
\Lambda_s^{c,3,n}(z)&=\Big(\Phi(\delta_s^n, Y_{s}^n+c(\Xn_{s})z)-\Phi(\delta_s^n, Y_{s}^n)\Big)
-\Big(\Phi(\delta_s^n,\Xn_{s}+c(\Xn_{s})z)-\Phi(\delta_s^n, \Xn_{s})\Big)\\
&=\int_0^1
\Big( \Phi_x(\delta_s^n, Y_{s}^n+\theta c(\Xn_{s})z)
-\Phi_x(\delta_s^n, \Xn_{s} +\theta c(\Xn_{s})z) \Big)c(\Xn_{s}) z\,\di \theta\\
&=
\int_0^1\int_0^1
\partial_x \Big(\Phi_{x}(\delta_s^n, \zeta Y_{s}^n+(1-\zeta)\Xn_{s}+\theta c(\Xn_{s})z)     c(\Xn_{s})z\Big) (Y_{s}^n-\Xn_{s})\,\di \zeta\, \di \theta,
\ea
which implies that
\ba
\|\Lambda_s^{c,3,n}(z)\|\leq_C  \|Y_{s}^n-\Xn_{s}\|\|z\| .
\ea
Hence,
\ba
\|\Lambda_s^{c,3,n}(z)\|\leq_C U_{s}^n \|z\| .
\ea
4. By \eqref{e:Phih}, we have

\ba
\|\Lambda_s^{a,4,n}\|+\|\Lambda_s^{b,4,n}\|\leq_C h_n^{\gamma }(1+\|X^{(n)}_s\|^{\gamma\chi+1}).
\ea
Furthermore,
\ba
\Lambda_s^{c,4,n}(z)
&=\Big(\Phi(\delta_s^n,\Xn_{s}+c(\Xn_{s})z)-\Xn_{s}-c(\Xn_{s})z\Big)-\Big(\Phi(\delta_s^n, \Xn_{s})-\Xn_{s}\Big),
\ea
hence by \eqref{e:Phih} we have
\ba
\|\Lambda_s^{c,4,n}(z)\|\lc h_n^{\gamma }(1+\|X^{(n)}_{s}\|^{\gamma\chi+1})\|z\|.
\ea
Here we have used that
\ba
(1+\|\Xn_{s}+c(\Xn_{s})z\|^{\gamma\chi+1})\lc (1+\|X^{(n)}_{s}\|^{\gamma\chi+1}),
\ea
because $c(\cdot)$ and $\|z\|$ are bounded.

\noindent
5. Finally, by \eqref{e:Phij} we have
\ba
\|\Lambda_s^{b,5,n}\|\lc h_n^{\gamma}(1+\|X^{(n)}_s\|^{\gamma\chi})+h_n^{\gamma}(U_{s}^n)^{\gamma\chi}
\ea
and
\ba
\|\Lambda_s^{c,5,n}(z)\|
&\leq \frac{1}{2}\sup_{\theta\in[0,1]}\|\Phi_{xx}(\delta_s^n, Y_{s}^n+\theta c(\Xn_{s})z)\|\|c(\Xn_{s})z\|^2 \\
&\lc h_n^{\gamma}(1+\|X^{(n)}_s\|^{\gamma\chi})\|z\|^{2}+h_n^{\gamma}(U_{s}^n)^{\gamma\chi}\|z\|^{2}.
\ea
Summarizing the above estimates and taking into account that $\|z\|^{2}\leq \|z\|$ for $\|z\|\leq 1$, we get
\ba
\label{Lambda1}
\Lambda_s^{a,n}&:= \sum_{i=1}^4 \|\Lambda_s^{a,i,n}\|
\leq_C h_n^{\gamma }(1+\|X^{(n)}_{s}\|^{\gamma\chi+1})+U_{s}^n+h_n^{\gamma}(U_s^n)^{\gamma\chi},\\
\Lambda_s^{b,n}&:= \sum_{i=1}^5 \|\Lambda_s^{b,i,n}\|
\leq_C  h_n^{\gamma }(1+\|X^{(n)}_{s}\|^{\gamma\chi+1})+U_{s}^n+h_n^{\gamma}(U_s^n)^{\gamma\chi},\\
\Lambda_s^{c,n}(z)&:= \sum_{i=1}^5\|\Lambda_s^{c,i,n}(z)\|
\leq_C \Big(h_n^{\gamma }(1+\|X^{(n)}_{s}\|^{\gamma\chi+1})+U_{s}^n+h_n^{\gamma}(U_s^n)^{\gamma\chi}\Big)\|z\|.
\ea
Alternatively, using that coefficients $a,b,c$ and the derivarives $\Phi_x$ and $\Phi_{xx}$ are bounded
(see \eqref{e:Phix01} and \eqref{e:Phixx}) we immediately get that
\ba
\label{Lambda2}
\Lambda_s^{a,n}&\leq_C 1,\quad \Lambda_s^{b,n}\leq_C 1,\quad \Lambda_s^{c,n}(z)\leq_C \|z\|.
\ea
iv) Estimates of the terms in
\eqref{Lambda1}.
Denote $t_k^{n,*}=\eta^n_{t_k}+h_n$
and observe that for $t\in [t_k^{n,*}, t_{k+1})$,
\ba
\mathrm{i)}
&\quad\widehat Z^n_t:=\widehat{Z}_t-\widehat{Z}_{\eta^n_{t}} \text{ and } B_t^n\text{ are independent of }\mathcal{F}_{t_k},\\
\mathrm{ii)}
&\quad   \Xn_t=X^n_{\eta^n_t}\text{ with }\eta^n_t\in [t_k^{n,*}, t_{k+1}).
\ea
On the other hand, for $t\in [t_k,t_k^{n,*})$
\ba
\mathrm{i)}
&\quad \widehat Z^n_t:=\widehat{Z}_t- \widehat{Z}_{t_k} + \widehat{Z}_{t_k} -\widehat{Z}_{\eta^n_{t_k}} \text{ and }
B^n_t:=B_t- B_{t_k} + B_{t_k} -B_{\eta^n_{t_k}},\\
&\quad \widehat{Z}_t- \widehat{Z}_{t_k},\ B_t- B_{t_k} \text{ are independent of }\mathcal{F}_{t_k},\\
&\quad \widehat{Z}_{t_k} -\widehat{Z}_{\eta^n_{t_k}},\ B_{t_k} -B_{\eta^n_{t_k}}
\text{ are measurable w.r.t.\ }\mathcal{F}_{t_k},\\
\mathrm{ii)}
&\quad   \Xn_t= \Xn_{t_k}=X^n_{\eta^n_{t_k}}.
\ea
For any $q>0$ and any $s,t\in[0,h_n]$ we have
\ba
\label{e:momentsNoise}
\E^{(N)}_{t_1,\dots,t_N}\|B_t-B_s\|^q\lc h_n^{\frac{q}{2}},
\quad \E^{(N)}_{t_1,\dots,t_N}\|\widehat Z^n_t - \widehat Z^n_t\|^q\lc h_n^{\frac{q}{2}}+ h_n.
\ea
Consequently, for any $q\in(0,\infty)$ we estimate:
\ba
\E^{(N)}_{t_1,\dots,t_N}[(U_t^n)^q\Big| \cF_{t_k}]&\leq_C \begin{cases}
                                    (U_{t_k}^n)^q + h_n^{\frac{q}{2}}+ h_n,\quad &t\in [t_k,t_k^{n,*}),\\
                                    h_n^{\frac{q}{2}}+ h_n,\quad &t\in [t_k^{n,*},t_{k+1}).
                                   \end{cases},\\
\E^{(N)}_{t_1,\dots,t_N}[\|X_t^{(n)}\|^q\Big| \cF_{t_k}]&\leq_C \begin{cases}
                                    \|X_{t_k}^{(n)}\|^q,\quad &t\in [t_k,t_k^{n,*}),\\
                                     \sup_{t\in[t_k,t_{k+1})} \E^{(N)}_{t_1,\dots,t_N} \|X_t^{n}\|^q
                                     \leq_C 1+ \|X^n_{t_k}\|^q
                                     ,\quad &t\in [t_k^{n,*},t_{k+1}).
                                   \end{cases},\\
\ea
where in the latter inequality we used
Lemma \ref{l:estXn}.

Furthermore, we estimate the integrals
\ba
\int_{t_k}^{t_{k+1}} \E^{(N)}_{t_1,\dots,t_N}\Big[(U_t^n)^q\Big| \cF_{t_k}\Big]  \,\di s
&=\Big(\int_{t_k}^{t_k^{n,*}} +
\int_{t_k^{n,*}}^{t_{k+1}}\Big) \E^{(N)}_{t_1,\dots,t_N}\Big[(U_t^n)^q\Big| \cF_{t_k}\Big]  \,\di s\\
&\leq_C  h_n \Big((U_{t_k}^n)^q + h_n^{\frac{q}{2}}+ h_n\Big) + h_n^{\frac{q}{2}}+ h_n\\
&\leq_C  h_n  (U_{t_k}^n)^q + h_n^{\frac{q}{2}}+ h_n.
\ea
Analogously,
\ba
\int_{t_k}^{t_{k+1}} \E^{(N)}_{t_1,\dots,t_N}\Big[\|X_t^{(n)}\|^q\Big| \cF_{t_k}\Big]  \,\di s
&=\Big(\int_{t_k}^{t_k^{n,*}} +
\int_{t_k^{n,*}}^{t_{k+1}}\Big) \E^{(N)}_{t_1,\dots,t_N}\Big[\|X_t^{(n)}\|^q\Big| \cF_{t_k}\Big]  \,\di s\\
&\leq_C  h_n \|X_{t_k}^{(n)}\|^q  +  1+ \|X^n_{t_k}\|^q.
\ea
v) Final estimates.
Now we can complete the proof. We rewrite \eqref{estDelta} as
\ba
\label{e:I*}
\E^{(N)}_{t_1,\dots,t_N}\Big[\|\Delta_t^n\|^r\Big|\mathcal{F}_{t_k}\Big]
\leq_C \|\Delta_{t_k}^n\|^r
&+ \int_{t_k}^{t_{k+1}}
\E^{(N)}_{t_1,\dots,t_N}\Big[
\Big[(\Lambda_s^{a,n})^r\Big|\mathcal{F}_{t_k}\Big]\,\di s\\
&+ \int_{t_k}^{t_{k+1} } \E^{(N)}_{t_1,\dots,t_N}\Big[(\Lambda_s^{b,n})^r\Big|\mathcal{F}_{t_k}\Big]\,\di s\\
&+\int_{t_k}^{t_{k+1} } \E^{(N)}_{t_1,\dots,t_N}\Big[ \int_{\|z\|\leq 1}(\Lambda_s^{c,n}(z))^r\,\nu(\di z)\Big|\mathcal{F}_{t_k}\Big]\,\di s\\
&+\sum_{i=1}^4 \E^{(N)}_{t_1,\dots,t_N}\Big[\Big( \int_{\|z\|\leq 1} \|\Lambda_t^{c,i,n} (z) \|^2\,\nu(\di z)\Big)^\frac{r}{2}\Big|\mathcal{F}_{t_k}\Big]\,\di s\\
&+ \sum_{i=1}^4 \E^{(N)}_{t_1,\dots,t_N}\Big[\Big( \int_{\|z\|\leq 1}  \|z\|^2\|\Lambda_t^{c,i,n} (z)\|^{r-2}\,\nu(\di z) \Big)^{\frac{r}{r-2}}\Big|\mathcal{F}_{t_k}\Big]\,\di s\\
&\hspace{-3cm}\leq_C
 \|\Delta_{t_k}^n\|^r + \int_{t_k}^{t_{k+1} }
\E^{(N)}_{t_1,\dots,t_N}\Big[ h_n^{r\gamma }(1+\|X^{(n)}_{s}\|^{r\gamma\chi+r})+(U_{s}^n)^r+
h_n^{r\gamma}(U_s^n)^{r\gamma\chi}  \Big|\mathcal{F}_{t_k}\Big]\,\di s.
\ea
Eventually, this results in the desired estimate:
\ba\label{e:*I}
\E^{(N)}_{t_1,\dots,t_N}\Big[\|\Delta_t^n\|^r\Big|\mathcal{F}_{t_k}\Big]
\leq_C \|\Delta_{t_k}^n\|^r
&+
h_n^{r\gamma }(1+h_n \|X_{t_k}^{(n)}\|^{r\gamma\chi+r}  +  1+ \|X^n_{t_k}\|^{r\gamma\chi+r})\\
&+h_n  (U_{t_k}^n)^r + h_n^{\frac{r}{2}}+ h_n
+ h_n^{r\gamma}\Big(h_n  (U_{t_k}^n)^{r\gamma\chi} + h_n^{\frac{r\gamma\chi}{2}}+ h_n \Big)\\
&\leq_C \|\Delta_{t_k}^n\|^r
+h_n^{r\gamma +1}(1+\|X_{t_k}^{(n)}\|^{r\gamma\chi+r}) +h_n^{r\gamma }(1 + \|X^n_{t_k}\|^{r\gamma\chi+r})\\
&+h_n  (U_{t_k}^n)^r
+ h_n^{r\gamma+1} (U_{t_k}^n)^{r\gamma\chi}  + h_n.
\ea
\end{proof}

As a corollary, we get the following estimate for the moments of $\Delta^n$ of the order $q<p$, where
$p\in(0,\infty)$ denotes the exponent from the condition $\mathbf{H}_{\nu,p}$.

\begin{cor}
\label{cA3} For any $q\in (0,p)$
\ba
\sup_{t\in[t_k, t_{k+1})} \E^{(N)}_{t_1,\dots,t_N}\Big[\|\Delta_t^n\|^q\Big|\mathcal{F}_{t_k}\Big]
\leq_C \|\Delta_{t_k}^n\|^q
& + h_n^{(\frac{p-q}{\chi}\wedge q) + (1\wedge\frac{q}{2})}(1+\|X_{t_k}^{(n)}\|^p)
+ h_n^{\frac{p-q}{\chi}\wedge q}(1+\|X_{t_k}^n\|^{p})\\
&+ h_n^{1\wedge\frac{q}{2}}  (U_{t_k}^n)^q
+h_n^{(\frac{p-q}{\chi }\wedge q)   +(1\wedge\frac{q}{2})} (U_{t_k}^n)^{(p-q)\wedge (q\chi)}
+h_n^{1\wedge\frac{q}{2}}
\ea
for $n$ large enough uniformly over
$N\in\bN_0$, partitions $\{t_1,\dots, t_N\}$,  and $k=0, \dots, N$.
\end{cor}
\begin{proof}
If $q\in [2,\infty)$, we use Lemma \ref{lA1Delta} with $r=q$ to get
\ba
\sup_{t\in[t_k, t_{k+1})}\E^{(N)}_{t_1,\dots,t_N}\Big[\|\Delta_t^n\|^q\Big|\mathcal{F}_{t_k}\Big]
&\leq_C  \|\Delta_{t_k}^n\|^q
+h_n^{q\gamma +1}(1+\|X_{t_k}^{(n)}\|^{q\gamma\chi+q}) +h_n^{q\gamma}(1+\|X_{t_k}^n\|^{q\gamma\chi+q})\\
&+h_n  (U_{t_k}^n)^q
+ h_n^{q\gamma+1} (U_{t_k}^n)^{q\gamma\chi}  + h_n.
\ea
If $q\in (0,2)$ we  use Lemma \ref{lA1Delta} with $r=2$ and the H\"older inequality to get
\ba
\sup_{t\in[t_k, t_{k+1})}\E^{(N)}_{t_1,\dots,t_N}\Big[\|\Delta_t^n\|^q\Big|\mathcal{F}_{t_k}\Big]
&\leq
\sup_{t\in[t_k, t_{k+1})}\Big(\E^{(N)}_{t_1,\dots,t_N}\Big[\|\Delta_t^n\|^2\Big|\mathcal{F}_{t_k}\Big]\Big)^\frac{q}{2}\\
&\leq_C
\|\Delta_{t_k}^n\|^q
+h_n^{q\gamma +\frac{q}{2}}(1+\|X_{t_k}^{(n)}\|^{q\gamma\chi+q}) +h_n^{q\gamma }(1 + \|X^n_{t_k}\|^{q\gamma\chi+q})\\
&+h_n^{\frac{q}{2}}  (U_{t_k}^n)^q
+ h_n^{q\gamma+\frac{q}{2}} (U_{t_k}^n)^{q\gamma\chi}  + h_n^{\frac{q}{2}}.
\ea
Thus, in any case,
\ba\label{e:any case}
\sup_{t\in[t_k, t_{k+1})}\E^{(N)}_{t_1,\dots,t_N}\Big[\|\Delta_t^n\|^q\Big|\mathcal{F}_{t_k}\Big]
&\leq_C  \|\Delta_{t_k}^n\|^q
+h_n^{q\gamma +(1\wedge\frac{q}{2})}(1+\|X_{t_k}^{(n)}\|^{q\gamma\chi+q}) +h_n^{q\gamma}(1+\|X_{t_k}^n\|^{q\gamma\chi+q})\\
&+h_n  (U_{t_k}^n)^q
+ h_n^{q\gamma+(1\wedge\frac{q}{2})} (U_{t_k}^n)^{q\gamma\chi}  + h_n.
\ea
Taking
\ba
\gamma=\frac{p-q}{\chi q}\wedge 1,
\ea
we get
\ba
q\gamma\chi+q\leq p, \quad q\gamma=\frac{p-q}{\chi}\wedge q.
\ea
Then
\ba
1+\|X_{t_k}^n\|^{q\gamma\chi+q}\leq_C 1+\|X_{t_k}^n\|^p,
\ea
and \eqref{e:any case} yields the required estimate.
\end{proof}

The next step incorporates into the estimate the large jump at time $t=t_k$.
\begin{lem}
\label{lA4cor} For any $q\in(0,p)$
\ba
\sup_{t\in[t_k, t_{k+1})}\E^{(N)}_{t_1,\dots,t_N}\Big[\|\Delta_t^n\|^q\Big|\mathcal{F}_{t_k-}\Big]
&\leq_C \|\Delta_{t_k-}^n\|^q\\
&+h_n^{\frac{p-q}{\chi}\wedge q}(1+\|X_{t_k-}^{(n)}\|^p)
+ h_n^{\frac{p-q}{\chi}\wedge q}(1+\|X_{t_k-}^n\|^{p})\\
&+ (U^n_{t_k-})^q
+h_n^{1\wedge \frac{q}2}(U_{t_k-}^n)^{(p-q)\wedge (q\chi)}+h_n^{1\wedge \frac{q}2}
\ea
for $n$ large enough uniformly over
$N\in\bN$, partitions $\{t_1,\dots, t_N\}$,  and $k=1, \dots, N$.
\end{lem}
\begin{rem} The  above estimate still holds for $N=0$ and/or $k=0$, if we adopt the convention
\ba
\mathcal{F}_{0-}=\mathcal{F}_{0},\quad \Delta_{0-}^n=0, \quad X_{0-}^{(n)}=X_{0-}^{n}=x, \quad U^n_{0-}=0.
\ea
In this case there is no actual jump at the time moment $t_0=0$, and the required estimate is provided by Corollary \ref{cA3}.
\end{rem}
\begin{proof}
We apply Corollary \ref{cA3} to get
\ba
\sup_{t\in[t_k, t_{k+1})}\E^{(N)}_{t_1,\dots,t_N}&\Big[\|\Delta_t^n\|^q\Big|\mathcal{F}_{t_k-}\Big]
\leq \sup_{t\in[t_k, t_{k+1})}\E^{(N)}_{t_1,\dots,t_N}
\Big[\E^{(N)}_{t_1,\dots,t_N}\Big[\|\Delta_t^n\|^q\Big|\mathcal{F}_{t_k}\Big]\Big|\mathcal{F}_{t_k-}\Big]\\
&\leq_C
\|\Delta_{t_k-}^n\|^q
+ \E^{(N)}_{t_1,\dots,t_N}\Big[
\|\Delta_{t_k}^n- \Delta_{t_k-}^n\|^q\Big|\mathcal{F}_{t_k-}\Big]\\
&+\E^{(N)}_{t_1,\dots,t_N}\Big[ h_n^{(\frac{p-q}{\chi}\wedge q)+(1\wedge\frac{q}{2})}(1+\|X_{t_k}^{(n)}\|^p)
+ h_n^{\frac{p-q}{\chi}\wedge q}(1+\|X_{t_k}^n\|^{p})\Big|\mathcal{F}_{t_k-} \Big]\\
&+\E^{(N)}_{t_1,\dots,t_N}\Big[  h_n^{1\wedge\frac{q}{2}}  (U_{t_k}^n)^q
+h_n^{(\frac{p-q}{\chi }\wedge q)   +(1\wedge\frac{q}{2})} (U_{t_k}^n)^{(p-q)\wedge (q\chi)}
+h_n^{1\wedge\frac{q}{2}} \Big|\mathcal{F}_{t_k-} \Big].
\ea
a) We have
\ba
X_{t_k}^n=X_{t_k-}^n+\Phi(\delta_{t_k}^n, Y_{t_k-}^n+c(\Xn_{t_k})J_k)-\Phi(\delta_{t_k}^n, Y^n_{t_k-}),
\ea
thus with the help of \eqref{e:Phix01} we get
\begin{equation}
\label{e:XJ}
1+\|X_{t_k}^n\|^{p}\leq_C 1+\|X_{t_k-}^n\|^{p}+\|J_k\|^p\leq ( 1+\|X_{t_k-}^n\|^{p})(1+\|J_k\|^p).
\end{equation}
b) Without loss of generality we can assume that $t_k\notin \{t^n_j\}$, so that
\ba
X_{t_k}^{(n)}=X_{t_k-}^{(n)}=X^n_{\eta^n_{t_k}}.
\ea
c)  Next,
\ba
\Delta_{t_k}^n-\Delta_{t_k-}^n
&=\Big(\Phi(\delta_{t_k}^n, Y_{t_k-}^n+c(\Xn_{t_k-})J_k)-\Phi(\delta_{t_k}^n, Y^n_{t_k-})\Big)-c(X_{t_k-})J_k\\
&=c(X_{t_k-}^n)J_k-c(X_{t_k-})J_k\\
&+c(\XN_{t_k-})J_k-c(X_{t_k-}^n)J_k\\
&+c(\Xn_{t_k-})J_k-c(\XN_{t_k-})J_k\\
&+\Phi(\delta_{t_k}^n, Y^n_{t_k-}+c(\Xn_{t_k-})J_k)-\Phi(\delta_{t_k}^n, \Xn_{t_k-}+c(\Xn_{t_k-})J_k)\\
&+\Phi(\delta_{t_k}^n, \Xn_{t_k-}+c(\Xn_{t_k-})J_k)-\Phi(\delta_{t_k}^n, \Xn_{t_k-})-c(\Xn_{t_k-})J_k\\
&+\Phi(\delta_{t_k}^n, \Xn_{t_k-})-\Phi(\delta_{t_k}^n, Y^n_{t_k-})\\
&=:\sum_{i=0}^{5}\Lambda^{d,i,n}_k .
\ea
For $i=0,1,2$ the estimates for $\Lambda^{d,i,n}$ are analogous to those for $\Lambda^{c,i,n}$ from Lemma \ref{lA1Delta}:
\ba
\|\Lambda^{d,0,n}_k\|\leq_C\|\Delta_{t_k-}^n\|\|J_k\|, \quad
\|\Lambda^{d,1,n}_k\|\leq_CU^n_{t_k-}\|J_k\|, \quad
\|\Lambda^{d,2,n}_k\|\leq_C h_n^{\gamma}(1+\|\Xn_{t_k-}\|^{\gamma\chi+1})\|J_k\|.
\ea
To estimate $\Lambda^{d,3,n},$ we observe that,
by \eqref{e:Phix01}, \eqref{e:Phix-x01} and \eqref{e:Phih}, for any $t\in[0,1]$ and $x,v\in\bR^d$
\ba
\|\Phi(t, x+v) - \Phi(t,x) - v \|&\leq \|v\|\sup_{\theta\in [0,1]}\|\Phi_x(t,x+\theta v) \|\\
&\leq_C t^{\gamma}(1+\|x\|^{\gamma\chi}+\|v\|^{\gamma\chi})\|v\|\\
&\leq_C t^{\gamma}(1+\|x\|^{\gamma\chi})(1+\|v\|^{\gamma\chi+1}).
\ea
Therefore
\ba
\|\Lambda^{d,3,n}_k\|\leq_C h_n^{\gamma}(1+\|\Xn_{t_k-}\|^{\gamma\chi})(1+\|J_k\|^{\gamma\chi+1}).
\ea
Finally, by \eqref{e:Phix01} we have
\ba
  \|\Lambda^{d,4,n}_k\|+\|\Lambda^{d,5,n}_k\|\leq_C \|Y^n_{t_k-}-\Xn_{t_k-}\|\leq_C U^n_{t_k-}.
\ea
Collecting the above estimates, we get
\ba
\|\Delta_{t_k}^n-\Delta_{t_k-}^n\|^q
\lc
\|\Delta_{t_k-}^n\|^q\|J_k\|^q
+ (U^n_{t_k-})^q \|J_k\|^q
+h_n^{\gamma q}(1+\|\Xn_{t_k-}\|^{\gamma\chi q+q})\|J_k\|^{\gamma\chi q +q}
+ (U^n_{t_k-})^q.
\ea
Taking
\ba
\gamma=\frac{p-q}{q\chi}\wedge 1
\ea
we get
\ba
\|\Delta_{t_k}^n-\Delta_{t_k-}^n\|^q
\lc \|\Delta_{t_k-}^n\|^q\|J_k\|^q+ (U^n_{t_k-})^q(1+\|J_k\|^q)+
   h_n^{\frac{p-q}{\chi}\wedge q}(1+\|\Xn_{t_k-}\|^{p})(1+\|J_k\|^{p}).
\ea
Recall that $J_k$ are identically distributed, independent of $\mathcal{F}_{t_k-}$
for each $k\in\bN$, and have finite $p$-th moment, hence
\ba
\E^{(N)}_{t_1,\dots,t_N}\Big[\|\Delta_{t_k}^n-\Delta_{t_k-}^n\|^q\Big|\mathcal{F}_{t_k-}\Big] \lc \|\Delta_{t_k-}^n\|^q+ (U^n_{t_k-})^q
+   h_n^{\frac{p-q}{\chi}\wedge q}(1+\|\Xn_{t_k-}\|^{p}).
\ea
d)
Finally, we note that
\ba
|U^n_{t_k}-U^n_{t_k-}|\leq_C 1,
\ea
and therefore
\ba
\E^{(N)}_{t_1,\dots,t_N}  \Big[  h_n^{1\wedge\frac{q}{2}}  (U_{t_k}^n)^q
&+h_n^{(\frac{p-q}{\chi }\wedge q)   +(1\wedge\frac{q}{2})} (U_{t_k}^n)^{(p-q)\wedge (q\chi)}
\Big|\mathcal{F}_{t_k-} \Big]\\
&\leq h_n^{1\wedge \frac{q}2}(U^n_{t_k-})^q
+h_n^{1\wedge \frac{q}2}(U_{t_k-}^n)^{(p-q)\wedge (q\chi)}+ h_n^{1\wedge\frac{q}{2}}.
\ea
Summarizing the above estimates, we get  the required inequality.
\end{proof}

Recall that we have assumed the number $N$ of large jumps and their time instants $t_k$, $k=1, \dots, N$
to be fixed. In what follows, we will study separately two scenarios.
Let
\ba
\label{typical}
C_N^n=\Big\{(t_1,\dots, t_N)\colon  t_{k}-t_{k-1}> h_n, \quad k=2,\dots, N\Big\}
\ea
be the ``typical'' set of well separated large jump times, and let
$D_N^n=(C_N^n)^c$ be
its complement on which
two large jumps can occur within one $h_n$-step of the numerical scheme.
\begin{lem}
\label{cA5}
For any $q\in(0,p)$,
there exists $C\in(0,\infty)$ such that
\ba
\label{e:estC}
 \sup_{t\in[0, T]}\E^{(N)}_{t_1,\dots,t_N}\Big[\|\Delta_t^n\|^q\Big]
\leq C^{N+1}h_n^{\frac{p-q}{\chi}\wedge q}(1+2\|x\|^p) +(2N+1) C^{N+1}h_n^{\frac{q}{2}\wedge 1}
\ea
for $n$ large enough uniformly over
$N\in\bN_0$ and $N$-tuples $\{t_1,\dots, t_N\}\in C^n_N$.
\end{lem}
\begin{proof}
For $N\in\bN_0$, denote
\ba
L_k^n&=\|\Delta_{t_k-}^n\|^q, \quad k=1, \dots, N+1,\quad && L_0^n=0,\\
u_k^n&=(U_{t_{k}-}^n)^q+h_n^{1\wedge \frac{q}2}(U_{t_k-}^n)^{(p-q)\wedge (q\chi)}, \quad k=1, \dots, N+1,\quad &&u_0^n=0,\\
\Xi_k^n&=h_n^{\frac{p-q}{\chi}\wedge q}(1+\|X^n_{t_{k}-}\|^p + \|X^{(n)}_{t_{k}-}\|^p),
\quad k=1, \dots, N+1,\quad &&\Xi_0^n=h_n^{\frac{p-q}{\chi}\wedge q}(1+2\|x\|^p).
\ea
Since we aim to establish the estimate \eqref{e:estC} for all $N\in\mathbb N_0$ with the same constant $C$, we henceforth abandon the simplified notation $\leq_C$ and track significant constants explicitly.

\noindent
i) By Lemma \ref{lA4cor} and the Fatou lemma, there is $C_1\in(0,\infty)$ such that, for $k=0, \dots, N$
\ba\label{e:est-norm}
\E^{(N)}_{t_1,\dots,t_N}\Big[L^n_{k+1}\Big|\mathcal{F}_{t_k-}\Big]&\leq \sup_{t\in [t_k, t_{k+1})} \E^{(N)}_{t_1,\dots,t_N}\Big[\|\Delta_{t}^n\|^q\Big|\mathcal{F}_{t_k-}\Big]
\\&\leq C_1\Big( \|\Delta_{t_k-}^n\|^q+ h_n^{\frac{p-q}{\chi}\wedge q}(1+\|X_{t_k-}^n\|^{p}+\|X_{t_k-}^{(n)}\|^{p})
\\&\hspace*{2cm}+(U^n_{t_k-})^q+h_n^{1\wedge \frac{q}2}(U_{t_k-}^n)^{(p-q)\wedge (q\chi)}+h_n^{\frac{q}2\wedge 1}\Big)\\
&\leq C_1 L_k^n + C_1 \Xi_k^n + C_1 u_k^n +C_1 h_n^{\frac{q}2\wedge 1}.
\ea
ii) By Lemma \ref{l:estXn},
the Fatou lemma, and \eqref{e:XJ}
\ba
\E^{(N)}_{t_1,\dots,t_N}\Big[\|X^n_{t_{k+1}-}\|^p\Big|\mathcal{F}_{t_k-}\Big]
&\leq
 \sup_{t\in[t_k, t_{k+1})}\E^{(N)}_{t_1,\dots,t_N}\Big[\|X^n_{t}\|^p\Big|\mathcal{F}_{t_k-}\Big]
\\&\leq
 \E^{(N)}_{t_1,\dots,t_N}\Big[\sup_{t\in[t_k, t_{k+1})} \E^{(N)}_{t_1,\dots,t_N}
\Big[\|X^n_{t }\|^p\Big|\mathcal{F}_{t_k}\Big]\Big|\mathcal{F}_{t_k-}\Big]
\\
&\leq_C  \E^{(N)}_{t_1,\dots,t_N}\Big[1+ \|X^n_{t_{k}}\|^p\Big|\mathcal{F}_{t_k-}\Big]\\
&\leq_C  1+ \|X^n_{t_{k}-}\|^p.
\ea
In addition, $X^{(n)}_{t_{k+1}-}=X^{n}_{\eta^n_{t_{k+1}}}$ with $\eta^n_{t_{k+1}}\in [t_k, t_{k+1})$, so that analogously
\ba\label{e:cor1}
\E^{(N)}_{t_1,\dots,t_N}\Big[ \|X^{(n)}_{t_{k+1}-}\|^p\Big|\mathcal{F}_{t_k-}\Big]
&
 \leq_C  \E^{(N)}_{t_1,\dots,t_N}\Big[1+ \|X^n_{t_{k}}\|^p\Big|\mathcal{F}_{t_k-}\Big]\\
&\leq_C  1+ \|X^n_{t_{k}-}\|^p.
\ea
Therefore there is a constant $C_2\in(0,\infty)$ such that
\ba\label{e:cor2}
\E^{(N)}_{t_1,\dots,t_N}\Big[\Xi^n_{k+1}\Big|\mathcal{F}_{t_k-}\Big]\leq C_2 \Xi^n_{k}.
\ea
For $k=0$ this estimate follows directly from \eqref{e:estXn}.

\noindent
iii) Finally, under \eqref{typical} we have
\ba
\label{e:ueq}
u^n_{k+1}&=(h_n+\|B_{t_{k+1}}-B_{\eta^n_{t_{k+1}}}\|+\|\widehat Z_{t_{k+1}-}-\widehat Z_{\eta^n_{t_{k+1}}}\|)^q
\\&\hspace*{2cm}+h_n^{1\wedge \frac{q}2}(h_n+\|B_{t_{k+1}}-B_{\eta^n_{t_{k+1}}}\|+\|\widehat Z_{t_{k+1}-}-\widehat Z_{\eta^n_{t_{k+1}}}\|)^{(p-q)\wedge (q\chi)},
\ea
which is independent on $\mathcal{F}_{t_{k}-}$ because $\eta^n_{t_{k+1}}>t_k$. Hence by \eqref{e:momentsNoise}
there exists $C_3\in(0,\infty)$ such that
\begin{equation}
\label{e:uest}
\E^{(N)}_{t_1,\dots,t_N}\Big[ u_{k+1}\Big|\mathcal{F}_{t_k-}\Big]\leq C_3 h_n^{\frac{q}{2}\wedge 1}, \quad k=0, \dots, N.
\end{equation}
We can summarize these estimates as follows. Let $C_4=\max\{C_1,C_2, C_3\}$, then
\ba
\E^{(N)}_{t_1,\dots,t_N}\left[\left.
\begin{pmatrix}
    L_{k+1}^n \\
    \Xi_{k+1}^n \\
    u_{k+1}^n \\
\end{pmatrix}
\right|\mathcal{F}_{t_k-}\right]
\leq C_4
\begin{pmatrix}
                   1 & 1 & 1 \\
                   0 & 1 & 0 \\
                   0 & 0 & 0 \\
\end{pmatrix}
\begin{pmatrix}
   L_{k}^n \\
    \Xi_{k}^n \\
    u_{k}^n \\
\end{pmatrix}
+C_4 h_n^{\frac{q}{2}\wedge 1}
\begin{pmatrix}
1 \\
0 \\
1 \\
\end{pmatrix} , \quad k=0, \dots, N.
\ea
Denote for a vector $v=(v^1, v^2, v^3)^T\in \mathbb{R}^3$ and a matrix $Q\in \mathbb{R}^{3\times3}$
\ba
\|v\|_1:=|v_1|+|v_2|+|v_3|, \quad \|Q\|_1:=\sup_{\|v\|_1=1}\|Qv\|_1.
\ea
Then it is a simple calculation that
\ba
\left\|
\begin{pmatrix}
                   1 & 1 & 1 \\
                   0 & 1 & 0 \\
                   0 & 0 & 0 \\
\end{pmatrix}
\right\|_1=2,
\quad
\left\|\begin{pmatrix}
              1 \\
               0 \\
               1 \\
             \end{pmatrix}
\right\|_1=2.
\ea
Using \eqref{e:est-norm}, we get then
 for each $k=0,\dots, N$ and $C=(2C_4)\vee 1\geq  C_1,$
\ba
\sup_{t\in [t_k, t_{k+1)}}\E^{(N)}_{t_1,\dots,t_N} \|\Delta_t^n\|^q
&\leq
C_1\E^{(N)}_{t_1,\dots,t_N}
\left\|\begin{pmatrix}
    L_{k}^n \\
    \Xi_{k}^n \\
    u_{k}^n \\
\end{pmatrix}\right\|_1+C_1h_n^{\frac{q}{2}\wedge1}
\\&\leq C_1(2C_4)^{k} \left\|\left(
  \begin{array}{c}
    L_{0}^n \\
    \Xi_{0}^n \\
    u_{0}^n \\
  \end{array}
\right)\right\|_1+2h_n^{\frac{q}{2}\wedge1}\sum_{j=0}^{k-1}(2C_4)^{j+1}+C_1h_n^{\frac{q}{2}\wedge1}
\\&\leq C^{k+1}h_n^{\frac{p-q}{\chi}\wedge q}(1+2\|x\|^p) +(2k+1)C^{k+1}h_n^{\frac{q}{2}\wedge1}.
\ea
Since
\ba
\Delta_{t_{N+1}-}=\Delta_{t_N}\hbox{ a.s.,}
\ea
we finally get
\ba
\sup_{t\in [0, T]}\E^{(N)}_{t_1,\dots,t_N} \|\Delta_t^n\|^q
&=\max_{k=0,\dots,N}\sup_{t\in [t_k, t_{k+1)}}\E^{(N)}_{t_1,\dots,t_N} \|\Delta_t^n\|^q\\
&\leq C^{N+1}h_n^{\frac{p-q}{\chi}\wedge q}(1+2\|x\|^p) +(2N+1) C^{N+1}h_n^{\frac{q}{2}\wedge 1 }.
\ea
\end{proof}
Without restriction \eqref{typical} on the times $(t_1,\dots,t_N)$, we still have the following weaker version of Lemma \ref{cA5}.
\begin{lem}
\label{cA7}
 For any $q\in(0,p)$,
there exists $C\in(0,\infty)$ such that
\ba
\sup_{t\in[0, T]}\E^{(N)}_{t_1,\dots,t_N}\Big[\|\Delta_t^n\|^q\Big]
\leq C^{N+1}h_n^{\frac{p-q}{\chi}\wedge q}(1+2\|x\|^{p})+ (2N+1)C^{N+1}
\ea
for $n$ large enough uniformly over
$N\in\bN_0$ and $N$-tuples $\{t_1,\dots, t_N\}\subset (0,T)$.
\end{lem}
\begin{proof} The proof repeats the proof of Lemma \ref{cA5}, with just two minor changes, which we explain now.
When \eqref{typical} fails, several time instants, say, $t_{k}, t_{k+1}$, may belong
to the same discretization interval
\ba
{}[\eta^n_{t_{k+1}}, \eta^n_{t_{k+1}} + h_n)= [\eta^n_{t_{k}}, \eta^n_{t_{k}} + h_n).
\ea
In this case \eqref{e:cor1} is not true. However, in this case we have
$X^{(n)}_{t_{k+1}-}=X^{(n)}_{t_{k}-}$ and \eqref{e:cor2} still holds.

The estimate \eqref{e:ueq} in this case fails to be true, and instead we have
\begin{equation}
\label{e:ueq1}
\begin{aligned}
u^n_{k+1}
&\lc (\|B_{t_{k+1}}-B_{t_{k}}\|+\|\widehat Z_{t_{k+1}-}-\widehat Z_{t_{k}-}\|)^q+\|J_k\|^q\\
&+h_n^{1\wedge \frac{q}2}(h_n+\|B_{t_{k+1}}-B_{t_k}\|+\|\widehat Z_{t_{k+1}-}-\widehat Z_{t_k}\|)^{(p-q)\wedge (q\chi)}+
h_n^{1\wedge \frac{q}2}\|J_k\|^{(p-q)\wedge (q\chi)}+u_k.
\end{aligned}
\end{equation}
Therefore, instead of \eqref{e:uest}, we arrive to the estimate
\begin{equation}
\label{e:uest1}
\E^{(N)}_{t_1,\dots,t_N}\Big[u_{k+1}\Big|\cF_{t_k-} \Big]\leq C_3(h_n^{\frac{q}{2}\wedge 1}+1+u_{k}).
\end{equation}
Repeating literally the rest of the proof of Lemma \ref{cA5} with \eqref{e:uest} replaced by a weaker estimate \eqref{e:uest1} for every $k$, we get the required statement.
\end{proof}

Finally, we prove Theorem \ref{t:thm_q}.

\begin{proof}
We recall that conditioned on the event $\{N_T=N\}$, the jump times of the compound Poisson process $Q$
have law of the uniform order statistics on the interval $[0,T]$. Hence by the formula of the total probability
we get
\ba
\label{CPF}
\sup_{t\in[0, T]} \E\|\Delta_t^n\|^q=\ex^{-\lambda T}\sum_{N=0}^\infty \frac{\lambda^N}{T^N }\idotsint_{0<t_1\dots<t_N<T}
\sup_{t\in[0, T]}\E^{(N)}_{t_1,\dots, t_N}\Big[\|\Delta_t^n\|^q\Big]\,\di t_1\dots \di t_N.
\ea
By Lemmas \ref{cA5} and \ref{cA7}, there is $C\in(0,\infty)$ such that for all $N\in\bN_0$, all $N$-tuples
$\{t_1,\dots,t_N\}\subset (0,T)$ and $n$ large enough
\ba
\label{e:CD}
\sup_{t\in[0, T]}\E^{(N)}_{t_1,\dots, t_N}\|\Delta_t^n\|^q
&\leq C^{N+1}h_n^{\frac{p-q}{\chi}\wedge q}(1+2\|x\|^{p})+ (2N+1)C^{N+1}h_n^{\frac{q}2\wedge 1},
\quad (t_1,\dots, t_N)\in C_N^n,
\\
\sup_{t\in[0, T]} \E^{(N)}_{t_1,\dots, t_N}\|\Delta_t^n\|^q&\leq C^{N+1}h_n^{\frac{p-q}{\chi}\wedge q}(1+2\|x\|^{p})
+ (2N+1)C^{N+1}, \quad (t_1,\dots, t_N)\in D_N^n.
\ea
On the other hand,
\ba
& \int_{t_1<\cdots<t_N}\,\di t_1\dots \di t_N=\frac{T^N}{ N!},\\
& \int_{C_N^n}\,\di t_1\dots \di t_N
= \int_{Nh_n}^T \di t_N\int_{(N-1)h_n}^{t_N-h_n}\, \di t_{N-1}\cdots\int_{h_n}^{t_{3}-h_n}\, \di t_2
\int_0^{t_{2}-h_n}\, \di t_1=\frac{(T-Nh_n)^N_+}{ N!}\leq \frac{T^N}{ N!},\\
&\int_{D_N^n}\di t_1\dots \di t_N=\frac{T^N}{N!}\Big(1-\Big(1-\frac{Nh_n}{T}\Big)^N_+\Big)
\leq \frac{T^N}{N!} \cdot  \frac{N^2h_n}{T}.
\ea
Hence,
\ba
\E\sup_{t\in[0, T]}\|\Delta_t^n\|^q
&\leq \left(\sum_{N=0}^\infty \frac{C^{N+1}T^N}{ N!} \right)h_n^{\frac{p-q}{\chi}\wedge q}(1+2\|x\|^{p})
+ \left(\sum_{N=0}^\infty (2N+1)\frac{C^{N+1}T^N}{ N!} \right)h_n^{\frac{q}2\wedge 1}\\
&+ \left(\sum_{N=0}^\infty \frac{N^2h_n}{T}\frac{C^{N+1}T^N}{ N!} \right)h_n^{\frac{p-q}{\chi}\wedge q}(1+2\|x\|^{p})
+  \left(\sum_{N=0}^\infty (2N+1)\frac{N^2h_n}{T}\frac{C^{N+1}T^N}{ N!}\right)\\
&
\lc h_n^{\frac{p-q}{\chi}\wedge q}(1+2\|x\|^{p})+h_n^{\frac{q}2\wedge 1}\\
& \lc h_n^{\bar \delta}(1+\|x\|^{p}).
\ea
\end{proof}

\section{Proof of Theorem \ref{t:thm_q_sup}\label{s:5}}

Proof of Theorem \ref{t:thm_q_sup} follows the line of the one of Theorem \ref{t:thm_q}, with one additional  step which improves the moment estimates from Lemma \ref{lA1Delta} to maximal moment estimates.
\begin{lem}
\label{lA2}
Let $r\in[2,\infty)$. Then, for any $\gamma\in [0,1]$
\ba
\label{e:r}
\E^{(N)}_{t_1,\dots,t_N}\Big[\sup_{t\in[t_k, t_{k+1})}\|\Delta_t^n\|^r\Big|\mathcal{F}_{t_k}\Big]
&\leq_C \|\Delta_{t_k}^n\|^r
+h_n^{r\gamma +\frac{1}{2}}(1+\|X_{t_k}^{(n)}\|^{r\gamma\chi+r}) +h_n^{r\gamma }(1 + \|X^n_{t_k}\|^{r\gamma\chi+r})\\
&+h_n  (U_{t_k}^n)^r
+ h_n^{r\gamma+\frac12} (U_{t_k}^n)^{r\gamma\chi}  + h_n^{\frac{1}{2}}.
\ea
for $n$ large enough uniformly over
$N\in\bN_0$, partitions $\{t_1,\dots, t_N\}$,  and $k=0, \dots, N$.
\end{lem}
\begin{proof}
We use the notation of Lemma \ref{lA1Delta}. We also denote
\ba
\Xi_{\gamma,r}^{n,k}=h_n^{r\gamma +1}(1+\|X_{t_k}^{(n)}\|^{r\gamma\chi+r}) +h_n^{r\gamma }(1 + \|X^n_{t_k}\|^{r\gamma\chi+r})
+h_n  (U_{t_k}^n)^r
+ h_n^{r\gamma+1} (U_{t_k}^n)^{r\gamma\chi}  + h_n.
\ea
By \eqref{Ito2}, we have
\ba
\label{e:Delta*}
\sup_{t\in[t_k, t_{k+1})}\|\Delta_t^n\|^r&\leq \|\Delta_{t_k}^n\|^r
+\int_{t_k}^{t_{k+1}}\big(\mathrm{D}_s^{n,k}\big)_+\, \di s + \sup_{t\in[t_k, t_{k+1})} |\mathrm{M}_t^{n,k}|.\\
\ea
The drift part has already been treated in the proof of Lemma \ref{lA1Delta}, where we obtained the estimate
\ba
\label{est_intG}
\int_{t_k}^{t_{k+1}}\E^{(N)}_{t_1,\dots,t_N}\Big[ \big(\mathrm{D}_s^{n,k}\big)_+ \Big|\mathcal{F}_{t_k}\Big]\, \di s
\leq_C\|\Delta_{t_k}^n\|^r&+\Xi_{\gamma,r}^{n,k}.
\ea
For the martingale part, by the Burkholder--Davis--Gundy inequality we have
\ba
\label{e:BDG0}
\E^{(N)}_{t_1,\dots,t_N} \Big[\sup_{t\in[t_k, t_{k+1})} |\mathrm{M}_t^{n,k}| \Big|\mathcal{F}_{t_k}\Big]
\lc \E^{(N)}_{t_1,\dots,t_N} \Big[[\mathrm{M}^{n,k}]_{t_{k+1}}^{\frac{1}{2}}\Big|\mathcal{F}_{t_k}\Big]
\lc \E^{(N)}_{t_1,\dots,t_N} \Big[[\mathrm{M}^{n,k}]_{t_{k+1}}\Big|\mathcal{F}_{t_k}\Big]^{\frac{1}{2}}.
\ea
We have
\ba
{}[\mathrm{M}^{n,k}]_{t_{k+1}}
&\leq_C \int_{t_k}^{t_{k+1}}  \|\Delta_t^n\|^{2r-2}\sum_{i=0}^4\|\Lambda_t^{b,i,n}\|^2 \,\di t\\&
+\int_{t_k}^{t_{k+1}}\int_{\|z\|\leq 1} \Big(\Big\|\Delta^n_{t-}+\sum_{i=0}^4\Lambda_{t-}^{c,i,n} (z)\Big\|^r-\|\Delta^n_{t-}\|^r\Big)^2{N}(\di z, \di t).
\ea
By the Taylor formula,
\ba
\Big(\Big\|\Delta^n_{t-}+\sum_{i=0}^4\Lambda_{t-}^{c,i,n} (z)\Big\|^r-\|\Delta^n_{t-}\|^r\Big)^2
&\lc \Big(\|\Delta_{t-}^n\|^{r-1}
+ \Big\|\sum_{i=0}^4\Lambda_{t-}^{c,i,n} (z)\Big\|^{r-1}\Big)^2 \Big\|\sum_{i=0}^4\Lambda_{t-}^{c,i,n} (z)\Big\|^2\\
&\lc  \|\Delta_{t-}^n\|^{2r-2} \sum_{i=0}^4\|\Lambda_{t-}^{c,i,n} (z)\|^2 + \sum_{i=0}^4\|\Lambda_{t-}^{c,i,n} (z)\|^{2r}.
\ea
Since  $b(\cdot), c(\cdot)$ are Lipschitz continuous, we have
\ba
\|\Lambda_{t}^{b,0,n} (z)\|\lc \|\Delta_{t}^n\|,\quad \|\Lambda_{t-}^{c,0,n} (z)\|\lc \|\Delta_{t-}^n\|\|z\|.
\ea
Then, applying the Young inequality, we get
\ba
\label{e:brakM}
\E^{(N)}_{t_1,\dots,t_N} \Big[ [\mathrm{M}^{n,k}]_{t_{k+1}}| \Big|\mathcal{F}_{t_k}\Big]
&\leq_C \int_{t_k}^{t_{k+1}} \E^{(N)}_{t_1,\dots,t_N}  \Big[\|\Delta_t^n\|^{2r}\Big|\mathcal{F}_{t_k}\Big]\\
&+ \int_{t_k}^{t_{k+1}} \E^{(N)}_{t_1,\dots,t_N}
\Big[\sum_{i=1}^{4}\|\Lambda_t^{b,i,n}\|^{2r} \Big|\mathcal{F}_{t_k}\Big]\, \di t\\
&+ \int_{t_k}^{t_{k+1}} \int_{\|z\|\leq 1}
\E^{(N)}_{t_1,\dots,t_N}
\Big[\sum_{i=1}^{4}\|\Lambda_t^{c,i,n}(z)\|^{2r} \Big|\mathcal{F}_{t_k}\Big]\, \nu(\di z)\,\di t.
\ea
We use the statement of Lemma \ref{lA1Delta} with $r'=2r>2$ to estimate the expectation of $\|\Delta_{t}^n\|^{2r}$
and repeat, with the same $r'$, the estimates \eqref{e:I*},
\eqref{e:*I} for the terms with $\Lambda_t^{b,i,n}, \Lambda_t^{c,i,n}(z)$.
This yields
\ba
\E^{(N)}_{t_1,\dots,t_N} \Big[ [\mathrm{M}^{n,k}]_{t_{k+1}}| \Big|\mathcal{F}_{t_k}\Big]\lc \|\Delta_{t_k}^n\|^{2r}+\Xi_{\gamma,2r}^{n,k}.
\ea
Therefore, by \eqref{e:Delta*}, \eqref{est_intG}, and \eqref{e:BDG0}, we get
\ba
\E^{(N)}_{t_1,\dots,t_N}\Big[\sup_{t\in[t_k, t_{k+1})}\|\Delta_t^n\|^r\Big|\mathcal{F}_{t_k}\Big]
&\leq_C \|\Delta_{t_k}^n\|^r+ \Xi_{\gamma,r}^{n,k}+\Big(\Xi_{\gamma,2r}^{n,k}\Big)^{\frac{1}{2}},
\ea
which yields the required estimate.
\end{proof}

The rest of the proof of Theorem \ref{t:thm_q_sup} repeats the one of
Theorem \ref{t:thm_q}, with the estimate from Lemma \ref{lA1Delta} replaced by the one from Lemma \ref{lA2}.
Namely, repeating literally the argument from the proof of Corollary \ref{cA3}, we get
for any $q\in (0,p)$
\ba
\E^{(N)}_{t_1,\dots,t_N}\Big[\sup_{t\in[t_k, t_{k+1})} \|\Delta_t^n\|^q\Big|\mathcal{F}_{t_k}\Big]
\leq_C \|\Delta_{t_k}^n\|^q
& + h_n^{(\frac{p-q}{\chi}\wedge q) + (\frac{1}{2}\wedge\frac{q}{4})}(1+\|X_{t_k}^{(n)}\|^p)
+ h_n^{\frac{p-q}{\chi}\wedge q}(1+\|X_{t_k}^n\|^{p})\\
&+ h_n^{\frac{1}{2}\wedge\frac{q}{4}}  (U_{t_k}^n)^q
+h_n^{(\frac{p-q}{\chi }\wedge q)   +(\frac{1}{2}\wedge\frac{q}{4})} (U_{t_k}^n)^{(p-q)\wedge (q\chi)}
+h_n^{\frac{1}{2}\wedge\frac{q}{4}}.
\ea
Repeating then the proof of Lemma \ref{lA4cor}, we get
\ba
\E^{(N)}_{t_1,\dots,t_N}\Big[\sup_{t\in[t_k, t_{k+1})}\|\Delta_t^n\|^q\Big|\mathcal{F}_{t_k-}\Big]
&\leq_C \|\Delta_{t_k-}^n\|^q\\
&+h_n^{\frac{p-q}{\chi}\wedge q}(1+\|X_{t_k-}^{(n)}\|^p)
+ h_n^{\frac{p-q}{\chi}\wedge q}(1+\|X_{t_k-}^n\|^{p})\\
&+ (U^n_{t_k-})^q
+h_n^{\frac{1}{2}\wedge\frac{q}{4}}(U_{t_k-}^n)^{(p-q)\wedge (q\chi)}+h_n^{\frac{1}{2}\wedge\frac{q}{4}}.
\ea
Then the same estimates as in Lemma \ref{cA5} and Lemma \ref{cA7} give
\ba
 \E^{(N)}_{t_1,\dots,t_N}\Big[\sup_{t\in[0, T]}\|\Delta_t^n\|^q\Big]
\leq (N+1)C^{N+1}h_n^{\frac{p-q}{\chi}\wedge q}(1+2\|x\|^p) +(N+1)(2N+1) C^{N+1}h_n^{\frac{1}{2}\wedge\frac{q}{4}}
\ea
if \eqref{typical} is satisfied, and
\ba
\label{e:jhj}
\E^{(N)}_{t_1,\dots,t_N}\Big[\sup_{t\in[0, T]}\|\Delta_t^n\|^q\Big]
\leq (N+1)C^{N+1}h_n^{\frac{p-q}{\chi}\wedge q}(1+2\|x\|^p) +(N+1)(2N+1) C^{N+1}
\ea
otherwise. Overall, we get the convergence rate $\delta=\frac{p-q}{\chi}\wedge\frac{1}{2}\wedge\frac{q}{4}$.
The extra multiplier $(N+1)$ in \eqref{e:jhj} comes from the estimate
\ba
\E^{(N)}_{t_1,\dots,t_N}\Big[\sup_{t\in[0, T]}\|\Delta_t^n\|^q\Big]\leq \sum_{k=0}^N \E^{(N)}_{t_1,\dots,t_N}\Big[\sup_{t\in[t_k, t_{k+1)}}\|\Delta_t^n\|^q\Big]
\leq (N+1)\max_{k=0, \dots, N} \E^{(N)}_{t_1,\dots,t_N}\Big[\sup_{t\in[t_k, t_{k+1)}}\|\Delta_t^n\|^q\Big],
\ea
and does not cause any changes in the final part of the proof, based on \eqref{CPF}.

\section{Proof of Theorem \ref{t:P0}}

Let
\ba
0<q<p\leq p_X<p+\kappa-1.
\ea
Under the conditions of the Theorem \ref{t:P0}, the statements of Theorems
\ref{t:Xpdiss}, \ref{t:Xnpdiss} and
\ref{t:thm_q} hold, in particular, for each $T\in[0,\infty)$
\ba
\sup_{t\in[0,T]}
\E \|\Delta_t^n\|^{p_X}\leq_C (1+\|x\|^{p_X})
\ea
and
\ba
\sup_{t\in[0,T]}
\E \|\Delta_t^n\|^q\leq_C h_n^{\bar \delta}(1+\|x\|^{p}).
\ea
For $\e\in (0,p+\kappa-1-p_X)$ choose
\ba
\gamma=\gamma_{q,\e}:=\frac{p+\kappa-1 - p_X -\e}{p+\kappa-1-q-\e}\in (0,1)
\ea
such that
\ba
q \gamma + (p+\kappa-1-\e)(1-\gamma)=p_X.
\ea
By Hölder's inequality, we get
\ba
\E \|\Delta_t^n\|^{p_X}=\E \Big[(\|\Delta_t^n\|^q)^\gamma (\|\Delta_t^n\|^{p+\kappa-1-\e})^{1-\gamma}\Big]
&\leq (\E \|\Delta_t^n\|^q)^\gamma (\E \|\Delta_t^n\|^{p+\kappa-1-\e} )^{1-\gamma}\\
&\leq_C h_n^{\bar \delta\gamma}(1+\|x\|^{p\gamma+ (p+\kappa-1-\e)(1-\gamma)})
\\& \leq_C h_n^{\bar \delta\gamma}(1+\|x\|^{p+\kappa-1}),
\ea
where in the last inequality we have used that
\ba
p\gamma+ (p+\kappa-1-\e)(1-\gamma)<p+\kappa-1.
\ea
Maximizing the exponent $\bar \delta\gamma$ over $q\in(0,p)$ and $\e$ we get
\ba
\sup_{q,\e}\Big(\frac{p-q}{\chi}\wedge \frac{q}{2}\wedge 1\Big)\cdot \frac{p+\kappa-1 - p_X -\e}{p+\kappa-1-q-\e}
&=\sup_q \Big(\frac{p-q}{\chi}\wedge \frac{q}{2}\wedge 1\Big)\cdot \frac{p+\kappa-1 - p_X}{p+\kappa-1-q}\\
&=\begin{cases}
\frac{p}{2+\chi}\frac{p+\kappa-1 - p_X}{p+\kappa-1-\frac{2p}{2+\chi}},\quad \frac{p}{2+\chi} \in (0,1],\\
\frac{p+\kappa-1 - p_X}{\kappa+\chi-1},\quad  \frac{p}{2+\chi} \in (1,\infty).
\end{cases}
\ea
The supremum is attained at $q^*=\frac{2p}{2+\chi}$ and $q^*=p-\chi$ respectively.

\section{Continuous case. Proof of Theorems \ref{t:XG} and \ref{t:XnG}\label{s:cont}}

The proof of Theorems \ref{t:XG} and \ref{t:XnG} is based on the estimates of Lyapunov functions.

For $\lambda\in (0,\infty)$, let $V(x)=\ex^{\frac{\lambda}{1+\kappa}\|x\|^{1+\kappa}}$, $V\in C^2(\bR^d,\bR)$.
We have
\ba
V_{x}(x)&=\lambda V(x) \|x\|^{\kappa-1} x,\\
V_{xx} (x) &=\lambda V(x)\Big[\lambda\|x\|^{2\kappa-2}xx^T
+ (\kappa-1)\|x\|^{\kappa-3}x x^T
+ \|x\|^{\kappa-1}\mathrm{Id}\Big].
\ea
The It\^o formula combined with the localization and the Fatou Lemma argument yields:
\ba
\E V(X_t)\leq  V(x) + \E \int_0^t \cL V(X_s)\,\di s,
\ea
where
\ba
\cL V(x) & = V(x)\Big[ \lambda \|x\|^{\kappa-1}  \langle A(x),x\rangle  + \lambda \|x\|^{\kappa-1}  \langle a(x),x\rangle
+\frac12 \operatorname{tr} \Big(V_{xx}(x) b(x)b(x)^T \Big)\Big].
\ea
Note that the Hesse matrix $V_{xx}$ is symmetric and positive semi-definite, so that
\ba
\operatorname{tr} \Big(V_{xx}(x) b(x)b(x)^T \Big)
&\leq \operatorname{tr} \Big(V_{xx}(x)\Big)\|b(x)b(x)^T\| \\
&\leq \operatorname{tr} \Big(V_{xx}(x)\Big)\|b(x)\|^2\\
&\leq \lambda  \|b \|^2  V(x)  \Big[\lambda\|x\|^{2\kappa}+(\kappa-1+d)\|x\|^{\kappa-1}\Big].
\ea
Consequently, for $\lambda\in (0,\Lambda)$ and some $C\in(0,\infty)$
\ba
\cL V(x)
& \leq \lambda V(x)\Big[ - C_\text{diss} \|x\|^{2\kappa} +C\|x\|^{\kappa-1} + C\|x\|^{\kappa+1} + C\|x\|^\kappa
+\frac{\lambda}{2}\|b\|^2 \|x\|^{2\kappa} + \frac{\kappa-1+d}{2} \|b\|^2 \|x\|^{\kappa-1 }
\Big]\\
&\leq  \lambda V(x)\Big[ -\Big( C_\text{diss}-\frac{\lambda}{2}\|b\|^2  \Big) \|x\|^{2\kappa}
+  C(\|x\|^{\kappa+1}+1)
\Big].
\ea
Since $-(C_\text{diss}-\lambda/(2\|b\|^2)<0$ and $2\kappa>\kappa+1$ for $\kappa\in (1,\infty)$ we obtain that $\cL V(\cdot)$ is bounded from above, i.e.,
\ba
\cL V(x) \leq_C 1
\ea
and thus \eqref{e:XmomG} follows.

To obtain the uniform estimate, we assume that $\lambda\in [0,\frac{\Lambda}{2})$, and apply the Doob maximal inequality
to the stochastic integral to get
\ba
\E \sup_{t\in[0,T]}\Big|\int_0^t \langle V_x(X_s), b (X_s)  \, \di B_s\rangle\Big|^2
&\leq_C \E \int_0^T \|V_x(X_s) b (X_s)\|^2  \, \di s\\
& \leq_C  \E \int_0^T  V(X_s)^2 \|X_s\|^{2\kappa}  \, \di s.
\ea
Let us fix
$\widetilde \lambda$ such that
$0<2\lambda<\widetilde\lambda<\Lambda$, and let $\widetilde V(x)=\ex^{\frac{\widetilde\lambda}{1+\kappa}}\|x\|^{1+\kappa}$.
Then, $V(x)\|x\|^{2\kappa}\leq_C \widetilde V(x)$, and applying
\eqref{e:XmomG} to $\widetilde V$ we get
\ba
\E \sup_{t\in[0,T]}\Big|\int_0^t \langle V_x(X_s), b (X_s)  \, \di B_s\rangle\Big|^2
& \leq_C  \E \int_0^T  \widetilde V(X_s) \, \di s<\infty.
\ea

To prove Theorem \ref{t:XnG},
we apply the It\^o formula to see that $X^n$ is an It\^o semimartigale with the following representation:
\ba
\label{e:XnB}
X^n_t= x + \int_{0}^t A(X^n_s)\,\di s
&+ \int_0^t \Phi_x(\delta^n_s, Y^n_s) a(\Xn_s)\,\di s\\
&+\frac12 \int_0^t  \operatorname{tr}(\Phi_{xx}(\delta^n_s, Y^n_s)bb^T(\Xn_s))\, \di s\\
&+ \int_0^t \Phi_x(\delta^n_s, Y^n_s) b(\Xn_s) \,\di B_s.
\ea
Applying the It\^o formula to $V(X^n)$ combined with the localization and the Fatou Lemma argument yields
\ba
V(X^n_t)
&=
\lambda \int_{0}^t V(X^n_s) \|X^n_s\|^{\kappa-1}\langle X^{n}_s,A(X^n_s)\rangle   \,\di s\\
&+ \lambda \int_{0}^t V(X^n_s) \|X^n_s\|^{\kappa-1}
\langle X^{n}_s,\Phi_x(\delta^n_s, Y^n_s) a(\Xn_s)\rangle   \,\di s\\
&+ \frac{\lambda}{2} \int_{0}^t V(X^n_s) \|X^n_s\|^{\kappa-1}
\langle X^{n}_s, \operatorname{tr}(\Phi_{xx}(\delta^n_s, Y^n_s)bb^T(\Xn_s))  \rangle   \,\di s\\
&+\frac12\int_0^t  \operatorname{tr}\Big( \Phi_x(\delta^n_s, Y^n_s)^T
V_{xx}(X^n_s) \Phi_x(\delta^n_s, Y^n_s) b(\Xn_s) b(\Xn_s)^T  \Big)\,\di s\\
&+  \int_0^t
\langle V_x(X^n_s),\Phi_x(\delta^n_s,Y^n_s) b (\Xn_s)  \, \di B_s\rangle.
\ea
We have the following estimates:
\ba
V(X^n_s) \|X^n_s\|^{\kappa-1}\langle X^{n}_s,A(X^n_s)\rangle
&\leq -C_\mathrm{diss}V(X^n_s) \|X_s^n\|^{2\kappa},\\
\Big|V(X^n_s) \|X^n_s\|^{\kappa-1} \langle X^{n}_s,\Phi_x(\delta^n_s, Y^n_s) a(\Xn_s)\rangle\Big|
&\leq_C V(X^n_s) \|X^n_s\|^{\kappa} \|\Phi_x(\delta^n_s, Y^n_s)\|\\
&\leq_C V(X^n_s) \|X^n_s\|^{\kappa},\\
\Big|V(X^n_s) \|X^n_s\|^{\kappa-1}\langle X^{n}_s, \operatorname{tr}(\Phi_{xx}(\delta^n_s, Y^n_s)bb^T(\Xn_s))  \rangle \Big|
&\leq_C V(X^n_s) \|X^n_s\|^{\kappa}\|\Phi_{xx}(\delta^n_s, Y^n_s)\| \\
&\leq_C  V(X^n_s) \|X^n_s\|^{\kappa}
\ea
and, with the help of \eqref{e:Phix01},
\ba
\Big|\operatorname{tr} \Big(  b(X^{(n)}_s)^T \Phi_x(\delta^n_s, Y^n_s)^T
&V_{xx}(X^n_s) \Phi_x(\delta^n_s, Y^n_s) b(X^{(n)}_s) \Big)\Big|\\
&\leq \|b(X^{(n)}_s)\|^2 \|\Phi_x(\delta^n_s,  Y^n_s)\|^2 \operatorname{tr}\Big(V_{xx}(X^n_s)\Big)\\
&\leq\lambda d \|b\|^2  (1+Ch_n)^2V(X^n_s) \Big(\|X^n_s\|^{2\kappa} +(\kappa-1+d)\|X^n_s\|^{\kappa-1}\Big).
\ea
For $\lambda\in (0,\Lambda)$, we choose $n_0\in\bN$ large enough such that
\ba
\sup_{n\geq n_0}\Big(-C_\mathrm{diss} +\frac{\lambda d}{2} \|b\|^2(1+Ch_n)^2  \Big)< 0.
\ea
Then the above estimates combined with the localization and the Fatou Lemma argument yield
\ba
\E V(X_t^n)\leq V(x)+ C t
\ea
with some constant $C\in(0,\infty)$.
The uniform estimate \eqref{e:momXnGsup} is obtained as above.

\section{Continuous case. Proof of Theorem \ref{t:XucpG}\label{s:cont2}}

The following Lemma is the adapted and simplified version of Lemmas \ref{lA1Delta} and \ref{lA2}.
\begin{lem}
\label{t:thmGr}
Let $c(\cdot)\equiv 0$ and let assumptions $\mathbf{H}^{\mathrm{Lip}_+}_A$, $\mathbf{H}_{A_x,A_{xx}}$, and
$\mathbf{H}_{a,b,c}^{\mathrm{Lip}_b}$ hold. Then for any $T\in[0,\infty)$ and  $r\in (0,\infty)$,
\ba
\E \sup_{t\in[0, T]}\|X_t^n-X_t\|^r \leq_C h_n^\frac{r}{2} (1+ \|x\|^{r(\chi+1)}),\quad x\in\mathbb R^d,
\ea
for $n$ large enough.
\end{lem}
\begin{proof}
For any $r\in[2,\infty)$, the formula \eqref{Ito} holds with all the terms containing $\Lambda^{c,i,n}$
being zero.
The estimate \eqref{estDelta} takes the form
\ba
\label{estDelta1}
\E \|\Delta_t^n\|^r
\leq_C \|\Delta_{t_k}^n\|^r+ \sum_{i=1}^4 \int_{0}^{t}\E \|\Lambda_t^{a,i,n}\|^r \,\di s
+\sum_{i=1}^5 \int_0^t \|\Lambda_t^{b,i,n}\|^r\,\di s.
\ea
The estimates \eqref{Lambda1} hold, too, namely, taking $\gamma=1$ we get
\ba
\sum_{i=1}^4 \|\Lambda_s^{a,i,n}\|^r+ \sum_{i=1}^b \|\Lambda_s^{b,i,n}\|^r
\leq_C h_n^{r}(1+\|X^{(n)}_{s}\|^{r\chi+r})+(U_{s}^n)^r+h_n^{r }(U_s^n)^{r\chi}.
\ea
Since $U^n$ does not contain a jump component,
\ba
\E (U^n_s)^r\leq_C  h_n^\frac{r}{2}.
\ea
It can be easily shown with the help of the Gronwall argument that
\ba
\sup_{t\in[0,T]}\E \|X^{(n)}_{s}\|^{r\chi+r}\leq_C  1+ \|x\|^{r\chi+r}.
\ea
Finally, since the special consideration of the time intervals $[t_k,t_k^{n,*}]$ is obsolete, we immediately get the
estimate
\ba
\label{e:nnn}
\sup_{t\in[0,T]}\E\|\Delta^n_t\|^r\leq_C  h_n^\frac{r}{2} (1+ \|x\|^{r\chi+r}).
\ea
To obtain the uniform estimate, we continue as in Lemma \ref{lA2} and apply
the Burkholder--Davies--Gundy inequality with $p=1$ to the stochastic integral. This yields the desired estimate.
\end{proof}

To prove Theorem \ref{t:XucpG}, we use the Hölder inequality and estimates from Theorems \ref{t:XG}, \ref{t:XnG},
and Lemma \ref{t:thmGr}.
Let $\lambda\in (0,\frac{\Lambda}{2})$ and $r\in (0,\infty)$, and let $\lambda'\in (\lambda,\frac{\Lambda}{2})$.
Then
\ba
\E\Big[\sup_{t\in[0,T]}\ex^{\frac{\lambda}{1+\kappa}\|X^n_t\|^{1+\kappa}}\|\Delta^n_t\|^r\Big]
&\leq \E\Big[\sup_{t\in[0,T]}\ex^{\frac{\lambda}{1+\kappa}\|X^n_t\|^{1+\kappa}}\cdot \sup_{t\in[0,T]}\|\Delta^n_t\|^r\Big]\\
&\leq
\Big(\E\sup_{t\in[0,T]}\ex^{\frac{\lambda'}{1+\kappa}\|X^n_t\|^{1+\kappa}}\Big)^{\frac{\lambda}{\lambda'}}\cdot
\Big(\E \sup_{t\in[0,T]}\|\Delta^n_t\|^\frac{r\lambda'}{\lambda'-\lambda }\Big)^{\frac{\lambda'-\lambda}{\lambda'}}\\
&\leq_C
\Big(\ex^{\frac{\lambda'}{1+\kappa}\|x\|^{1+\kappa}}\Big)^{\frac{\lambda}{\lambda'}}\cdot
\Big(h_n^\frac{r\lambda'}{2(\lambda'-\lambda)}
(1+ \|x\|^\frac{r\lambda'(1+\chi)}{\lambda'-\lambda })\Big)^{\frac{\lambda'-\lambda}{\lambda'}}\\
&\leq h_n^\frac{r}{2}(1+\|x\|^{r(1+\chi)})\ex^{\frac{\lambda}{1+\kappa}\|x\|^{1+\kappa}}.
\ea
The same estimate holds for the weight $\ex^{\frac{\lambda}{1+\kappa}\|X^n_t\|^{1+\kappa}}$.
The estimate \eqref{e:Gauss0} is obtained analogously with the help of \eqref{e:XmomG} and \eqref{e:momXnG}.

\section{Notation\label{s:nota}}
In this paper we use the following notation.

For real valued functions $f$, $g$, we write $f(x)\lc g(x)$ if there exists a constant $C\in (0,\infty)$ such that
\ba
f(x)\leq C g(x), \quad x\in \bR^d,
\ea
and we do not need this constant for a further reference. In case that functions $f$, $g$
depend on additional parameters $t,n,\omega$, etc, the above inequality should hold uniformly over these parameters.

By $\bR^{d\times d}$ we denote the space of square $d$-dimensional real valued matrices.
For a matrix $b\in \bR^{d\times d}$, $b^T$ denotes the
transposed matrix. The identity matrix is denoted by $\mathrm{Id}$.
Vectors (columns) in $\bR^d$ are denoted by $x=(x^1,\dots,x^d)^T$, and for $x,y\in \bR^d$
\ba
\langle x,y\rangle= x^Ty= x^1y^1 + \cdots + x^d y^d
\ea
is the usual  scalar product with the associated Euclidean norm
\ba
\|x\|=\sqrt{\langle x,x\rangle}= \sqrt{(x^1)^2 + \cdots + (x^d)^2}.
\ea
For vector-valued functions $A,\Phi\colon \bR^d\to\bR^d$ etc, we write
\ba
A&=\begin{pmatrix}
A^1\\
\vdots\\
A^d\\
\end{pmatrix},\quad
A_x =
\begin{pmatrix}
             A^1_{x_1}& \cdots & A^1_{x_d}\\
            \vdots  & \ddots  &    \vdots  \\
             A^d_{x_1} & \cdots & A^d_{x_d}\\
           \end{pmatrix}
 =\begin{pmatrix}
A_{x_1},\ldots,A_{x_d}
\end{pmatrix},\\
A^i_x&=(A^i_{x_1},\ldots, A^i_{x_d}),\quad i=1,\dots,d.
\ea
By $A_{xx}$ we denote the second-order gradient tensor $(A^i_{x^jx^k})$ with the components
\ba
A_{x^ix^j}=\begin{pmatrix}
A^1_{x^ix^j}\\
\vdots\\
A^d_{x^ix^j}\\
\end{pmatrix},
\quad
A^i_{xx}:=
\begin{pmatrix}
            A^i_{x^1x^1} & \cdots & A^i_{x^1x^d}\\
            \vdots&\ddots&\vdots\\
            A^i_{x^1x^d} & \cdots & A^i_{x^dx^d}\\
           \end{pmatrix}.
\ea
Let $b\colon \bR^d\to\bR^{d\times d}$ be a matrix valued mapping.
We denote
\ba
\operatorname{tr}(A_{xx} bb^T )
=
\begin{pmatrix}
             \operatorname{tr}(A^1_{xx} bb^T )\\
            \vdots\\
             \operatorname{tr}(A^d_{xx} bb^T )\\
           \end{pmatrix}.
\ea
For a matrix-valued function $b$, we use the Euclidean matrix norm
\ba
\|b(x)\|&=\sqrt{\sum_{i,j=1}^d |b^i_j(x)|^2}=\sqrt{\operatorname{tr}(b(x)b^T(x))}=\sqrt{\operatorname{tr}(b^T(x)b(x))}.
\ea
For the tensors $A_{xx}$, $\Phi_{xx}$ etc, we use the norm
\ba
\|A_{xx}(x)\|&=\sqrt{\sum_{i,j,k=1}^d |A^k_{x^ix^j}(x)|^2} .%=\sqrt{\operatorname{tr}(bb^T(x))}=\sqrt{\operatorname{tr}(b^Tb(x))}
\ea
With some abuse of notation, we also denote
\ba
\|b\|&:=\sup_{x\in\bR^d}\|b(x)\|,\\
\|A_{xx}\|&:=\sup_{x\in\bR^d}\|A_{xx}(x)\|.
\ea
We use the following estimates (recall that $\Phi^i_{xx}$ and $bb^T$ are symmetric and positive semi-definite):
\ba
\operatorname{tr}(\Phi^i_{xx}(x)b(x)b^T(x))\leq  \operatorname{tr}(\Phi^i_{xx}(x)) \operatorname{tr}(b(x)b^T(x))
\leq \operatorname{tr}(\Phi^i_{xx}(x)) \|b(x)\|^2 \leq \sqrt{d}\|\Phi^i_{xx}(x)\| \|b(x)\|^2.
\ea

For $a,b\in\bR$, we denote $a\wedge b:=\min\{a,b\}$, $a_+:=\max\{a, 0\}$, and $a_-:=\max\{-a, 0\}$.

We often use the following elementary inequalities: for any $\gamma\in [0,1]$,
\ba
\label{e:max}
a\wedge b\leq a^\gamma b^{1-\gamma},\quad a,b\in(0,\infty),
\ea
and for any fixed $r\in(0,\infty)$, $m\in \bN$,
\ba
\Big(\sum_{k=1}^m a_k\Big)^r\leq_C \sum_{k=1}^m a_k^r, \quad a_1, \dots, a_m\in[ 0,\infty).
\ea
We also use two Young inequalities: for $a,b\in[ 0,\infty)$ we have
\ba
\label{e:young}
ab^{r-1}\leq \frac{1}{r}a^r+ \frac{r-1}{r}b^r, \quad r\in[1,\infty),\\
a^2b^{r-2}\leq \frac{2}{r}a^r+ \frac{r-2}{r}b^r, \quad r\in[2,\infty).
\ea
For $\phi\in C^2(\bR^d,\bR)$, the following Taylor formulae hold:
\begin{align}
\label{e:taylor1}
\phi(x+y)-\phi(x)
&= \int_0^1 \langle y,  \phi_x(x+ \theta y) \rangle \,\di \theta,\\
\label{e:taylor2}
\phi(x+y)-\phi(x) - \langle \phi_x(x),y\rangle
&=\int_0^1  \langle y, \phi_{xx}(x+ \theta y) y\rangle     (1-\theta)\,\di \theta.
\end{align}

Finally, the norm mapping $x\mapsto \|x\|^r$, $r\in[2,\infty)$, has the following derivatives:
\ba
\label{e:dnorm}
\partial_{x}\|x\|^r&= r \|x\|^{r-2}x^T,\\
\partial_{xx}\|x\|^r&= r(r-2) \|x\|^{r-4}x x^T + r \|x\|^{r-2}\mathrm{Id}.
\ea

\section*{Acknowledgments}
O.A.\ acknowledges funding from the DFG project AR 1717/2-1 (548113512).
The text of this paper was partially reviewed for spelling and grammar using
ChatGPT.

\end{document}